\documentclass[12pt,draftcls,onecolumn]{IEEEtran}
\input epsf

\usepackage{amsmath,amssymb,pstricks,color}
\usepackage{psfrag,graphicx,epsfig}

\newcommand{\nnum}{\nonumber}

\newcommand{\EQ}{\begin{eqnarray}}
\newcommand{\EN}{\end{eqnarray}}
\newcommand{\EQQ}{\begin{eqnarray*}}
\newcommand{\ENN}{\end{eqnarray*}}

\newcommand{\bremark}{\begin{remark} \begin{rm} }
\newcommand{\eremark}{ \end{rm} \rule{1mm}{2mm}
\end{remark} }
\newcommand{\btheorem}{\begin{theorem} \begin{rm} }
\newcommand{\etheorem}{ \end{rm} \rule{1mm}{2mm}
\end{theorem} }
\newcommand{\blemma}{\begin{lemma} \begin{rm} }
\newcommand{\elemma}{ \end{rm} \rule{1mm}{2mm}
\end{lemma} }
\newcommand{\bcorollary}{\begin{corollary} \begin{rm} }
\newcommand{\ecorollary}{ \end{rm} \rule{1mm}{2mm}
\end{corollary} }
\newcommand{\bdefinition}{\begin{definition}\begin{rm} }
\newcommand{\edefinition}{ \end{rm} \rule{1mm}{2mm}
\end{definition} }
\newcommand{\bproposition}{\begin{proposition} \begin{rm} }
\newcommand{\eproposition}{ \end{rm} \rule{1mm}{2mm}
\end{proposition} }
\newcommand{\bexample}{\begin{example} \begin{rm} }
\newcommand{\eexample}{ \end{rm} \rule{1mm}{2mm}
\end{example} }
\newcommand{\basm}{\begin{assumption} \begin{rm}}
\newcommand{\easm}{\end{rm} 
\end{assumption}}

\newcommand{\real}{\mathbb{R}}

\newcommand{\DD}{\mathcal{D}}

\newcommand{\GG}{\mathcal{G}}
\newcommand{\HH}{\mathcal{H}}

\newcommand{\LL}{\mathcal{L}}

\newcommand{\miset}{M^{[i]}(\tilde{\mu})}
\newcommand{\mset}{M(\tilde{\mu})}

\newcommand{\co}{\operatorname{co}}
\newcommand{\aff}{\operatorname{aff}}

\newcommand{\until}[1]{\{1,\dots, #1\}}

\newtheorem{theorem}{\bf Theorem}[section]
\newtheorem{lemma}{\bf Lemma}[section]
\newtheorem{definition}{\bf Definition}[section]
\newtheorem{remark}{\bf Remark}[section]
\newtheorem{corollary}{\bf Corollary}[section]
\newtheorem{proposition}{\bf Proposition}[section]
\newtheorem{example}{\bf Example}[section]
\newtheorem{assumption}{\bf Assumption}[section]

\newcommand\oprocendsymbol{\hbox{$\bullet$}}
\newcommand\oprocend{\relax\ifmmode\else\unskip\hfill\fi\oprocendsymbol}

\newcommand\emptyboxsymbol{\hbox{$\Box$}}
\newcommand\emptybox{\relax\ifmmode\else\unskip\hfill\fi\emptyboxsymbol}

\date{}

\begin{document}

\title{On distributed convex optimization under inequality and
  equality constraints via primal-dual subgradient methods}

\author{ Minghui Zhu and Sonia Mart{\'\i}nez \thanks{The authors are
   with Department of Mechanical and Aerospace Engineering,
   University of California, San Diego, 9500 Gilman Dr, La Jolla CA,
   92093, {\tt\small \{mizhu,soniamd\}@ucsd.edu}}}

\maketitle

\begin{abstract}
  We consider a general multi-agent convex optimization problem where
  the agents are to collectively minimize a global objective function
  subject to a global inequality constraint, a global equality
  constraint, and a global constraint set. The objective function is defined by a sum of local objective functions, while the global constraint set is
  produced by the intersection of local constraint sets. In
  particular, we study two cases: one where the equality constraint is
  absent, and the other where the local constraint sets are identical.
  We devise two distributed primal-dual subgradient algorithms which
  are based on the characterization of the primal-dual optimal
  solutions as the saddle points of the Lagrangian and penalty
  functions. These algorithms can be implemented over networks with
  changing topologies but satisfying a standard connectivity property,
  and allow the agents to asymptotically agree on optimal solutions
  and optimal values of the optimization problem under the Slater's
  condition.
\end{abstract}

\section{Introduction}\label{sec:introduction}

Recent advances in sensing, communication and computation technologies
are challenging the way in which control mechanisms are designed for
their efficient exploitation in a coordinated manner.  This has
motivated a wealth of algorithms for information processing,
cooperative control, and optimization of large-scale networked
multi-agent systems performing a variety of tasks. Due to a lack of a
centralized authority, the proposed algorithms aim to be executed by
individual agents through local actions, with the main feature of
being robust to dynamic changes of network topologies.

In this paper, we consider a general multi-agent optimization problem
where the goal is to minimize a global objective function, given as a
sum of local objective functions, subject to global constraints, which
include an inequality constraint, an equality constraint and a (state)
constraint set. Each local objective function is convex and only known
to one particular agent. On the other hand, the inequality
(resp. equality) constraint is given by a convex (resp. affine)
function and known by all agents. Each node has its own convex
constraint set, and the global constraint set is defined as their
intersection. This problem is motivated by others in distributed
estimation~\cite{RDN:03}~\cite{SSR-AN-VVV:08b}, distributed source
localization~\cite{MGR-RDN:04}, network utility maximization~\cite{KPK-AM-DT:98}, optimal flow control in power systems~\cite{AO-SS-LN:03,HW-HS-JK-PY:98} and optimal shape changes of mobile robots~\cite{JD-JS:07}. An important feature of the problem is that the objective and (or) constraint functions depend upon a global decision vector. This requires the design of distributed algorithms where, on the one hand, agents can align their
decisions through a local information exchange and, on the other hand,
the common decisions will coincide with an optimal solution and the
optimal value.

\emph{Literature Review.} In~\cite{DPB-JNT:97} and~\cite{JNT:84}, the
authors develop a general framework for parallel and distributed
computation over a set of processors. Consensus problems, a class of
canonical problems on networked multi-agent systems, have been
intensively studied since then. A necessarily incomplete list of references
includes~\cite{JAF-RMM:04,ROS-RMM:03c} tackling continuous-time
consensus,~\cite{VDB-JMH-AO-JNT:05,AJ-JL-ASM:02,LM:05} investigating discrete-time versions, and~\cite{MM-DS-JP-SHL-RMM:07a} where asynchronous implementation of consensus algorithms is discussed. The papers~\cite{SB-AG-BP-DS:06,AK-TB-RS:07,ATS-AJ:08} treat randomized consensus via gossip communication, achieving
consensus through quantized information and consensus over random graphs, respectively. The convergence rate of consensus algorithms is discussed, e.g., in~\cite{AO-JNT:07,LX-SB:04}, and the author in~\cite{Cortes:06} derives conditions to achieve different consensus values.

In robotics and control communities, convex optimization has been exploited to design algorithms coordinating mobile multi-agent systems. In~\cite{MCD-AJ:06}, in order to increase the connectivity of a multi-agent system, a distributed supergradient-based algorithm is proposed to maximize the second smallest eigenvalue of the Laplacian matrix of the state dependent proximity graph of agents. In~\cite{JD-JS:07}, optimal shape changes of mobile robots are achieved through
second-order cone programming techniques. In~\cite{JD-JS-MAH:09}, a target tracking problem is addressed by means of a generic semidefinite program where the constraints of network connectivity and full target coverage are articulated as linear-matrix inequalities. In~\cite{AIM-SIR:05}, in order to attain the highest possible positioning accuracy for mobile robots, the authors express the covariance matrix of the pose errors as a functional relation of measurement frequencies, and then formulate a optimal sensing problem as a convex programming of measurement frequencies.


The recent papers~\cite{AN-AO:09,AN-AO-PAP:08} are the most relevant
to our work. In~\cite{AN-AO:09}, the authors solve a multi-agent
unconstrained convex optimization problem through a novel combination
of average consensus algorithms with subgradient methods. More
recently, the paper~\cite{AN-AO-PAP:08} further takes local constraint
sets into account. To deal with these constraints, the authors
in~\cite{AN-AO-PAP:08} present an extension of their distributed
subgradient algorithm, by projecting the original algorithm onto the
local constraint sets. Two cases are solved in~\cite{AN-AO-PAP:08}:
the first assumes that the network topologies can dynamically change
and satisfy a periodic strong connectivity assumption (i.e., the union
of the network topologies over a bounded period of time is strongly
connected), but then the local constraint sets are identical; the
second requires that the communication graphs are (fixed and) complete
and then the local constraint sets can be different. Another related
paper is~\cite{BJ-TK-MJ-KHJ:08} where a special case
of~\cite{AN-AO-PAP:08}, the network topology is fixed and all the
local constraint sets are identical, is addressed.

\emph{Statement of Contributions.} Building on the
work~\cite{AN-AO-PAP:08}, this paper further incorporates global
inequality and equality constraints. More precisely, we study two
cases: one in which the equality constraint is absent, and the other in
which the local constraint sets are identical. For the first case, we
adopt a Lagrangian relaxation approach, define a Lagrangian dual
problem and devise the distributed Lagrangian primal-dual subgradient
algorithm (DLPDS, for short) based on the characterization of the
primal-dual optimal solutions as the saddle points of the Lagrangian
function. The DLPDS algorithm involves each agent updating its
estimates of the saddle points via a combination of an average
consensus step, a subgradient (or supgradient) step and a primal (or
dual) projection step onto its local constraint set (or a compact set
containing the dual optimal set). The DLPDS algorithm is shown to
asymptotically converge to a pair of primal-dual optimal solutions
under the Slater's condition and the periodic strong connectivity
assumption. Furthermore, each agent asymptotically agrees on the
optimal value by implementing a dynamic average consensus algorithm
developed in~\cite{MZ-SM:08a}, which allows a multi-agent system to
track time-varying average values.

For the second case, to dispense with the additional equality
constraint, we adopt a penalty relaxation approach, while defining a
penalty dual problem and devising the distributed penalty primal-dual
subgradient algorithm (DPPDS, for short). Unlike the first case, the
dual optimal set of the second case may not be bounded, and thus the
dual projection steps are not involved in the DPPDS algorithm. It
renders that dual estimates and thus (primal) subgradients may not be
uniformly bounded. This challenge is addressed by a more careful
choice of step-sizes. We show that the DPPDS algorithm asymptotically
converges to a primal optimal solution and the optimal value under the
Slater's condition and the periodic strong connectivity assumption.

For the special case where the global inequality and equality
constraints are not taken into account, this paper extends the results
in~\cite{AN-AO-PAP:08} to a more general scenario where the network
topologies satisfy the periodic strong connectivity assumption, and
the local constraint sets can be different, while relaxing an
interior-point condition requirement. We refer the readers to
Section~\ref{specialcase} for additional information.


\section{Problem formulation and assumptions}\label{sec:formulation}

\subsection{Problem formulation}

Consider a network of agents labeled by $V := \until{N}$ that can only
interact with each other through local communication.  The objective
of the multi-agent group is to cooperatively solve the following optimization problem:
\begin{align}
  \min_{x\in {\real}^n } f(x) := \sum_{i=1}^N f^{[i]}(x),\quad {\rm s.t.} \quad
  g(x)\leq 0,\quad h(x) = 0,\quad x \in X:= \cap_{i=1}^N X^{[i]},\label{e2}
\end{align}
where $f^{[i]} : {\real}^n \rightarrow {\real}$ is the convex objective
function of agent~$i$, $X^{[i]}\subseteq {\real}^n$ is the compact and
convex constraint set of agent~$i$, and $x$ is a global decision
vector. Assume that $f^{[i]}$ and $X^{[i]}$ are only known by agent~$i$, and
probably different. The function $g : {\real}^n \rightarrow {\real}^m$
is known to all the agents with each component $g_{\ell}$, for $\ell
\in \until{m}$, being convex. The inequality $g(x)\leq0$ is understood
component-wise; i.e., $g_{\ell}(x)\leq0$, for all $\ell \in \until{m}$,
and represents a global inequality constraint. The function $h:
\real^n \rightarrow \real^\nu$, defined as $h(x) := A x - b$ with $A
\in {\real}^{\nu\times n}$, represents a global equality constraint,
and is known to all the agents. We denote $Y := \{x\in{\real}^n \; | \;
g(x)\leq0,\quad h(x) = 0\}$, and assume that the set of feasible points
is non-empty; i.e., $X \cap Y \neq \emptyset$. Since $X$ is compact
and $Y$ is closed, then we can deduce that $X\cap Y$ is compact. The
convexity of $f^{[i]}$ implies that of $f$ and thus $f$ is continuous. In
this way, the optimal value $p^*$ of the problem~\eqref{e2} is finite
and $X^*$, the set of primal optimal points, is non-empty. Throughout
this paper, we suppose the following Slater's condition holds:

\begin{assumption}[Slater's Condition] There exists a vector $\bar{x}\in X$ such that $g(\bar{x}) < 0$ and $h(\bar{x}) = 0$. And there exists a relative interior point $\tilde{x}$ of $X$, i.e., $\tilde{x}\in X$ and there exists an open sphere $S$ centered at $\tilde{x}$ such that $S\cap{\aff}(X)\subset X$ with ${\aff}(X)$ being the affine hull of $X$, such that $h(\tilde{x}) = 0$.\label{asm5}
\end{assumption}

\begin{remark} In this paper, the quantities (e.g., functions, scalars and sets) associated with agent $i$ will be indexed by the superscript $[i]$.\label{rem9}\end{remark}

In this paper, we will study two particular cases of
problem~\eqref{e2}: one in which the global equality constraint $h(x)
= 0$ is not included, and the other in which all the local constraint
sets are identical. For the case where the constraint $h(x) = 0$ is
absent, the Slater's condition~\ref{asm5} reduces to the existence of
a vector $\bar{x}\in X$ such that $g(\bar{x}) < 0$.


\subsection{Network model}

We will consider that the multi-agent network operates
synchronously. The topology of the network at time $k \geq 0$ will be
represented by a directed weighted graph ${\GG}(k) = (V,E(k),A(k))$
where $A(k) := [a^i_j(k)] \in \real^{N\times N}$ is the adjacency
matrix with $a^i_j(k)\geq0$ being the weight assigned to the edge
$(j,i)$ and $E(k)\subset V\times V\setminus {\rm diag}(V)$ is the set
of edges with non-zero weights $a^i_j(k)$. The in-neighbors of node
$i$ at time $k$ are denoted by ${\mathcal{N}}^{[i]}(k) = \{j\in V \; | \;
(j,i)\in E(k) \text{ and } j\neq i\}$. We here make the following
assumptions on the network communication graphs, which are standard in
the analysis of average consensus algorithms; e.g.,
see~\cite{ROS-RMM:03c},~\cite{AO-JNT:07}, and distributed optimization
in~\cite{AN-AO:09},~\cite{AN-AO-PAP:08}.

\begin{assumption} [Non-degeneracy]
  There exists a constant $\alpha>0$ such that $a_i^i(k)\geq\alpha$,
  and $a^i_j(k)$, for $i\neq j$, satisfies $a^i_j(k)\in \{0\}\cup
  [\alpha,\;1],\;$ for all $k\geq0$. \label{asm2}
\end{assumption}

\begin{assumption}[Balanced Communication]\footnote{It is also referred
    to as double stochasticity.} It holds that $\sum_{j=1}^N
  a_j^i(k)=1$ for all $i\in V$ and $k\geq0$, and $\sum_{i=1}^N
  a_j^i(k)=1$ for all $j\in V$ and $k\geq0$.\label{asm3}
\end{assumption}

\begin{assumption} [Periodical Strong Connectivity]
  There is a positive integer $B$ such that, for all $k_0\geq0$,
  the directed graph $(V,\bigcup_{k=0}^{B-1}E(k_0 + k))$ is strongly
  connected. \label{asm1}
\end{assumption}

\subsection{Notion and notations}

The following notion of saddle point plays a critical role in our paper.
\begin{definition}[Saddle point]
  Consider a function $\phi : X\times M \rightarrow {\real}$ where $X$
  and $M$ are non-empty subsets of ${\real}^{\bar{n}}$ and ${\real}^{\bar{m}}$. A pair
  of vectors $(x^*,\mu^*)\in X\times M$ is called a saddle point of
  $\phi$ over $X\times M$ if $\phi(x^*, \mu)\leq \phi(x^*,
    \mu^*) \leq \phi(x, \mu^*)$ hold for all
  $(x,\mu)\in X\times M$.\label{def1}
\end{definition}

\begin{remark} Equivalently, $(x^*,\mu^*)$ is a saddle point of
  $\phi$ over $X\times M$ if and only if $(x^*,\mu^*)\in X\times M$, and $\sup_{\mu\in M}\phi(x^*, \mu)\leq \phi(x^*,\mu^*) \leq \inf_{x\in X}\phi(x, \mu^*)$.\oprocend\label{rem3}
\end{remark}

In this paper, we do not assume the differentiability of $f^{[i]}$ and
$g_{\ell}$. At the points where the function is not differentiable,
the subgradient plays the role of the gradient. For a given convex
function $F : {\real}^{\bar{n}} \rightarrow {\real}$ and a point
$\tilde{x}\in{\real}^{\bar{n}}$, a \emph{subgradient} of the function
$F$ at $\tilde{x}$ is a vector ${\DD}F(\tilde{x})\in{\real}^{\bar{n}}$
such that the following subgradient inequality holds for any $x \in \real^{\bar{n}}$: \begin{align*}{\DD}F(\tilde{x})^T(x-\tilde{x})\leq
F(x)-F(\tilde{x}).\end{align*}

Similarly, for a given concave function $G : {\real}^{\bar{m}} \rightarrow {\real}$
and a point $\bar{\mu}\in{\real}^{\bar{m}}$, a \emph{supgradient} of
the function $G$ at $\bar{\mu}$ is a vector
${\DD}G(\bar{\mu})\in{\real}^{\bar{m}}$ such that the following supgradient inequality holds for any $\mu \in \real^{\bar{m}}$:
\begin{align*}{\DD}G(\bar{\mu})^T(\mu-\bar{\mu})\geq G(\mu)-G(\bar{\mu}).\end{align*}

  Given a set $S$, we denote by $\co(S)$ its convex hull. We let the
  function $[\cdot]^+ :
  {\real}^{\bar{m}}\rightarrow{\real}^{\bar{m}}_{\geq0}$ denote the
  projection operator onto the non-negative orthant in
  ${\real}^{\bar{m}}$. For any vector $c\in{\real}^{\bar{n}}$, we
  denote $|c| := (|c_1|,\cdots,|c_{\bar{n}}|)^T$, while $\|\cdot\|$ is
  the 2-norm in the Euclidean space.

\section{Case (i): absence of equality constraint}
\label{sec:equality}

In this section, we study the case of problem~\eqref{e2} where the equality
constraint $h(x) = 0$ is absent; i.e., problem~\eqref{e2} becomes \begin{align} \min_{x\in {\real}^n } \sum_{i=1}^N
  f^{[i]}(x),\quad {\rm s.t.} \quad g(x)\leq 0,\quad x\in \cap_{i=1}^N
  X^{[i]}.\label{e1}
\end{align} We first provide some preliminaries, including a
Lagrangian saddle-point characterization of problem~\eqref{e1} and
finding a superset containing the Lagrangian dual optimal set of
problem~\eqref{e1}. After that, we present the distributed Lagrangian
primal-dual subgradient algorithm and summarize its convergence
properties.


\subsection{Preliminaries}

We here develop some preliminary results which are essential to the
design of the distributed Lagrangian primal-dual subgradient
algorithm.

\subsubsection{A Lagrangian saddle-point characterization}

Firstly, problem~\eqref{e1} is equivalent to
\begin{align}
  \min_{x\in {\real}^n } f(x),\quad {\rm s.t.} \quad N g(x)\leq
  0,\quad x\in X,\nnum
\end{align}
with associated Lagrangian dual problem given by
\begin{align}
  \max_{\mu\in{\real}^m} q_L(\mu),\quad {\rm s.t.}\quad \mu\geq0.\nnum
\end{align}
Here, the Lagrangian dual function, $q_L :
{\real}^m_{\geq0}\rightarrow {\real}$, is defined as $q_L(\mu) :=
\inf_{x\in X} {\LL}(x,\mu),$ where ${\LL} :
{\real}^n\times{\real}^m_{\geq0} \rightarrow {\real}$ is the
\emph{Lagrangian function} ${\LL}(x,\mu) = f(x) + N \mu^T g(x).$ We
denote the Lagrangian dual optimal value of the Lagrangian dual
problem by $d^*_L$ and the set of Lagrangian dual optimal points by $D^*_L$. As
is well-known, under the Slater's condition~\ref{asm5}, the property
of strong duality holds; i.e., $p^* = d^*_L$, and $D^*_L\neq
\emptyset$. The following theorem is a standard result on Lagrangian
duality stating that the primal and Lagrangian dual optimal solutions
can be characterized as the saddle points of the Lagrangian function.
\begin{theorem}[Lagrangian Saddle-point Theorem~\cite{DPB:09}]
  The pair of $(x^*, \mu^*) \in X\times {\real}^m_{\geq0}$ is a saddle
  point of the Lagrangian function $\LL$ over $X\times
  {\real}^m_{\geq0}$ if and only if it is a pair of primal and
  Lagrangian dual optimal solutions and the following \emph{Lagrangian
    minimax equality} holds:
  \begin{equation*}
    \sup_{\mu\in{\real}^m_{\geq0}}\inf_{x\in X}{\LL}(x,\mu) =
  \inf_{x\in X}\sup_{\mu\in{\real}^m_{\geq0}}{\LL}(x,\mu).
  \end{equation*}
  \label{the2}
\end{theorem}

This following lemma presents some preliminary analysis of Lagrangian saddle points.

\begin{lemma}[Preliminary results of Lagrangian saddle points] Let $M$ be any superset of $D^*_L$.

  (a) If $(x^*,\mu^*)$ is a saddle point of ${\LL}$ over $X\times
  {\real}^m_{\geq0}$, then $(x^*,\mu^*)$ is also a saddle point of
  ${\LL}$ over $X\times M$.

  (b) There is at least one saddle point of ${\LL}$ over $X\times M$.

  (c) If $(\check{x},\check{\mu})$ is a saddle point of ${\LL}$ over
  $X\times M$, then ${\LL}(\check{x},\check{\mu}) = p^*$ and
  $\check{\mu}$ is Lagrangian dual optimal.\label{lem9}
\end{lemma}


\begin{proof}
  (a) It just follows from the definition of saddle point of ${\LL}$
  over $X\times M$.

  (b) Observe that \begin{align*}&\sup_{\mu\in{\real}^m_{\geq0}}\inf_{x\in
    X}{\LL}(x,\mu)= \sup_{\mu\in{\real}^m_{\geq0}}q_L(\mu) =
  d^*_L,\\ &\inf_{x\in
    X}\sup_{\mu\in{\real}^m_{\geq0}}{\LL}(x,\mu)= \inf_{x\in X\cap
    Y}f(x) = p^*.\end{align*} Since the Slater's condition~\ref{asm5} implies
  zero duality gap, the Lagrangian minimax equality holds. From
  Theorem~\ref{the2} it follows that the set of saddle points of
  ${\LL}$ over $X\times {\real}^m_{\geq0}$ is the Cartesian product
  $X^*\times D^*_L$. Recall that $X^*$ and $D^*_L$ are non-empty, so
  we can guarantee the existence of the saddle point of ${\LL}$ over
  $X\times {\real}^m_{\geq0}$. Then by (a), we have that (b) holds.

  (c) Pick any saddle point $(x^*,\mu^*)$ of ${\LL}$ over $X\times
  {\real}^m_{\geq0}$. Since the Slater's condition~\ref{asm5} holds,
  from Theorem~\ref{the2} one can deduce that $(x^*,\mu^*)$ is a pair
  of primal and Lagrangian dual optimal solutions. This implies that
  \begin{align*}d^*_L = \inf_{x\in X}{\LL}(x,\mu^*) \leq {\LL}(x^*,\mu^*) \leq
  \sup_{\mu\in{\real}^m_{\geq0}}{\LL}(x^*,\mu) = p^*.\end{align*} From
  Theorem~\ref{the2}, we have $d^*_L = p^*$. Hence, ${\LL}(x^*,\mu^*)
  = p^*$. On the other hand, we pick any saddle point
  $(\check{x},\check{\mu})$ of ${\LL}$ over $X\times M$. Then for all
  $x\in X$ and $\mu\in M$, it holds that ${\LL}(\check{x}, \mu)\leq
  {\LL}(\check{x}, \check{\mu}) \leq {\LL}(x, \check{\mu})$. By
  Theorem~\ref{the2}, then $\mu^*\in D^*_L \subseteq M$. Recall
  $x^*\in X$, and thus we have ${\LL}(\check{x}, \mu^*)\leq {\LL}(\check{x},
    \check{\mu}) \leq {\LL}(x^*, \check{\mu})$. Since $\check{x}\in X$ and
  $\check{\mu}\in{\real}^m_{\geq0}$, we have ${\LL}(x^*, \check{\mu})\leq
    {\LL}(x^*, \mu^*) \leq {\LL}(\check{x},
    \mu^*)$.
  Combining the above two relations gives that ${\LL}(\check{x},
  \check{\mu}) = {\LL}(x^*, \mu^*) = p^*$. From Remark~\ref{rem3} we
  see that ${\LL}(\check{x},\check{\mu}) \leq \inf_{x\in
    X}{\LL}(x,\check{\mu}) = q_L(\check{\mu})$. Since
  ${\LL}(\check{x},\check{\mu}) = p^* = d^*_L \geq q_L(\check{\mu})$,
  then $q_L(\check{\mu}) = d^*_L$ and thus $\check{\mu}$ is a
  Lagrangian dual optimal solution.
\end{proof}

\begin{remark}
  Despite that (c) holds, the reverse of (a) may not be true in
  general. In particular, $x^*$ may be infeasible; i.e.,
  $g_{\ell}(x^*) > 0$ for some
  $\ell\in\until{m}$. \oprocend\label{rem5}
\end{remark}

\subsubsection{A upper estimate of the Lagrangian dual optimal set}

In what follows, we will find a compact superset of $D^*_L$. To
do that, we define the following primal problem for each agent~$i$:
\begin{align}
  \min_{x\in {\real}^n } f^{[i]}(x),\quad {\rm s.t.} \quad g(x)\leq
  0,\quad x\in X^{[i]}.\nnum
\end{align}
Due to the fact that $X^{[i]}$ is compact and the $f^{[i]}$ are continuous,
the primal optimal value $p_i^*$ of each agent's primal problem is
finite and the set of its primal optimal solutions is non-empty. The
associated dual problem is given by
\begin{align}
  \max_{\mu\in{\real}^m} q^{[i]}(\mu),\quad {\rm s.t.}\quad
  \mu\geq0.\nnum
\end{align}
Here, the dual function $q^{[i]} : {\real}^m_{\geq0}\rightarrow {\real}$
is defined by $q^{[i]}(\mu) := \inf_{x\in X^{[i]}} {\LL}^{[i]}(x,\mu),$
where ${\LL}^{[i]} : {\real}^n\times{\real}^m_{\geq0} \rightarrow
{\real}$ is the Lagrangian function of agent $i$ and given by ${\LL}^{[i]}(x,\mu) = f^{[i]}(x) + \mu^T g(x)$. The corresponding dual optimal value is
denoted by $d_i^*$. In this way, $\LL$ is decomposed into a sum of local Lagrangian functions; i.e., ${\LL}(x,\mu) = \sum_{i=1}^N {\LL}^{[i]}(x,\mu)$.

Define now the set-valued map $Q:{\real}_{\geq0}^m \rightarrow
2^{({\real}_{\geq0}^m)}$ by $Q(\tilde{\mu}) = \{\mu\in
{\real}^m_{\geq0} \; | \; q_L(\mu)\geq
q_L(\tilde{\mu})\}$. Additionally, define a function $\gamma : X
\rightarrow {\real}$ by $\gamma(x) = \min_{\ell \in
  \until{m}}\{-g_{\ell}(x)\}$. Observe that if $x$ is a Slater vector,
then $\gamma(x) > 0$. The following lemma is a direct result of
Lemma~1 in~\cite{AN-AO:08b}.

\begin{lemma}[Boundedness of dual solution sets]
  The set $Q(\tilde{\mu})$ is bounded for any $\tilde{\mu}\in
  {\real}^m_{\geq0}$, and, in particular, for any Slater vector
  $\bar{x}$, it holds that $\max_{\mu\in Q(\tilde{\mu})}\|\mu\| \leq
  \frac{1}{\gamma(\bar{x})}(f(\bar{x})-q_L(\tilde{\mu}))$.\label{lem13}\emptybox
\end{lemma}

Notice that $D^*_L = \{\mu\in{\real}_{\geq0}^m \;|\; q_L(\mu)\geq
d^*_L\}$. Picking any Slater vector $\bar{x}\in X$, and letting
$\tilde{\mu} = \mu^*\in D^*_L$ in Lemma~\ref{lem13} gives that
\begin{align}
  \max_{\mu^*\in D^*_L}\|\mu^*\| \leq
  \frac{1}{\gamma(\bar{x})}(f(\bar{x})-d^*_L).\label{e37}
\end{align}

Define the function $r : X\times{\real}^m_{\geq0}\rightarrow {\real}\cup \{ + \infty \}$ by $r(x,\mu) :=
\frac{N}{\gamma(x)}\max_{i\in V}\{f^{[i]}(x)-q^{[i]}(\mu)\}$.  By the property
of weak duality, it holds that $d_i^*\leq p_i^*$ and thus $f^{[i]}(x)\geq
q^{[i]}(\mu)$ for any $(x,\mu) \in X\times \real^m_{\geq 0}$. Since
$\gamma(\bar{x}) > 0$, thus $r(\bar{x},\mu)\geq0$ for any
$\mu\in{\real}^m_{\geq0}$. With this observation, we pick any
$\tilde{\mu}\in{\real}^m_{\geq0}$ and the following set is
well-defined: $\bar{M}^{[i]}(\bar{x}, \tilde{\mu}) := \{\mu\in{\real}_{\geq0}^m \; |\;
  \|\mu\|\leq r(\bar{x},\tilde{\mu})+\theta^{[i]}\}$ for some $\theta^{[i]}\in{\real}_{>0}$. Observe that for all $\mu\in{\real}^m_{\geq0}$:
\begin{align}
  q_L(\mu) = \inf_{x\in \cap_{i=1}^N X^{[i]}} \sum_{i=1}^N (f^{[i]}(x) + \mu^T g(x))
  \geq \sum_{i=1}^N\inf_{x\in X^{[i]}} (f^{[i]}(x) + \mu^T g(x)) =
  \sum_{i=1}^N q^{[i]}(\mu).
\label{e38}
\end{align}
Since $d^*_L \geq q_L(\tilde{\mu})$, it follows from~\eqref{e37}
and~\eqref{e38} that
\begin{align}\max_{\mu^*\in D^*_L}\|\mu^*\| & \leq
  \frac{1}{\gamma(\bar{x})}(f(\bar{x})-q_L(\tilde{\mu}))\leq
  \frac{1}{\gamma(\bar{x})}(f(\bar{x})-\sum_{i=1}^N
  q^{[i]}(\tilde{\mu}))\nnum\\&\leq \frac{N}{\gamma(\bar{x})}\max_{i\in
    V}\{f^{[i]}(\bar{x})-q^{[i]}(\tilde{\mu})\} = r(\bar{x},
  \tilde{\mu}).\nnum
\end{align}
Hence, we have $D^*_L\subseteq \bar{M}^{[i]}(\bar{x}, \tilde{\mu})$ for all
$i\in V$.

Note that in order to compute $\bar{M}^{[i]}(\bar{x}, \tilde{\mu})$, all
the agents have to agree on a common Slater vector $\bar{x}\in
\cap_{i=1}^N X^{[i]}$ which should be obtained in a distributed
fashion. To handle this difficulty, we now propose a distributed
algorithm, namely \emph{Distributed Slater-vector Computation
  Algorithm}, which allows each agent~$i$ to compute a superset of
$\bar{M}^{[i]}(\bar{x}, \tilde{\mu})$.

Initially, each agent~$i$ chooses a common value
$\tilde{\mu}\in{\real}^m_{\geq0}$; e.g.,~$\tilde{\mu}=0$, and computes
two positive constants $b^{[i]}(0)$ and $c^{[i]}(0)$ such that $b^{[i]}(0)\geq
\sup_{x\in J^{[i]}}\{f^{[i]}(x) - q^{[i]}(\tilde{\mu})\}$ and $c^{[i]}(0)\leq \min_{1\leq
  \ell\leq m}\inf_{x\in J^{[i]}}\{-g_{\ell}(x)\}$ where $J^{[i]} := \{x\in X^{[i]}
\; | \; g(x) < 0\}$.

At every time $k\geq0$, each agent~$i$ updates its
estimates by using the following rules:
\begin{align*}
  b^{[i]}(k+1) = \max_{j\in{\mathcal{N}}^{[i]}(k)\cup\{i\}} b^{[j]}(k),\quad
  c^{[i]}(k+1) = \min_{j\in{\mathcal{N}}^{[i]}(k)\cup\{i\}} c^{[j]}(k).
\end{align*}


\begin{lemma}[Convergence properties of the distributed Slater-vector Computation
  Algorithm] Assume that the periodical strong connectivity assumption~\ref{asm1} holds. Consider the sequences of $\{b^{[i]}(k)\}$ and $\{c^{[i]}(k)\}$ generated by the Distributed Slater-vector Computation Algorithm. It holds that after at most $(N-1)B$ steps, all the
agents reach the consensus, i.e., $b^{[i]}(k) = b^* := \max_{j\in V}
b^{[j]}(0)$ and $c^{[i]}(k) = c^* := \min_{j\in V} c^{[j]}(0)$ for all $k\geq
(N-1)B$. Furthermore, we have $\miset := \{\mu\in{\real}_{\geq0}^m \; |\;
\|\mu\|\leq \frac{N b^*}{c^*}+\theta^{[i]}\} \supseteq \bar{M}^{[i]}(\bar{x},
  \tilde{\mu})$ for $i\in V$. \label{lem10}
\end{lemma}



\begin{proof} It is not difficult to verify achieving max-consensus and min-consensus by using the periodical strong connectivity assumption~\ref{asm1}. Note that $J := \{x\in X \; | \;
g(x) < 0\} \subseteq J^{[i]}$, $\forall i\in V$. Hence, we have
\begin{align*}
  &\max_{i\in V}\sup_{x\in J}\{f^{[i]}(x)-q^{[i]}(\tilde{\mu})\} \leq
  \max_{i\in V}\sup_{x\in J^{[i]}}\{f^{[i]}(x)-q^{[i]}(\tilde{\mu})\} \leq b^*,\nnum\\
  &\inf_{x\in J}\min_{1\leq \ell\leq m}\{-g_{\ell}(x)\}\geq \min_{i\in
    V}\inf_{x\in J^{[i]}}\min_{1\leq \ell\leq m}\{-g_{\ell}(x)\}\geq c^*.
\end{align*}
Since $\bar{x}\in J$, then the following estimate on
$r(\bar{x},\tilde{\mu})$ holds:
\begin{align*}
  r(\bar{x},\tilde{\mu}) \leq \frac{N\sup_{x\in J}\max_{i\in
      V}\{f^{[i]}(x)-q^{[i]}(\tilde{\mu})\}}{\inf_{x\in J}\min_{1\leq \ell\leq
      m}\{-g_{\ell}(x)\}} \leq \frac{N b^*}{c^*}.
\end{align*}
The desired result immediately follows.
\end{proof}

From Lemma~\ref{lem10} and the fact that $D^*_L\subseteq
\bar{M}^{[i]}(\bar{x},\tilde{\mu})$, we can see
that the set of $\mset := \cap_{i=1}^N \miset$
contains $D^*_L$. In addition, $\miset$ and $M(\tilde{\mu})$
are non-empty, compact and convex. To simplify the notations, we
will use the shorthands $M^{[i]} := \miset$ and $M := \mset$.

\subsubsection{Convexity of $\LL$}

For each $\mu\in{\real}^m_{\geq0}$, we define the function
${\LL}^{[i]}_{\mu} : {\real}^n\rightarrow {\real}$ as
${\LL}^{[i]}_{\mu}(x) : = {\LL}^{[i]}(x,\mu)$. Note that
${\LL}^{[i]}_{\mu}$ is convex since it is a nonnegative weighted
sum of convex functions. For each $x\in{\real}^n$, we define the
function ${\LL}^{[i]}_x : {\real}^m_{\geq0}\rightarrow {\real}$
as ${\LL}^{[i]}_x(\mu) : = {\LL}^{[i]}(x,\mu)$. It is easy to
check that ${\LL}^{[i]}_x$ is a concave (actually affine)
function. Then the Lagrangian function ${\LL}$ is the sum of a
collection of convex-concave local functions. This property motivates us to
significantly extend primal-dual subgradient methods
in~\cite{KJA-LH-HU:58,AN-AO:08a} to the networked multi-agent scenario.

\subsection{Distributed Lagrangian primal-dual subgradient
  algorithm}\label{sec:algorithm}

Here, we introduce the \emph{Distributed Lagrangian
  Primal-Dual Subgradient} Algorithm (DLPDS, for short) to find a
saddle point of the Lagrangian function $\LL$ over $X\times M$ and the
optimal value. This saddle point will coincide with a pair of
primal and Lagrangian dual optimal solutions which is not always the case; see
Remark~\ref{rem5}.

Through the algorithm, at each time $k$, each agent~$i$ maintains the
estimate of $(x^{[i]}(k), \mu^{[i]}(k))$ to the saddle point of the Lagrangian
function $\LL$ over $X\times M$ and the estimate
of $y^{[i]}(k)$ to $p^*$. To produce $x^{[i]}(k+1)$ (resp. $\mu^{[i]}(k+1)$),
agent~$i$ takes a convex combination $v^{[i]}_x(k)$ (resp. $v^{[i]}_{\mu}(k)$) of its estimate $x^{[i]}(k)$
(resp.~$\mu^{[i]}(k)$) with the estimates sent from its neighboring agents
at time $k$, makes a subgradient (resp. supgradient) step to minimize
(resp. maximize) the local Lagrangian function ${\LL}^{[i]}$, and
takes a primal (resp. dual) projection onto the local constraint $X^{[i]}$
(resp.~$M^{[i]}$). Furthermore, agent~$i$ generates the
estimate $y^{[i]}(k+1)$ by taking a convex combination $v^{[i]}_y(k)$ of its estimate
$y^{[i]}(k)$ with the estimates of its neighbors at time $k$ and taking
one step to track the variation of the local objective function
$f^{[i]}$. The DLPDS algorithm is formally stated as follows:

Initially, each agent~$i$ picks a common
$\tilde{\mu}\in{\real}^m_{\geq0}$ and computes the set $M^{[i]}$ with some $\theta^{[i]} > 0$ by using
the Distributed Slater-vector Computation Algorithm. Furthermore,
agent~$i$ chooses any initial state $x^{[i]}(0)\in X^{[i]}$, $\mu^{[i]}(0)\in
{\real}^m_{\geq0}$, and $y^{[i]}(1) = N f^{[i]}(x^{[i]}(0))$.

At every $k\geq0$, each agent~$i$ generates
$x^{[i]}(k+1)$, $\mu^{[i]}(k+1)$ and $y^{[i]}(k+1)$ according to the following
rules:
\begin{align}
  &v^{[i]}_x(k) = \sum_{j=1}^N a^i_j(k)x^{[j]}(k),\quad v^{[i]}_{\mu}(k) = \sum_{j=1}^N a^i_j(k)\mu^{[j]}(k),\quad
  v^{[i]}_{y}(k) = \sum_{j=1}^N a^i_j(k)y^{[j]}(k),\nnum\\
  &x^{[i]}(k+1) = P_{X^{[i]}}[v^{[i]}_x(k) - \alpha(k) {\DD}_x^{[i]}(k)],\quad \mu^{[i]}(k+1) = P_{M^{[i]}}[v^{[i]}_{\mu}(k) + \alpha(k) {\DD}_{\mu}^{[i]}(k)],\nnum\\
  &y^{[i]}(k+1) = v^{[i]}_{y}(k) + N(f^{[i]}(x^{[i]}(k)) - f^{[i]}(x^{[i]}(k-1))),\nnum
\end{align}
where $P_{X^{[i]}}$ (resp.~$P_{M^{[i]}}$) is the projection operator onto the
set $X^{[i]}$ (resp.~$M^{[i]}$), the scalars $a^i_j(k)$ are non-negative
weights and the scalars $\alpha(k) > 0$ are step-sizes\footnote{Each
  agent~$i$ executes the update law of $y^{[i]}(k)$ for $k\geq1$.}. We use
the shorthands $\DD_x^{[i]}(k) \equiv \DD \LL^{[i]}_{v_\mu^{[i]}(k)}(v_x^{[i]}(k))$,
and $\DD_\mu^{[i]}(k) \equiv \DD \LL^{[i]}_{v^{[i]}_x(k)}(v^{[i]}_\mu(k))$.



%

The following theorem summarizes the convergence properties of the DLPDS algorithm where agents asymptotically agree upon a pair of primal-dual optimal solutions.

\begin{theorem}[Convergence properties of the DLPDS algorithm]
  Consider the optimization problem~\eqref{e1}. Let the non-degeneracy
  assumption~\ref{asm2}, the balanced communication
  assumption~\ref{asm3} and the periodic strong connectivity
  assumptions~\ref{asm1} hold. Consider the sequences of $\{x^{[i]}(k)\}$,
  $\{\mu^{[i]}(k)\}$ and $\{y^{[i]}(k)\}$ of the distributed Lagrangian
  primal-dual subgradient algorithm with the step-sizes
  $\{\alpha(k)\}$ satisfying
  $\displaystyle{\lim_{k\rightarrow+\infty}\alpha(k) = 0}$,
  $\displaystyle{\sum_{k=0}^{+\infty}\alpha(k) = +\infty}$, and
  $\displaystyle{\sum_{k=0}^{+\infty}\alpha(k)^2 < +\infty}$. Then,
  there is a pair of primal and Lagrangian dual optimal solutions
  $(x^*, \mu^*)\in X^*\times D^*_L$ such that
  $\displaystyle{\lim_{k\rightarrow+\infty}\|x^{[i]}(k) - x^*\| = 0}$ and
  $\displaystyle{\lim_{k\rightarrow+\infty}\|\mu^{[i]}(k) - \mu^*\| = 0}$
  for all $i\in V$. Furthermore, we have that
  $\displaystyle{\lim_{k\rightarrow+\infty}\|y^{[i]}(k) - p^*\| = 0}$ for
  all $i\in V$.\label{the1}
\end{theorem}

\begin{remark} For a convex-concave function, continuous-time gradient-based methods are proved in~\cite{KJA-LH-HU:58} to converge globally towards the saddle-point. Recently, ~\cite{AN-AO:08a} presents (discrete-time) primal-dual subgradient methods which relax the differentiability in~\cite{KJA-LH-HU:58} and further incorporate state constraints. The method in~\cite{KJA-LH-HU:58} is adopted by~\cite{PM-ME:08} and~\cite{AR:08} to study a distributed optimization problem on fixed graphs where objective functions are separable.

The DLPDS algorithm is a generalization of primal-dual subgradient
methods in~\cite{AN-AO:08a} to the networked multi-agent
scenario. It is also an extension of the distributed projected
subgradient algorithm in~\cite{AN-AO-PAP:08} to solve
  multi-agent convex optimization problems with inequality constraints. Additionally, the DLPDS
  algorithm enables agents to find the optimal value. Furthermore, the DLPDS algorithm objective is that of reaching a saddle point of the Lagrangian function in contrast to achieving a (primal) optimal solution in~\cite{AN-AO-PAP:08}.\oprocend\label{rem2}
\end{remark}

\section{Case (ii): identical local constraint
  sets}\label{sec:inequality}

In last section, we study the case where the equality constraint is absent in problem~\eqref{e2}. In this section, we turn our attention to another case of problem~\eqref{e2} where $h(x) = 0$ is taken into account but we require that local constraint sets are identical; i.e., $X^{[i]} = X$ for all $i\in V$. We
first adopt a penalty relaxation and provide a penalty saddle-point
characterization of primal problem~\eqref{e2} with $X^{[i]} = X$. We
then introduce the distributed penalty primal-dual subgradient
algorithm, followed by its convergence properties and some remarks.

\subsection{Preliminaries}

Some preliminary results are presented in this part, and these results are essential to the development of the distributed penalty primal-dual subgradient algorithm.

\subsubsection{A penalty saddle-point characterization}

Note that the primal problem~\eqref{e2} with $X^{[i]} = X$ is trivially
equivalent to the following:
\begin{align}
  \min_{x\in {\real}^n } f(x),\quad {\rm s.t.} \quad N g(x)\leq
  0,\quad N h(x) = 0, \quad x\in X,\label{e62}
\end{align}
with associated penalty dual problem given by \begin{align}
  \max_{\mu\in{\real}^m,\lambda\in{\real}^{\nu}}
  q_P(\mu,\lambda),\quad {\rm s.t.}\quad \mu\geq0,\quad
  \lambda\geq0.\label{e63}\end{align} Here, the penalty dual function,
$q_P : {\real}^m_{\geq0}\times{\real}^{\nu}_{\geq0}\rightarrow
{\real}$, is defined by $q_P(\mu,\lambda) := \inf_{x\in X}
{\HH}(x,\mu,\lambda),$ where ${\HH} :
{\real}^n\times{\real}^m_{\geq0}\times{\real}^{\nu}_{\geq0}
\rightarrow {\real}$ is the \emph{penalty function} given by
${\HH}(x,\mu,\lambda) = f(x) + N \mu^T [g(x)]^+ + N \lambda^T
|h(x)|$. We denote the penalty dual optimal value by $d^*_P$ and the
set of penalty dual optimal solutions by $D^*_P$. We define the
penalty function ${\HH}^{[i]}(x,\mu,\lambda) :
{\real}^n\times{\real}^m_{\geq0}\times{\real}^{\nu}_{\geq0}
\rightarrow {\real}$ for each agent~$i$ as follows:
${\HH}^{[i]}(x,\mu,\lambda) = f^{[i]}(x) + \mu^T [g(x)]^+ +
\lambda^T|h(x)|$. In this way, we have that ${\HH}(x,\mu,\lambda) =
\sum_{i=1}^N {\HH}^{[i]}(x,\mu,\lambda)$. As proven in the next lemma, the
Slater's condition~\ref{asm5} ensures zero duality gap and the
existence of penalty dual optimal solutions.

\begin{lemma}[Strong duality and non-emptyness of the penalty dual optimal set]
  The values of $p^*$ and $d^*_P$ coincide, and $D^*_P$ is
  non-empty.\label{lem1}
\end{lemma}

\begin{proof}
  Consider the auxiliary Lagrangian function ${\LL}_a :
  {\real}^n\times{\real}^m_{\geq0}\times{\real}^{\nu} \rightarrow
  {\real}$ given by ${\LL}_a(x,\mu,\lambda) =  f(x) + N \mu^T g(x) + N
  \lambda^T h(x)$, with the associated dual problem defined by
\begin{align}
  \max_{\mu\in{\real}^m,\lambda\in{\real}^{\nu}}
  q_a(\mu,\lambda),\quad {\rm s.t.}\quad \mu\geq0.\label{e21}
\end{align}
Here, the dual function, $q_a : {\real}^m_{\geq0}\times{\real}^{\nu}
\rightarrow {\real}$, is defined by $q_a(\mu,\lambda) := \inf_{x\in X}
{\LL}_a(x,\mu,\lambda)$. The dual optimal value of problem~\eqref{e21} is denoted by $d_a^*$ and the set of dual optimal solutions is denoted $D_a^*$. Since $X$ is convex, $f$ and $g_{\ell}$,
for $\ell\in \until{m}$, are convex, $p^*$ is finite and the Slater's
condition~\ref{asm5} holds, it follows from Proposition 5.3.5
in~\cite{DPB:09} that $p^* = d_a^*$ and $D_a^* \neq \emptyset$. We now proceed to characterize $d^*_P$ and $D^*_P$. Pick any $(\mu^*,\lambda^*)\in D_a^*$. Since $\mu^*\geq0$, then
\begin{align}
  d_a^* &= q_a(\mu^*,\lambda^*) = \inf_{x\in
    X}\{f(x)+N(\mu^*)^Tg(x)+N(\lambda^*)^Th(x)\}\nnum\\&\leq
  \inf_{x\in X}\{f(x)+N(\mu^*)^T[g(x)]^++N|\lambda^*|^T|h(x)|\} =
  q_P(\mu^*,|\lambda^*|) \leq d^*_P.\label{e59}
\end{align}
On the other hand, pick any $x^*\in X^*$. Then $x^*$ is feasible,
i.e., $x^*\in X$, $[g(x^*)]^+ = 0$ and $|h(x^*)| = 0$. It implies that
$q_P(\mu,\lambda) \leq {\HH}(x^*,\mu,\lambda) = f(x^*) = p^*$ holds
for any $\mu\in{\real}^m_{\geq0}$ and
$\lambda\in{\real}^{\nu}_{\geq0}$, and thus $d^*_P =
\sup_{\mu\in{\real}^m_{\geq0},\lambda\in{\real}^{\nu}_{\geq0}}q_P(\mu,\lambda) \leq
p^* = d^*_a$. Therefore, we have $d^*_P = p^*$.

To prove the emptyness of $D^*_P$, we pick any $(\mu^*,\lambda^*)\in
D_a^*$. From~\eqref{e59} and $d_a^* = d^*_P$, we can see that
$(\mu^*,|\lambda^*|)\in D^*_P$ and thus $D^*_P\neq \emptyset$.
\end{proof}




The following is a slight extension of Theorem~\ref{the2} to penalty functions.
\begin{theorem}[Penalty Saddle-point Theorem]
  The pair of $(x^*, \mu^*, \lambda^*)$ is a saddle point of the
  penalty function $\HH$ over $X\times {\real}^m_{\geq0}\times
  {\real}^{\nu}_{\geq0}$ if and only if it is a pair of primal and
  penalty dual optimal solutions and the following \emph{penalty
    minimax equality} holds:
  \begin{align*}\sup_{(\mu,\lambda)\in{\real}^m_{\geq0}\times{\real}^{\nu}_{\geq0}}
  \inf_{x\in
    X}{\HH}(x,\mu,\lambda) = \inf_{x\in
    X}\sup_{(\mu,\lambda)\in{\real}^m_{\geq0}\times{\real}^{\nu}_{\geq0}}
    {\HH}(x,\mu,\lambda).\end{align*}\label{the0}
\end{theorem}


\begin{proof}
The proof is analogous to that of Proposition 6.2.4
  in~\cite{DPB-AN-AO:03a}, and for the sake of completeness, we provide
  the details here. It follows from Proposition 2.6.1
  in~\cite{DPB-AN-AO:03a} that $(x^*, \mu^*, \lambda^*)$ is a saddle
  point of $\HH$ over $X\times {\real}^m_{\geq0}\times
  {\real}^{\nu}_{\geq0}$ if and only if the penalty minimax equality
  holds and the following conditions are satisfied:
  \begin{align}
    \sup_{(\mu,\lambda)\in{\real}^m_{\geq0}\times{\real}^{\nu}_{\geq0}}{\HH}(x^*,\mu,\lambda) &= \min_{x\in X}\{\sup_{(\mu,\lambda)\in{\real}^m_{\geq0}\times{\real}^{\nu}_{\geq0}}{\HH}(x,\mu,\lambda)\},\label{e50}\\
    \inf_{x\in X}{\HH}(x,\mu^*,\lambda^*) &=
    \max_{(\mu,\lambda)\in{\real}^m_{\geq0}\times{\real}^{\nu}_{\geq0}}\{\inf_{x\in
      X}{\HH}(x,\mu,\lambda)\}\label{e64}.
  \end{align}
  Notice that $\inf_{x\in X}{\HH}(x,\mu,\lambda) = q_P(\mu,\lambda)$;
  and if $x\in Y$, then
  $\sup_{(\mu,\lambda)\in{\real}^m_{\geq0}\times{\real}^{\nu}_{\geq0}}{\HH}(x,\mu,\lambda)
  = f(x)$, otherwise,
  $\sup_{(\mu,\lambda)\in{\real}^m_{\geq0}\times{\real}^{\nu}_{\geq0}}{\HH}(x,\mu,\lambda)
  = +\infty$. Hence, the penalty minimax equality is equivalent to
  $d^*_P = p^*$. Condition~\eqref{e50} is equivalent to the fact that
  $x^*$ is primal optimal, and condition~\eqref{e64} is equivalent to
  $(\mu^*,\lambda^*)$ being a penalty dual optimal
  solution.
\end{proof}


\subsubsection{Convexity of $\HH$}

Since $g_{\ell}$ is convex and $[\cdot]^+$ is convex and
non-decreasing, thus $[g_{\ell}(x)]^+$ is convex in $x$ for each $\ell
\in \until{m}$. Denote $A := (a_1^T,\cdots,a_{\nu}^T)^T$. Since
$|\cdot|$ is convex and $a_{\ell}^T x - b_{\ell}$ is an affine
mapping, then $|a_{\ell}^T x - b_{\ell}|$ is convex in $x$ for each
$\ell \in \until{\nu}$.

We denote $w := (\mu,\lambda)$. For each
$w\in{\real}^m_{\geq0}\times{\real}^{\nu}_{\geq0}$, we define the
function ${\HH}^{[i]}_{w} : {\real}^n\rightarrow {\real}$ as ${\HH}^{[i]}_{w}(x)
: = {\HH}^{[i]}(x,w)$. Note that ${\HH}^{[i]}_{w}(x)$ is convex in $x$ by using
the fact that a nonnegative weighted sum of convex functions is
convex. For each $x\in{\real}^n$, we define the function ${\HH}^{[i]}_{x} :
{\real}^m_{\geq0}\times{\real}^{\nu}_{\geq0}\rightarrow {\real}$ as
${\HH}^{[i]}_{x}(w) : = {\HH}^{[i]}(x,w)$. It is easy to check that
${\HH}^{[i]}_{x}(w)$ is concave (actually affine) in $w$. Then the penalty
function ${\HH}(x,w)$ is the sum of convex-concave local functions.

\begin{remark}
  The Lagrangian relaxation does not fit to our approach here since
  the Lagrangian function is not convex in $x$ by allowing $\lambda$
  entries to be negative. \oprocend\label{rem7}
\end{remark}

%

\subsection{Distributed penalty primal-dual subgradient
  algorithm}\label{sec:algorithm2}

We are now in the position to devise the \emph{Distributed Penalty Primal-Dual Subgradient}
Algorithm (DPPDS, for short), that is based on the penalty
saddle-point theorem~\ref{the0}, to find the optimal value and a
primal optimal solution to primal problem~\eqref{e2} with $X^{[i]} =
X$. The DPPDS algorithm is formally described as follow.

Initially, agent~$i$ chooses any initial state $x^{[i]}(0)\in X$,
$\mu^{[i]}(0)\in {\real}^m_{\geq0}$, $\lambda^{[i]}(0)\in
{\real}^{\nu}_{\geq0}$, and $y^{[i]}(1) = N f^{[i]}(x^{[i]}(0))$. At every time
$k\geq0$, each agent~$i$ computes the following convex combinations:
\begin{align*}
  &v^{[i]}_x(k) = \sum_{j=1}^N a^i_j(k)x^{[j]}(k),\quad v^{[i]}_{y}(k) =
  \sum_{j=1}^N a^i_j(k)y^{[j]}(k),\nnum\\
  & v^{[i]}_{\mu}(k) = \sum_{j=1}^N a^i_j(k)\mu^{[j]}(k),\quad
  v^{[i]}_{\lambda}(k) = \sum_{j=1}^N a^i_j(k)\lambda^{[j]}(k),
\end{align*}
and generates $x^{[i]}(k+1)$, $y^{[i]}(k+1)$, $\mu^{[i]}(k+1)$ and
$\lambda^{[i]}(k+1)$ according to the following rules:
\begin{align}
  &x^{[i]}(k+1) = P_{X}[v^{[i]}_x(k) - \alpha(k) {\mathcal{S}}_x^{[i]}(k)],\quad y^{[i]}(k+1) = v^{[i]}_{y}(k) + N(f^{[i]}(x^{[i]}(k))-f^{[i]}(x^{[i]}(k-1))),\nnum\\
  &\mu^{[i]}(k+1) = v^{[i]}_{\mu}(k) + \alpha(k) [g(v^{[i]}_x(k))]^+,\quad
  \lambda^{[i]}(k+1) = v^{[i]}_{\lambda}(k) + \alpha(k)|h(v^{[i]}_x(k))|,\nnum
\end{align}
where $P_{X}$ is the projection operator onto the set $X$, the scalars
$a^i_j(k)$ are non-negative weights and the positive scalars
$\{\alpha(k)\}$ are step-sizes\footnote{Each agent~$i$ executes the
  update law of $y^{[i]}(k)$ for $k\geq1$.}. The vector
\begin{align*}{\mathcal{S}}^{[i]}_x(k) := {\DD} f^{[i]}(v^{[i]}_x(k)) +
\sum_{\ell=1}^mv^{[i]}_{\mu}(k)_{\ell}{\DD} [g_{\ell}(v^{[i]}_x(k))]^+ +
\sum_{\ell=1}^{\nu}v^{[i]}_{\lambda}(k)_{\ell}{\DD} |h_{\ell}|(v^{[i]}_x(k))\end{align*}
is a subgradient of ${\HH}^{[i]}_{w^{[i]}(k)}(x)$ at $x = v^{[i]}_x(k)$ where
$w^{[i]}(k) := (v_{\mu}^{[i]}(k), v_{\lambda}^{[i]}(k))$ is the convex combination of dual estimates of agent $i$ and its neighbors'.

Given a step-size sequence $\{\alpha(k)\}$, we define its sum over $[0,k]$ by $s(k) :=
\sum_{\ell=0}^k \alpha(\ell)$ and assume that:
\begin{assumption}[Step-size assumption]
  The step-sizes satisfy
  $\displaystyle{\lim_{k\rightarrow+\infty}\alpha(k) = 0}$,
  ${\sum_{k=0}^{+\infty}\alpha(k) = +\infty}$,
  ${\sum_{k=0}^{+\infty}\alpha(k)^2 < +\infty}$, and
  $\displaystyle{\lim_{k\rightarrow+\infty}\alpha(k+1)s(k) = 0}$,
  $\sum_{k=0}^{+\infty}\alpha(k+1)^2s(k) < +\infty$,
  $\sum_{k=0}^{+\infty}\alpha(k+1)^2s(k)^2 < +\infty$.\label{asm6}
\end{assumption}

The following theorem is the main result of this section,
characterizing convergence properties of the DPPDS algorithm where a optimal solution and the optimal value are asymptotically agreed upon.

\begin{theorem}[Convergence properties of the DPPDS algorithm]
  Consider the problem~\eqref{e2} with $X^{[i]} = X$. Let the non-degeneracy
  assumption~\ref{asm2}, the balanced communication
  assumption~\ref{asm3} and the periodic strong connectivity assumption~\ref{asm1} hold. Consider the sequences of $\{x^{[i]}(k)\}$ and $\{y^{[i]}(k)\}$ of the distributed penalty primal-dual subgradient
  algorithm where the step-sizes $\{\alpha(k)\}$ satisfy the step-size
  assumption~\ref{asm6}. Then there exists a primal optimal solution
  $\tilde{x}\in X^*$ such that
  $\displaystyle{\lim_{k\rightarrow+\infty}\|x^{[i]}(k) - \tilde{x}\| =
    0}$ for all $i\in V$. Furthermore, we have
  $\displaystyle{\lim_{k\rightarrow+\infty}\|y^{[i]}(k) - p^*\| = 0}$ for
  all $i\in V$.\label{the3}
\end{theorem}

We here provide some remarks to conclude this section,.

\begin{remark}
  As primal-dual (sub)gradient algorithm in~\cite{KJA-LH-HU:58,AN-AO:08a}, the
  DPPDS algorithm produces a pair of primal and dual estimates at each
  step. Main differences include: firstly, the DPPDS algorithm extends
  the primal-dual subgradient algorithm in~\cite{AN-AO:08a} to the
  multi-agent scenario; secondly, it further takes the equality
  constraint into account. The presence of the equality constraint can
  make $D^*_P$ unbounded. Therefore, unlike the DLPDS algorithm, the
  DPPDS algorithm does not involve the dual projection steps onto
  compact sets. This may cause the subgradient ${\mathcal{S}}_x^{[i]}(k)$
  not to be uniformly bounded, while the boundedness of subgradients
  is a standard assumption in the analysis of subgradient methods,
  e.g.,
  see~\cite{DPB:09,DPB-AN-AO:03a,AN-AO:08b,AN-AO:09,AN-AO:08a,AN-AO-PAP:08}. This
  difficulty will be addressed by a more careful choice of the
  step-size policy; i.e, assumption~\ref{asm6}, which is stronger than
  the more standard diminishing step-size scheme, e.g., in the DLPDS
  algorithm and~\cite{AN-AO-PAP:08}. We require this condition in
  order to prove, in the absence of the boundedness of
  $\{{\mathcal{S}}^{[i]}_x(k)\}$, the existence of a number of limits and
  summability of expansions toward Theorem~\ref{the3}. Thirdly, the
  DPPDS algorithm adopts the penalty relaxation instead of the
  Lagrangian relaxation
  in~\cite{AN-AO:08a}. \oprocend\label{rem1}\end{remark}

\begin{remark}
  Observe that $\mu^{[i]}(k)\geq0$, $\lambda^{[i]}(k)\geq0$ and $v^{[i]}_x(k)\in X$ (due to the fact that $X$ is convex). Furthermore,
  $([g(v^{[i]}_x(k))]^+, |h(v^{[i]}_x(k))|)$ is a supgradient of
  ${\HH}^{[i]}_{v^{[i]}_x(k)}(w^{[i]}(k))$; i.e. the following \emph{penalty
    supgradient inequality} holds for any $\mu\in{\real}^m_{\geq0}$
  and $\lambda\in{\real}^{\nu}_{\geq0}$:
\begin{align}
  &([g(v^{[i]}_x(k))]^+)^T(\mu-v^{[i]}_{\mu}(k)) +
  |h(v^{[i]}_x(k))|^T(\lambda-v^{[i]}_{\lambda}(k))\nnum\\
  &\geq{\HH}^{[i]}(v_x^{[i]}(k),\mu,\lambda)
  -{\HH}^{[i]}(v^{[i]}_x(k),v_{\mu}^{[i]}(k),v_{\lambda}^{[i]}(k)).\label{e56}
\end{align}\oprocend\label{rem8}
\end{remark}

\begin{remark}
  A
  step-size sequence that satisfies the step-size
  assumption~\ref{asm6} is the harmonic series $\{\alpha(k) =
  \frac{1}{k+1}\}_{k\in{\mathbb{Z}}_{\geq0}} $. It is obvious that
  $\displaystyle{\lim_{k\rightarrow+\infty}\frac{1}{k+1} = 0}$, and
  well-known that ${\sum_{k=0}^{+\infty}\frac{1}{k+1} = +\infty}$ and
  ${\sum_{k=0}^{+\infty}\frac{1}{(k+1)^2} < +\infty}$. We now proceed
  to check the property of
  $\displaystyle{\lim_{k\rightarrow+\infty}\alpha(k+1)s(k) = 0}$. For
  any $k\geq1$, there is an integer $n\geq1$ such that $2^{n-1}\leq k
  <2^n$. It holds that
\begin{align*}
  &s(k) \leq s(2^n) = 1 + \frac{1}{2} + (\frac{1}{3}+\frac{1}{4}) +
  \cdots+(\frac{1}{2^{n-1}+1}+\cdots+\frac{1}{2^n})\nnum\\&\leq 1 +
  \frac{1}{2} + (\frac{1}{3}+\frac{1}{3}) +
  \cdots+(\frac{1}{2^{n-1}+1}+\cdots+\frac{1}{2^{n-1}+1})\nnum\\&\leq
  1 + 1 + 1 + \cdots+1 = n \leq \log_2k + 1.
\end{align*}
Then we have $\displaystyle{\limsup_{k\rightarrow+\infty}
  \frac{s(k)}{k+2} \leq \lim_{k\rightarrow+\infty} \frac{\log_2k +
    1}{k+2} = 0}$. Obviously,
$\displaystyle{\liminf_{k\rightarrow+\infty} \frac{s(k)}{k+2} \geq
  0}$. Then we have the property of
$\displaystyle{\lim_{k\rightarrow+\infty}\alpha(k+1)s(k) = 0}$. Since
$\log_2 k \leq (\log_2 k)^2 < (k+2)^{\frac{1}{2}}$, then
\begin{align}
  &\sum_{k=0}^{+\infty}\alpha(k+1)^2s(k)^2\leq
  \sum_{k=0}^{+\infty}\frac{(\log_2k+1)^2}{(k+2)^2}
  =\sum_{k=0}^{+\infty}\big(\frac{(\log_2k)^2}{(k+2)^2}
  +\frac{2\log_2k}{(k+2)^2}+\frac{1}{(k+2)^2}\big)\nnum\\
  &\leq
  \sum_{k=0}^{+\infty}\frac{1}{(k+2)^{\frac{3}{2}}}
  +\sum_{k=0}^{+\infty}\frac{2}{(k+2)^{\frac{3}{2}}}
  + \sum_{k=0}^{+\infty}\frac{1}{(k+2)^2}<+\infty.\nnum
\end{align}
Additionally, we have $\sum_{k=0}^{+\infty}\alpha(k+1)^2s(k) \leq
\sum_{k=0}^{+\infty}\alpha(k+1)^2s(k)^2 <
+\infty$. \oprocend\label{rem6}
\end{remark}

\section{Convergence analysis}{\label{sec:analysis}}

In this sectiob, we provide the proofs for the main results, Theorem~\ref{the1} and~\ref{the3}, of this paper. We
start our analysis by providing some useful properties of the
sequences weighted by $\{\alpha(k)\}$.

\begin{lemma}[Convergence properties of weighted sequences]
  Let $K\geq0$. Consider the sequence $\{\delta(k)\}$ defined by
  $\delta(k) := \frac{\sum_{\ell =
      K}^{k-1}\alpha(\ell)\rho(\ell)}{\sum_{\ell =
      K}^{k-1}\alpha(\ell)}$ where $k\geq K+1$, $\alpha(k)>0$ and
  $\sum_{k=K}^{+\infty}\alpha(k) = +\infty$.

  (a) If $\displaystyle{\lim_{k\rightarrow+\infty}\rho(k) = +\infty}$,
  then $\displaystyle{\lim_{k\rightarrow+\infty}\delta(k) = +\infty}$.

  (b) If $\displaystyle{\lim_{k\rightarrow+\infty}\rho(k) = \rho^*}$,
  then $\displaystyle{\lim_{k\rightarrow+\infty}\delta(k) =
    \rho^*}$. \label{lem8}
\end{lemma}

\begin{proof}
  (a) For any $\Pi > 0$, there exists $k_1\geq K$ such that
  $\rho(k)\geq\Pi$ for all $k\geq k_1$. Then the following holds for
  all $k\geq k_1 + 1$:
  \begin{align*}
    \delta(k) \geq \frac{1}{\sum_{\ell =
        K}^{k-1}\alpha(\ell)}(\sum_{\ell=K}^{k_1-1}\alpha(\ell)\rho(\ell)
    + \sum_{\ell=k_1}^{k-1}\alpha(\ell)\Pi) =\Pi + \frac{1}{\sum_{\ell
        =
        K}^{k-1}\alpha(\ell)}(\sum_{\ell=K}^{k_1-1}\alpha(\ell)\rho(\ell)
    - \sum_{\ell=K}^{k_1-1}\alpha(\ell)\Pi).
  \end{align*}
  Take the limit on $k$ in the above estimate and we
  have $\displaystyle{\liminf_{k\rightarrow+\infty}\delta(k) \geq \Pi}$.
  Since $\Pi$ is arbitrary, then
  $\displaystyle{\lim_{k\rightarrow+\infty}\delta(k) = +\infty}$.

  (b) For any $\epsilon > 0$, there exists $k_2\geq K$ such that
  $\|\rho(k) - \rho^*\| \leq\epsilon$ for all $k\geq k_2+1$. Then we
  have
  \begin{align*}
    &\|\delta(k) - \rho^*\| = \|\frac{\sum_{\tau =
        K}^{k-1}\alpha(\tau)(\rho(\tau)-\rho^*)}{\sum_{\tau =
        K}^{k-1}\alpha(\tau)}\|&\nnum\\
    &\leq \frac{1}{\sum_{\tau =
        K}^{k-1}\alpha(\tau)}(\sum_{\tau=K}^{k_2-1}\alpha(\tau)\|\rho(\tau)-\rho^*\|
    + \sum_{\tau=k_2}^{k-1}\alpha(\tau)\epsilon)\leq
    \frac{\sum_{\tau=K}^{k_2-1}\alpha(\tau)\|\rho(\tau)-\rho^*\|}{\sum_{\tau
        = K}^{k-1}\alpha(\tau)} + \epsilon.
\end{align*}
Take the limit on $k$ in the above estimate and we have
$\displaystyle{\limsup_{k\rightarrow+\infty}\|\delta(k) - \rho^*\|
  \leq \epsilon}$. Since $\epsilon$ is arbitrary, then
$\displaystyle{\lim_{k\rightarrow+\infty}\|\delta(k) - \rho^*\| = 0}$.
\end{proof}

\subsection{Proofs of Theorem~\ref{the1}}

We now proceed to show Theorem~\ref{the1}. To do that, we first
rewrite the DLPDS algorithm into the following form:
\begin{align}
  x^{[i]}(k+1) = v^{[i]}_x(k) + e^{[i]}_x(k),\quad \mu^{[i]}(k+1) = v^{[i]}_{\mu}(k) +
  e^{[i]}_{\mu}(k),\quad y^{[i]}(k+1) = v^{[i]}_{y}(k) + u^{[i]}(k),\nnum
\end{align}
where $e^{[i]}_x(k)$ and $e^{[i]}_{\mu}(k)$ are projection errors described by
\begin{align*}
  e^{[i]}_x(k) := P_{X^{[i]}}[v^{[i]}_x(k) - \alpha(k) {\DD}_x^{[i]}(k)] -
  v^{[i]}_x(k),\quad e^{[i]}_{\mu}(k) := P_{M^{[i]}}[v^{[i]}_{\mu}(k) + \alpha(k)
  {\DD}_{\mu}^{[i]}(k)] - v^{[i]}_{\mu}(k),
\end{align*}
and $u^{[i]}(k) := N (f^{[i]}(x^{[i]}(k))-f^{[i]}(x^{[i]}(k-1)))$ is the local input which
allows agent~$i$ to track the variation of the local objective
function~$f^{[i]}$. In this manner, the update law of each estimate is
decomposed in two parts: a convex sum to fuse the information of each
agent with those of its neighbors, plus some local error or
input. With this decomposition, all the update laws are put into the
same form as the dynamic average consensus algorithm in the
Appendix. This observation allows us to divide the analysis of the
DLPDS algorithm in two steps. Firstly, we show all the estimates
asymptotically achieve consensus by utilizing the property that the
local errors and inputs are diminishing. Secondly, we further show
that the consensus vectors coincide with a pair of primal and
Lagrangian dual optimal solutions and the optimal value.

\begin{lemma}[Lipschitz continuity of ${\LL}^{[i]}_x$ and ${\LL}^{[i]}_{\mu}$]
  Consider ${\LL}^{[i]}_{\mu}$ and ${\LL}^{[i]}_{x}$. Then there are $L > 0$ and
  $R > 0$ such that $\|\DD{\LL}^{[i]}_{\mu}(x)\|\leq L$ and
  $\|\DD{\LL}^{[i]}_{x}(\mu)\|\leq R$ for each pair of $x\in
  \co(\cup_{i=1}^N X^{[i]})$ and $\mu\in \co(\cup_{i=1}^N
  M^{[i]})$. Furthermore, for each $\mu\in \co(\cup_{i=1}^N M^{[i]})$, the
  function ${\LL}^{[i]}_{\mu}$ is Lipschitz continuous with Lipschitz
  constant $L$ over $\co(\cup_{i=1}^N X^{[i]})$, and for each
  $x\in\co(\cup_{i=1}^N X^{[i]})$, the function ${\LL}^{[i]}_{x}$ is Lipschitz
  continuous with Lipschitz constant $R$ over $\co(\cup_{i=1}^N
  M^{[i]})$.\label{lem11}
\end{lemma}

\begin{proof}
  Observe that ${\DD}{\LL}^{[i]}_{\mu} = {\DD}f^{[i]} + {\mu}^T {\DD}g$ and
  ${\DD}{\LL}^{[i]}_{x} = g$.  Since $f^{[i]}$ and $g_{\ell}$ are convex, it
  follows from Proposition~5.4.2 in~\cite{DPB:09} that ${\partial}f^{[i]}$
  and ${\partial}g_{\ell}$ are bounded over the compact $
  \co(\cup_{i=1}^N X^{[i]})$. Since $ \co(\cup_{i=1}^N M^{[i]})$ is bounded,
  so is ${\partial}{\LL}^{[i]}_{\mu}$, i.e., for any $\mu\in
  \co(\cup_{i=1}^N M^{[i]})$, there exists $L > 0$ such that
  $\|\partial{\LL}^{[i]}_{\mu}(x)\|\leq L$ for all $x\in \co(\cup_{i=1}^N
  X^{[i]})$. Since $g_\ell$ is continuous (due to its convexity) and
  $\co(\cup_{i=1}^N X^{[i]})$ is bounded, then $g$ and thus
  ${\partial}{\LL}^{[i]}_{x}$ are bounded, i.e., for any $x\in
  \co(\cup_{i=1}^N X^{[i]})$, there exists $R > 0$ such that
  $\|\partial{\LL}^{[i]}_{x}(\mu)\|\leq R$ for all $\mu\in \co(\cup_{i=1}^N
  M^{[i]})$.

  It follows from the Lagrangian subgradient inequality that
  \begin{align*}
    {\DD}{\LL}^{[i]}_{\mu}(x)^T(x'-x)\leq
    {\LL}^{[i]}_{\mu}(x')-{\LL}^{[i]}_{\mu}(x),\quad
    {\DD}{\LL}^{[i]}_{\mu}(x')^T(x-x')\leq
    {\LL}^{[i]}_{\mu}(x)-{\LL}^{[i]}_{\mu}(x'),
  \end{align*}
  for any $x,x'\in\co(\cup_{i=1}^N X^{[i]})$. By using the boundedness of
  the subdifferentials, the above two inequalities give that
  $-L\|x-x'\|\leq{\LL}^{[i]}_{\mu}(x)-{\LL}^{[i]}_{\mu}(x')\leq L\|x-x'\|.$ This
  implies that $\|{\LL}^{[i]}_{\mu}(x)-{\LL}^{[i]}_{\mu}(x')\|\leq L\|x-x'\|$
  for any $x,x'\in\co(\cup_{i=1}^m X^{[i]})$. The proof for the Lipschitz
  continuity of ${\LL}^{[i]}_x$ is analogous by using the Lagrangian
  supgradient inequality.
\end{proof}

The following lemma provides a basic iteration relation used in the
convergence proof for the DLPDS algorithm.


\begin{lemma}[Basic iteration relation]
  Let the balanced communication assumption~\ref{asm3} and the periodic strong connectivity assumption~\ref{asm1} hold. For any $x\in X$, any ${\mu}\in M$ and all $k\geq 0$, the following estimates hold:
\begin{align}
  \sum_{i=1}^N \|e^{[i]}_x(k)+\alpha(k) {\DD}_x^{[i]}(k)\|^2&\leq \sum_{i=1}^N
  \alpha(k)^2\|{\DD}_x^{[i]}(k)\|^2
  +\sum_{i=1}^N\{\|x^{[i]}(k)-x\|^2 - \|x^{[i]}(k+1)-x\|^2\}\nnum\\
  &-\sum_{i=1}^N2\alpha(k)({\LL}^{[i]}(v^{[i]}_x(k),v^{[i]}_{\mu}(k))-{\LL}^{[i]}(x,v^{[i]}_{\mu}(k))),\label{e60}\\
  \sum_{i=1}^N \|e^{[i]}_{\mu}(k) - \alpha(k) {\DD}_{\mu}^{[i]}(k)\|^2&\leq
  \sum_{i=1}^N\alpha(k)^2\|{\DD}_{\mu}^{[i]}(k)\|^2
  +\sum_{i=1}^N\{\|\mu^{[i]}(k)-{\mu}\|^2 - \|\mu^{[i]}(k+1)-{\mu}\|^2\}\nnum\\
  &+\sum_{i=1}^N2\alpha(k)({\LL}^{[i]}(v^{[i]}_x(k),v^{[i]}_{\mu}(k))-{\LL}^{[i]}(v^{[i]}_x(k),{\mu}))\label{e61}.
\end{align}\label{lem2}
\end{lemma}

\begin{proof}
  By Lemma~\ref{lem5} with $Z = M^{[i]}$, $z =
  v^{[i]}_{\mu}(k)+\alpha(k){\DD}^{[i]}_{\mu}(k)$ and $y = \mu\in M$, we have
  that for all $k\geq0$
\begin{align}
  &\sum_{i=1}^N\|e^{[i]}_{\mu}(k) - \alpha(k){\DD}_{\mu}^{[i]}(k)\|^2\leq \sum_{i=1}^N \|v^{[i]}_{\mu}(k) + \alpha(k){\DD}_{\mu}^{[i]}(k)-\mu\|^2 - \sum_{i=1}^N \|\mu^{[i]}(k+1)-\mu\|^2\nnum\\
  &= \sum_{i=1}^N \|v^{[i]}_{\mu}(k) -\mu\|^2 + \sum_{i=1}^N\alpha(k)^2\|{\DD}_{\mu}^{[i]}(k)\|^2\nnum\\
  & + \sum_{i=1}^N 2\alpha(k){\DD}_{\mu}^{[i]}(k)^T(v^{[i]}_{\mu}(k)-\mu) - \sum_{i=1}^N \|\mu^{[i]}(k+1)-\mu\|^2\nnum\\&\leq \sum_{i=1}^N\alpha(k)^2\|{\DD}_{\mu}^{[i]}(k)\|^2  + \sum_{i=1}^N 2\alpha(k){\DD}_{\mu}^{[i]}(k)^T(v^{[i]}_{\mu}(k)-\mu)\nnum\\
  & +\sum_{i=1}^N \|\mu^{[i]}(k) -\mu\|^2 - \sum_{i=1}^N
  \|\mu^{[i]}(k+1)-\mu\|^2.\label{e23}
\end{align}
One can show~\eqref{e61} by substituting the following Lagrangian
supgradient inequality into~\eqref{e23}:
\begin{align*}
  {\DD}_{\mu}^{[i]}(k)^T(\mu-v^{[i]}_{\mu}(k))\geq
  {\LL}^{[i]}(v^{[i]}_x(k),\mu)-{\LL}^{[i]}(v^{[i]}_x(k),v^{[i]}_{\mu}(k)).
\end{align*}
Similarly, equality~\eqref{e60} can be shown by using the
following Lagrangian subgradient inequality: ${\DD}_x^{[i]}(k)^T(x-v^{[i]}_x(k))\leq
  {\LL}^{[i]}(x,v^{[i]}_{\mu}(k))-{\LL}^{[i]}(v^{[i]}_x(k),v^{[i]}_{\mu}(k))$.
\end{proof}

The following lemma shows that the consensus is asymptotically reached.

\begin{lemma}[Achieving consensus]
  Let the non-degeneracy assumption~\ref{asm2}, the balanced communication
  assumption~\ref{asm3} and the periodic strong connectivity assumption~\ref{asm1} hold. Consider the sequences of $\{x^{[i]}(k)\}$, $\{{\mu}^{[i]}(k)\}$ and $\{y^{[i]}(k)\}$ of the DLPDS algorithm with the
  step-size sequence $\{\alpha(k)\}$ satisfying
  $\displaystyle{\lim_{k\rightarrow+\infty}\alpha(k) = 0}$. Then there
  exist $x^*\in X$ and $\mu^*\in M$ such that
  $\displaystyle{\lim_{k\rightarrow+\infty} \|x^{[i]}(k) - x^*\| = 0}$,
  $\displaystyle{\lim_{k\rightarrow+\infty}\|{\mu}^{[i]}(k) - {\mu}^*\| =
    0}$ for all $i\in V$, and
  $\displaystyle{\lim_{k\rightarrow+\infty}\|y^{[i]}(k) - y^{[j]}(k)\| = 0}$
  for all $i,j\in V$.\label{lem3}
\end{lemma}

\begin{proof}
  Observe that $v^{[i]}_x(k)\in \co(\cup_{i=1}^N X^{[i]})$ and
  $v^{[i]}_{\mu}(k)\in \co(\cup_{i=1}^N M^{[i]})$. Then it follows from
  Lemma~\ref{lem11} that $\|{\DD}_x^{[i]}(k)\|\leq L$. From
  Lemma~\ref{lem2} it follows that
\begin{align}
  \sum_{i=1}^N \|x^{[i]}(k+1)-x\|^2&\leq\sum_{i=1}^N \|x^{[i]}(k)-x\|^2+\sum_{i=1}^N \alpha(k)^2L^2\nnum\\
  &+\sum_{i=1}^N 2\alpha(k)
  (\|{\LL}^{[i]}(v^{[i]}_x(k),v^{[i]}_{\mu}(k))\|+\|{\LL}^{[i]}(x,v^{[i]}_{\mu}(k))\|).
  \label{e24}
\end{align}
Notice that $v^{[i]}_x(k)\in\co(\cup_{i=1}^N X^{[i]})$,
$v^{[i]}_{\mu}(k)\in\co(\cup_{i=1}^N M^{[i]})$ and $x\in X$ are bounded. Since
${\LL}^{[i]}$ is continuous, then ${\LL}^{[i]}(v^{[i]}_x(k),v^{[i]}_{\mu}(k))$ and
${\LL}^{[i]}(x,v^{[i]}_{\mu}(k))$ are bounded. Since
$\displaystyle{\lim_{k\rightarrow+\infty}\alpha(k) = 0}$, the last two terms on the
right-hand side of~\eqref{e24} converge to zero as
$k\rightarrow+\infty$. Taking limits on both sides of~\eqref{e24}, one
can see that $\displaystyle{\limsup_{k\rightarrow+\infty}\sum_{i=1}^N
  \|x^{[i]}(k+1)-x\|^2 \leq\liminf_{k\rightarrow+\infty}\sum_{i=1}^N
  \|x^{[i]}(k)-x\|^2}$ for any $x\in X$, and thus
$\displaystyle{\lim_{k\rightarrow+\infty}\sum_{i=1}^N \|x^{[i]}(k)-x\|^2}$
exists for any $x\in X$. On the other hand, taking limits on both
sides of~\eqref{e60} we obtain
$\displaystyle{\lim_{k\rightarrow+\infty}\sum_{i=1}^N\|e^{[i]}_x(k)+\alpha(k){\DD}_x^{[i]}(k)\|^2
  = 0}$ and therefore we deduce that
$\displaystyle{\lim_{k\rightarrow+\infty}\|e^{[i]}_x(k)\| = 0}$ for all
$i\in V$. It follows from Proposition~\ref{pro1} in the Appendix that
$\displaystyle{\lim_{k\rightarrow+\infty}\|x^{[i]}(k) - x^{[j]}(k)\| = 0}$ for
all $i,j\in V$. Combining this with the property that
$\displaystyle{\lim_{k\rightarrow+\infty}\|x^{[i]}(k)-x\|}$ exists for any
$x\in X$, we deduce that there exists $x^*\in {\real}^n$ such that
$\displaystyle{\lim_{k\rightarrow+\infty}\|x^{[i]}(k) - x^*\| = 0}$ for
all $i\in V$. Since $x^{[i]}(k)\in X^{[i]}$ and $X^{[i]}$ is closed, it implies
that $x^*\in X^{[i]}$ for all $i\in V$ and thus $x^*\in X$. Similarly, one
can show that there is ${\mu}^*\in M$ such that
$\displaystyle{\lim_{k\rightarrow+\infty}\|{\mu}^{[i]}(k) - {\mu}^*\| =
  0}$ for all $i\in V$.

Since $\displaystyle{\lim_{k\rightarrow+\infty}\|x^{[i]}(k) - x^*\| = 0}$
and $f^{[i]}$ is continuous, then
$\displaystyle{\lim_{k\rightarrow+\infty}}\|u^{[i]}(k)\| = 0$. It follows
from Proposition~\ref{pro1} that
$\displaystyle{\lim_{k\rightarrow+\infty}\|y^{[i]}(k) - y^{[j]}(k)\| = 0}$ for
all $i,j\in V$.
\end{proof}

From Lemma~\ref{lem3}, we know that the sequences of $\{x^{[i]}(k)\}$ and
$\{{\mu}^{[i]}(k)\}$ of the DLPDS algorithm asymptotically agree on to
some point in $X$ and some point in $M$, respectively. Denote by
$\Theta \subseteq X\times M$ the set of such limit points. We further denote
by the average of primal and dual estimates $\hat{x}(k) := \frac{1}{N}\sum_{i=1}^N x^{[i]}(k)$ and $\hat{\mu}(k) :=
\frac{1}{N}\sum_{i=1}^N\mu^{[i]}(k)$, respectively. The following lemma further
characterizes that the points in $\Theta$ are saddle points of the
Lagrangian function $\LL$ over $X\times M$.

\begin{lemma}[Saddle-point characterization of $\Theta$]
  Each point in $\Theta$ is a saddle point of the Lagrangian function
  $\LL$ over $X\times M$.\label{lem7}
\end{lemma}

\begin{proof}
  Denote by the maximum deviation of primal estimates $\Delta_x(k) := \max_{i,j\in V}\|x^{[j]}(k)-x^{[i]}(k)\|$. Notice that
  \begin{align*}
    &\|v^{[i]}_x(k)-\hat{x}(k)\| = \|\sum_{j=1}^N a^i_j(k)x^{[j]}(k)-\sum_{j=1}^N\frac{1}{N}x^{[j]}(k)\|\nnum\\
    &=\|\sum_{j\neq i}a^i_j(k)(x^{[j]}(k)-x^{[i]}(k)) - \sum_{j\neq i}\frac{1}{N}(x^{[j]}(k)-x^{[i]}(k))\|\nnum\\
    &\leq\sum_{j\neq i}a^i_j(k)\|x^{[j]}(k)-x^{[i]}(k)\|+\sum_{j\neq
      i}\frac{1}{N}\|x^{[j]}(k)-x^{[i]}(k)\|\leq 2\Delta_x(k).
 \end{align*}
 Denote by the maximum deviation of dual estimates $\Delta_{\mu}(k) := \max_{i,j\in
   V}\|\mu^{[j]}(k)-\mu^{[i]}(k)\|$. Similarly, we have
 $\|v^{[i]}_\mu(k)-\hat{\mu}(k)\|\leq 2\Delta_{\mu}(k)$.

 We will show this lemma by contradiction. Suppose that there is
 $(x^*,\mu^*)\in \Theta$ which is not a saddle point of ${\LL}$ over
 $X\times M$. Then at least one of the following equalities holds:
\begin{align}
  \exists x\in X \quad &{\rm s.t.} \quad {\LL}(x^*,\mu^*) > {\LL}(x,\mu^*),\label{e30}\\
  \exists \mu\in M \quad &{\rm s.t.}\quad {\LL}(x^*,\mu) >
  {\LL}(x^*,\mu^*).\label{e29}\end{align} Suppose first
that~\eqref{e30} holds. Then, there exists $\varsigma > 0$ such that
${\LL}(x^*,\mu^*) = {\LL}(x,\mu^*) +\varsigma$. Consider the sequences
of $\{x^{[i]}(k)\}$ and $\{\mu^{[i]}(k)\}$ which converge respectively to
$x^*$ and $\mu^*$ defined above. Notice that estimate~\eqref{e60} leads to
\begin{align*}
  &\sum_{i=1}^N\|x^{[i]}(k+1)-x\|^2 \leq \sum_{i=1}^N\|x^{[i]}(k)-x\|^2\ +
  \alpha(k)^2\sum_{i=1}^N
  \|{\DD}^{[i]}_x(k)\|^2\\
  &-2\alpha(k)\sum_{i=1}^N(A_i(k) + B_i(k) + C_i(k) + D_i(k) + E_i(k)
  + F_i(k)),
\end{align*}
where
\begin{align}
  &A_i(k) :=
  {\LL}^{[i]}(v^{[i]}_x(k),v^{[i]}_{\mu}(k))-{\LL}^{[i]}(\hat{x}(k),v^{[i]}_{\mu}(k)),\;\;
  B_i(k) := {\LL}^{[i]}(\hat{x}(k),v^{[i]}_{\mu}(k))-{\LL}^{[i]}(\hat{x}(k),\hat{\mu}(k)),\nnum\\
  &C_i(k) := {\LL}^{[i]}(\hat{x}(k),\hat{\mu}(k))-{\LL}^{[i]}(x^*,\hat{\mu}(k)),\;\; D_i(k) := {\LL}^{[i]}(x^*,\hat{\mu}(k))-{\LL}^{[i]}(x^*,\mu^*),\nnum\\
  &E_i(k) := {\LL}^{[i]}(x^*,\mu^*)-{\LL}^{[i]}(x,\mu^*),\;\; F_i(k) =
  {\LL}^{[i]}(x,\mu^*)-{\LL}^{[i]}(x,v^{[i]}_{\mu}(k)).\nnum
\end{align}
It follows from the Lipschitz continuity property of ${\LL}^{[i]}$; see
Lemma~\ref{lem11}, that
\begin{align}
  \|A_i(k)\|&\leq L\|v^{[i]}_x(k)-\hat{x}(k)\|\leq 2L\Delta_x(k),\;\;
  \|B_i(k)\|\leq R\|v^{[i]}_{\mu}(k)-\hat{\mu}(k)\|\leq 2R\Delta_{\mu}(k),\nnum\\
  \|C_i(k)\|&\leq L\|\hat{x}(k)-x^*\|\leq \frac{L}{N}\sum_{i=1}^N \|x^{[i]}(k) - x^*\|,\nnum\\
  \|D_i(k)\|&\leq R\|\hat{\mu}(k)-\mu^*\|\leq \frac{R}{N}\sum_{i=1}^N \|\mu^{[i]}(k) - \mu^*\|,\nnum\\
  \|F_i(k)\|&\leq R\|\mu^*-v^{[i]}_{\mu}(k)\|
\leq R\|\mu^*-\hat{\mu}(k)\| + R\|\hat{\mu}(k)-v^{[i]}_{\mu}(k)\|\nnum\\
&\leq \frac{R}{N}\sum_{i=1}^N\|\mu^*(k)-\mu^{[i]}(k)\|
  +2R\Delta_{\mu}(k).\nnum
\end{align}
Since $\displaystyle{\lim_{k\rightarrow+\infty}\|x^{[i]}(k)-x^*\| = 0}$,
$\displaystyle{\lim_{k\rightarrow+\infty}\|\mu^{[i]}(k) - \mu^*\| = 0}$,
$\displaystyle{\lim_{k\rightarrow+\infty}\Delta_x(k) = 0}$ and
$\displaystyle{\lim_{k\rightarrow+\infty}\Delta_{\mu}(k) = 0}$, then
all $A_i(k), B_i(k), C_i(k), D_i(k), F_i(k)$ converge to zero as
$k\rightarrow+\infty$. Then there exists $k_0\geq 0$ such that for all
$k\geq k_0$, it holds that
\begin{align*}
  \sum_{i=1}^N\|x^{[i]}(k+1)-x\|^2 \leq \sum_{i=1}^N\|x^{[i]}(k)-x\|^2 + N \alpha(k)^2 L^2 - \varsigma\alpha(k).
\end{align*}
Following a recursive argument, we have that for all $k\geq k_0$, it
holds that
\begin{align*}
  \sum_{i=1}^N\|x^{[i]}(k+1)-x\|^2 \leq\sum_{i=1}^N\|x^{[i]}(k_0)-x\|^2 +
  NL^2\sum_{\tau=k_0}^k\alpha(\tau)^2 -
  \varsigma\sum_{\tau=k_0}^k\alpha(\tau).
\end{align*}
Since $\sum_{k=k_0}^{+\infty}\alpha(k) = +\infty$ and
$\sum_{k=k_0}^{+\infty}\alpha(k)^2<+\infty$ and $x^{[i]}(k_0)\in X^{[i]}$,
$x\in X$ are bounded, the above estimate yields a contradiction by
taking $k$ sufficiently large. In other words,~\eqref{e30} cannot
hold. Following a parallel argument, one can show that~\eqref{e29}
cannot hold either. This ensures that each $(x^*,\mu^*)\in \Theta$ is
a saddle point of ${\LL}$ over $X\times M$.
\end{proof}

The combination of (c) in Lemmas~\ref{lem9} and Lemma~\ref{lem7} gives
that, for each $(x^*,\mu^*)\in\Theta$, we have that ${\LL}(x^*,\mu^*)
= p^*$ and $\mu^*$ is Lagrangian dual optimal. We still need to verify
that $x^*$ is a primal optimal solution. We are now in the position to
show Theorem~\ref{the1} based on the following two claims.

\noindent\textbf{Proofs of Theorem~\ref{the1}:}

\noindent \textbf{Claim 1: Each point $(x^*, \mu^*)\in\Theta$ is a
  point in $X^*\times D^*_L$.}

\begin{proof}
  The Lagrangian dual optimality of $\mu^*$ follows from (c) in
  Lemma~\ref{lem9} and Lemma~\ref{lem7}. To characterize the primal
  optimality of $x^*$, we define an auxiliary sequence $\{z(k)\}$ by
  $z(k) := \frac{\sum_{\tau =
      0}^{k-1}\alpha(\tau)\hat{x}(\tau)}{\sum_{\tau =
      0}^{k-1}\alpha(\tau)}$ which is a weighted version of the average of primal estimates. Since
  $\displaystyle{\lim_{k\rightarrow+\infty}\hat{x}(k) = x^*}$, it
  follows from Lemma~\ref{lem8} (b) that
  $\displaystyle{\lim_{k\rightarrow+\infty}z(k) = x^*}$.

  Since $(x^*,\mu^*)$ is a saddle point of $\LL$ over $X\times M$,
  then ${\LL}(x^*,\mu)\leq{\LL}(x^*,\mu^*)$ for any $\mu\in M$; i.e.,
  the following relation holds for any $\mu\in M$:
  \begin{align}
    g(x^*)^T(\mu-\mu^*) \leq 0.\label{e34}
  \end{align}
  Choose $\mu_a = \mu^* + \min_{i\in V}\theta^{[i]}
  \frac{\mu^*}{\|\mu^*\|}$ where $\theta^{[i]} > 0$ is given in the
  definition of $M^{[i]}$. Then $\mu_a\geq0$ and $\|\mu_a\| \leq \|\mu^*\|
  + \min_{i\in V}\theta^{[i]}$ implying $\mu_a\in M$. Letting $\mu =
  \mu_a$ in~\eqref{e34} gives that \begin{align*}\frac{\min_{i\in
      V}\theta^{[i]}}{\|\mu^*\|} g(x^*)^T\mu^* \leq 0.\end{align*} Since $\theta^{[i]} >
  0$, we have $g(x^*)^T\mu^* \leq 0$. On the other hand, we choose
  $\mu_b = \frac{1}{2}\mu^*$ and then $\mu_b\in M$. Letting $\mu =
  \mu_b$ in~\eqref{e34} gives that $-\frac{1}{2}g(x^*)^T\mu^* \leq 0$
  and thus $g(x^*)^T\mu^* \geq 0$. The combination of the above two
  estimates guarantees the property of $g(x^*)^T\mu^* = 0$.

  We now proceed to show $g(x^*)\leq0$ by contradiction. Assume that
  $g(x^*)\leq0$ does not hold. Denote $J^+(x^*) := \{1\leq \ell\leq m
  \; | \; g_{\ell}(x^*) > 0\} \neq \emptyset$ and $\eta :=
  \min_{\ell\in J^+(x^*)}\{g_{\ell}(x^*)\}$.  Then $\eta > 0$. Since
  $g$ is continuous and $v^{[i]}_x(k)$ converges to $x^*$, there exists
  $K\geq0$ such that $g_{\ell}(v^{[i]}_x(k)) \geq \frac{\eta}{2}$ for all
  $k\geq K$ and all $\ell\in J^+(x^*)$. Since $v^{[i]}_{\mu}(k)$ converges
  to $\mu^*$, without loss of generality, we say that $\|v^{[i]}_{\mu}(k)
  - \mu^*\| \leq \frac{1}{2}\min_{i\in V}\theta^{[i]}$ for all $k\geq
  K$. Choose $\hat{\mu}$ such that $\hat{\mu}_{\ell} = \mu^*_{\ell}$
  for $\ell\notin J^+(x^*)$ and $\hat{\mu}_{\ell} = \mu^*_{\ell} +
  \frac{1}{\sqrt{m}}\min_{i\in V}\theta^{[i]}$ for $\ell\in J^+(x^*)$. Since $\mu^*\geq0$
  and $\theta^{[i]} > 0$, thus $\hat{\mu}\geq0$. Furthermore,
  $\|\hat{\mu}\| \leq \|\mu^*\| + \min_{i\in V}\theta^{[i]}$, then
  $\hat{\mu} \in M$. Equating $\mu$ to $\hat{\mu}$ and letting
  ${\DD}_{\mu}^{[i]}(k) = g(v^{[i]}_x(k))$ in the estimate~\eqref{e23}, the
  following holds for $k\geq K$:
\begin{align}
  &N |J^+(x^*)| \eta\min_{i\in V}\theta^{[i]}\alpha(k) \leq
  2\alpha(k)\sum_{i=1}^N\sum_{\ell\in
    J^+(x^*)}g_{\ell}(v^{[i]}_x(k))(\hat{\mu}-v^{[i]}_{\mu}(k))_{\ell}\nnum\\&\leq
  \sum_{i=1}^N \|{\mu}^{[i]}(k) -\hat{\mu}\|^2 - \sum_{i=1}^N \|\mu^{[i]}(k+1)-\hat{\mu}\|^2 + NR^2\alpha(k)^2\nnum\\
  & - 2\alpha(k)\sum_{i=1}^N\sum_{\ell\notin
    J^+(x^*)}g_{\ell}(v^{[i]}_x(k))(\hat{\mu}-v^{[i]}_{\mu}(k))_{\ell}.\label{e31}
\end{align}

Summing~\eqref{e31} over $[K, k-1]$ with $k\geq K+1$, dividing by
$\sum_{\tau=K}^{k-1}\alpha(\tau)$ on both sides, and using $-
\sum_{i=1}^N \|\mu^{[i]}(k) - \hat{\mu}\|^2 \leq 0$, we obtain
\begin{align}
  &N |J^+(x^*)| \eta\min_{i\in V}\theta^{[i]}\leq \frac{1}{\sum_{\tau=K}^{k-1}\alpha(\tau)} \{\sum_{i=1}^N
  \|{\mu}^{[i]}(K) -\hat{\mu}\|^2 +
  NR^2\sum_{\tau=K}^{k-1}\alpha(\tau)^2\nnum\\&-\sum_{\tau=K}^{k-1}2\alpha(\tau)\sum_{i=1}^N\sum_{\ell\notin
    J^+(x^*)}g_{\ell}(v^{[i]}_x(\tau))(\hat{\mu}-v^{[i]}_{\mu}(\tau))_{\ell}\}.\label{e20}
\end{align}

Since ${\mu}^{[i]}(K)\in M^{[i]}$, $\hat{\mu}\in M$ are bounded and
$\sum_{\tau = K}^{+\infty}\alpha(\tau) = +\infty$, then the limit of
the first term on the right hand side of~\eqref{e20} is zero as
$k\rightarrow+\infty$. Since $\sum_{\tau = K}^{+\infty}\alpha(\tau)^2
< +\infty$, then the limit of the second term is zero as
$k\rightarrow+\infty$. Since
$\displaystyle{\lim_{k\rightarrow+\infty}v^{[i]}_x(k) = x^*}$ and
$\displaystyle{\lim_{k\rightarrow+\infty}v^{[i]}_{\mu}(k) = \mu^*}$, thus
$\displaystyle{\lim_{k\rightarrow+\infty}2\sum_{i=1}^N\sum_{\ell\notin
    J^+(x^*)}g_{\ell}(v^{[i]}_x(k))(\hat{\mu}-v^{[i]}_{\mu}(k))_{\ell} =
  0}$. Then it follows from Lemma~\ref{lem8} (b) that then the limit
of the third term is zero as $k\rightarrow+\infty$. Then we have $N
|J^+(x^*)| \eta\min_{i\in V}\theta^{[i]}\leq 0$. Recall that $|J^+(x^*)| >
0$, $\eta > 0$ and $\theta^{[i]} > 0$. Then we reach a contradiction,
implying that $g(x^*)\leq0$.

Since $x^*\in X$ and $g(x^*)\leq 0$, then $x^*$ is a feasible solution
and thus $f(x^*)\geq p^*$. On the other hand, since $z(k)$ is a convex
combination of $\hat{x}(0),\cdots,\hat{x}(k-1)$ and $f$ is convex,
thus we have the following estimate:
\begin{align}
  f(z(k)) \leq \frac{\sum_{\tau=0}^{k-1}\alpha(\tau)f(\hat{x}(\tau))}{\sum_{\tau=0}^{k-1}\alpha(\tau)}
  =
  \frac{1}{\sum_{\tau=0}^{k-1}\alpha(\tau)}\{\sum_{\tau=0}^{k-1}\alpha(\tau){\LL}(\hat{x}(\tau),
  \hat{\mu}(\tau))
  -\sum_{\tau=0}^{k-1}N\alpha(\tau)\hat{\mu}(\tau)^Tg(\hat{x}(\tau))\}.\nnum
\end{align}

Recall the following convergence properties: \begin{align*}\lim_{k\rightarrow+\infty}z(k) = x^*,\quad
\lim_{k\rightarrow+\infty}{\LL}(\hat{x}(k), \hat{\mu}(k)) = {\LL}(x^*,
\mu^*) = p^*,\quad
\lim_{k\rightarrow+\infty}\hat{\mu}(k)^Tg(\hat{x}(k)) = g(x^*)^T\mu^*
= 0.\end{align*} It follows from Lemma~\ref{lem8} (b) that $f(x^*) \leq p^*.$
Therefore, we have $f(x^*) = p^*$, and thus $x^*$ is a primal optimal
point.
\end{proof}

\noindent \textbf{Claim 2: It holds that
  $\displaystyle{\lim_{k\rightarrow+\infty}\|y^{[i]}(k) - p^*\| = 0}$.}

\begin{proof}
  The following can be proven by induction on $k$ for a fixed
  $k'\geq1$:
  \begin{align}
    \sum_{i=1}^N y^{[i]}(k+1) = \sum_{i=1}^N y^{[i]}(k') + N \sum_{\ell=k'}^{k} \sum_{i=1}^N (f^{[i]}(x^{[i]}(\ell)) -
    f^{[i]}(x^{[i]}(\ell-1))).\label{e41}
  \end{align}
  Let $k'=1$ in~\eqref{e41} and recall that initial state $y^{[i]}(1) = N
  f^{[i]}(x^{[i]}(0))$ for all $i\in V$. Then we have
\begin{align}
  \sum_{i=1}^N y^{[i]}(k+1) &= \sum_{i=1}^N y^{[i]}(1) + N\sum_{i=1}^N
  (f^{[i]}(x^{[i]}(k)) - f^{[i]}(x^{[i]}(0))) = N\sum_{i=1}^N f^{[i]}(x^{[i]}(k)).\label{e45}
\end{align}
The combination of~\eqref{e45} with
$\displaystyle{\lim_{k\rightarrow+\infty}\|y^{[i]}(k)-y^{[j]}(k)\| = 0}$ gives
the desired result. We then finish the proofs of Theorem~\ref{the1}.
\end{proof}

\subsection{Proofs of Theorem~\ref{the3}}

In this part, we present the proofs of Theorem~\ref{the3}. In order to analyze the DPPDS algorithm, we first rewrite it into the
following form:
\begin{align}
  & \mu^{[i]}(k+1) = v^{[i]}_{\mu}(k) + u^{[i]}_{\mu}(k),\quad \lambda^{[i]}(k+1) =
  v^{[i]}_{\lambda}(k) + u^{[i]}_{\lambda}(k), \nnum\\ & x^{[i]}(k+1) = v^{[i]}_x(k) +
  e^{[i]}_x(k),\quad y^{[i]}(k+1) = v^{[i]}_{y}(k) + u^{[i]}_y(k),\nnum
\end{align}
where $e^{[i]}_x(k)$ is projection error described by
\begin{align*}
  e^{[i]}_x(k) &:= P_X[v^{[i]}_x(k) - \alpha(k) {\mathcal{S}}_x^{[i]}(k)] -
  v^{[i]}_x(k),
\end{align*}
and $u^{[i]}_{\mu}(k) := \alpha(k)[g(v^{[i]}_x(k))]^+$, $u^{[i]}_{\lambda}(k) :=
\alpha(k)|h(v^{[i]}_x(k))|$, $u^{[i]}_y(k) = N(f^{[i]}(x^{[i]}(k))-f^{[i]}(x^{[i]}(k-1)))$ are
some local inputs. Denote by the maximum deviations of dual estimates $M_{\mu}(k) := \max_{i\in V}\|\mu^{[i]}(k)\|$
and $M_{\lambda}(k) := \max_{i\in V}\|\lambda^{[i]}(k)\|$. We further denote by the averages of primal and dual estimates $\hat{x}(k) := \frac{1}{N}\sum_{i=1}^N x^{[i]}(k)$,
$\hat{\mu}(k) := \frac{1}{N}\sum_{i=1}^N \mu^{[i]}(k)$ and
$\hat{\lambda}(k) := \frac{1}{N}\sum_{i=1}^N \lambda^{[i]}(k)$.

Before showing Lemma~\ref{lem6}, we present some useful facts. Since $X$ is compact,
and $f^{[i]}$, $[g(\cdot)]^+$ and $h$ are continuous, there exist $F, G^+,
H > 0$ such that for all $x\in X$, it holds that $\|f^{[i]}(x)\| \leq F$
for all $i\in V$, $\|[g(x)]^+\| \leq G^+$ and $\|h(x)\|\leq H$. Since
$X$ is a compact set and $f^{[i]}$, $[g_{\ell}(\cdot)]^+$,
$|h_{\ell}(\cdot)|$ are convex, then it follows from Proposition 5.4.2
in~\cite{DPB:09} that there exist $D_F, {D_{G^+}}, D_H > 0$ such that
for all $x\in X$, it holds that $\|{\DD} f^{[i]}(x)\| \leq D_F$ ($i\in
V$), $m\|{\DD} [g_{\ell}(x)]^+\| \leq {D_{G^+}}$ ($1\leq \ell \leq m$)
and $\nu\|{\DD}|h_{\ell}|(x)\| \leq D_H$ ($1\leq \ell \leq
\nu$).

\begin{lemma}[Diminishing and summable properties]
  Suppose the balanced communication assumption~\ref{asm3} and the
  step-size assumption~\ref{asm6} hold.

  (a) It holds that
  $\displaystyle{\lim_{k\rightarrow+\infty}\alpha(k)M_{\mu}(k) = 0}$,
  $\displaystyle{\lim_{k\rightarrow+\infty}\alpha(k)M_{\lambda}(k) =
    0}$,
  $\displaystyle{\lim_{k\rightarrow+\infty}\alpha(k)\|{\mathcal{S}}^{[i]}_x(k)\|
    = 0}$, and the sequences of $\{\alpha(k)^2M^2_{\mu}(k)\}$,
  $\{\alpha(k)^2M^2_{\lambda}(k)\}$ and
  $\{\alpha(k)^2\|{\mathcal{S}}_x^{[i]}(k)\|^2\}$ are summable.

  (b) The sequences $\{\alpha(k)\|\hat{\mu}(k) - v^{[i]}_{\mu}(k)\|\}$,
  $\{\alpha(k)\|\hat{\lambda}(k) - v^{[i]}_{\lambda}(k)\|\}$,
  $\{\alpha(k)M_{\mu}(k)\|\hat{x}(k) - v^{[i]}_x(k)\|\}$,
  $\{\alpha(k)M_{\lambda}(k)\|\hat{x}(k) - v^{[i]}_x(k)\|\}$ and
  $\{\alpha(k)\|\hat{x}(k) - v^{[i]}_x(k)\|\}$ are summable.\label{lem6}
\end{lemma}

\begin{proof}
  (a) Notice that
  \begin{align*}
    \|v^{[i]}_{\mu}(k)\| = \|\sum_{j=1}^N a^i_j(k)\mu^{[j]}(k)\| \leq
    \sum_{j=1}^N a^i_j(k)\|\mu^{[j]}(k)\|\leq \sum_{j=1}^N
    a^i_j(k)M_{\mu}(k) = M_{\mu}(k),
  \end{align*}
  where in the last equality we use the balanced communication
  assumption~\ref{asm3}. Recall that $v^{[i]}_x(k)\in X$. This implies
  that the following inequalities hold for all $k\geq0$:
  \begin{align}
    \|\mu^{[i]}(k+1)\| \leq \|v^{[i]}_{\mu}(k)+\alpha(k)[g(v^{[i]}_x(k))]^+\| \leq
    \|v^{[i]}_{\mu}(k)\|+G^+\alpha(k)\leq M_{\mu}(k)+G^+\alpha(k).\nnum
  \end{align}
  From here, then we deduce the following recursive estimate on
  $M_{\mu}(k+1)$: $M_{\mu}(k+1) \leq
  M_{\mu}(k)+G^+\alpha(k)$. Repeatedly applying the above estimates
  yields that
  \begin{align}
    M_{\mu}(k+1) &\leq M_{\mu}(0)+G^+ s(k)\label{e52}.
  \end{align}
  Similar arguments can be employed to show that
  \begin{align}
    M_{\lambda}(k+1) &\leq M_{\lambda}(0) + H s(k).\label{e53}
  \end{align}
  Since $\displaystyle{\lim_{k\rightarrow+\infty}\alpha(k+1)s(k) = 0}$
  and $\displaystyle{\lim_{k\rightarrow+\infty}\alpha(k) = 0}$, then
  we know that
  $\displaystyle{\lim_{k\rightarrow+\infty}\alpha(k+1)M_{\mu}(k+1) =
    0}$ and
  $\displaystyle{\lim_{k\rightarrow+\infty}\alpha(k+1)M_{\lambda}(k+1)
    = 0}$. Notice that the following estimate on
  ${\mathcal{S}}^{[i]}_x(k)$ holds:
  \begin{align}
    \|{\mathcal{S}}^{[i]}_x(k)\|&\leq D_F + {D_{G^+}}M_{\mu}(k)
    +D_HM_{\lambda}(k).\label{e54}
  \end{align}
  Recall that $\displaystyle{\lim_{k\rightarrow+\infty}\alpha(k) =
    0}$, $\displaystyle{\lim_{k\rightarrow+\infty}\alpha(k)M_{\mu}(k)
    = 0}$ and
  $\displaystyle{\lim_{k\rightarrow+\infty}\alpha(k)M_{\lambda}(k) =
    0}$. Then the result of
  $\displaystyle{\lim_{k\rightarrow+\infty}\alpha(k)\|{\mathcal{S}}^{[i]}_x(k)\|
    = 0}$ follows. By~\eqref{e52}, we have
  \begin{align*}
    \sum_{k=0}^{+\infty}\alpha(k)^2M^2_{\mu}(k)\leq
    \alpha(0)^2M_{\mu}^2(0)+
    \sum_{k=1}^{+\infty}\alpha(k)^2(M_{\mu}(0)+G^+ s(k-1))^2.\nnum
  \end{align*}
  It follows from the step-size assumption~\ref{asm6} that
  $\sum_{k=0}^{+\infty}\alpha(k)^2M^2_{\mu}(k) < +\infty$. Similarly,
  one can show that $\sum_{k=0}^{+\infty}\alpha(k)^2M^2_{\lambda}(k) <
  +\infty$. By using~\eqref{e52},~\eqref{e53} and~\eqref{e54}, we have the
  following estimate:
  \begin{align}
    &\sum_{k=0}^{+\infty}\alpha(k)^2\|{\mathcal{S}}_x^{[i]}(k)\|^2
    \leq \alpha(0)^2(D_F + {D_{G^+}}M_{\mu}(0) +
    D_HM_{\lambda}(0))^2\nnum\\& + \sum_{k=1}^{+\infty}
    \alpha(k)^2(D_F + {D_{G^+}}(M_{\mu}(0)+G^+s(k-1)) +
    D_H(M_{\lambda}(0) + Hs(k-1)))^2.\nnum
  \end{align}
  Then the summability of $\{\alpha(k)^2\}$, $\{\alpha(k+1)^2s(k)\}$
  and $\{\alpha(k+1)^2s(k)^2\}$ verifies that of
  $\{\alpha(k)^2\|{\mathcal{S}}_x^{[i]}(k)\|^2\}$.

  (b) Consider the dynamics of $\mu^{[i]}(k)$ which is in the same form
  as the distributed projected subgradient algorithm in~\cite{AN-AO-PAP:08}. Recall that $\{[g(v^{[i]}_x(k))]^+\}$ is uniformly bounded. Then following from Lemma~\ref{lem4} in the Appendix with
  $Z = {\real}^m_{\geq0}$ and $d^{[i]}(k) = -[g(v^{[i]}_x(k))]^+$, we have the
  summability of $\{\alpha(k)\max_{i\in V}\|\hat{\mu}(k) -
  {\mu}^{[i]}(k)\|\}$. Then $\{\alpha(k)\|\hat{\mu}(k) - v^{[i]}_{\mu}(k)\|\}$
  is summable by using the following set of inequalities:
\begin{align}
  \|\hat{\mu}(k) - v^{[i]}_{\mu}(k)\| \leq \sum_{j=1}^N
  a^i_j(k)\|\hat{\mu}(k) - {\mu}^{[j]}(k)\| \leq \max_{i\in V}
  \|\hat{\mu}(k) - {\mu}^{[i]}(k)\|,\label{e57}
\end{align}
where we use $\sum_{j=1}^N a^i_j(k) = 1$. Similarly, it holds that
$\sum_{k=0}^{+\infty} \alpha(k)\|\hat{\lambda}(k) - v^{[i]}_{\lambda}(k)\|
< +\infty$.

We now consider the evolution of $x^{[i]}(k)$. Recall that $v^{[i]}_x(k)\in
X$. By Lemma~\ref{lem5} with $Z = X$, $z =
v^{[i]}_x(k)-\alpha(k){\mathcal{S}}^{[i]}_x(k)$ and $y = v^{[i]}_x(k)$, we have
\begin{align}
  &\|x^{[i]}(k+1) - v^{[i]}_x(k)\|^2\leq \|v^{[i]}_x(k)-\alpha(k){\mathcal{S}}^{[i]}_x(k)-v^{[i]}_x(k)\|^2\nnum\\
  &- \|x^{[i]}(k+1)-(v^{[i]}_x(k)-\alpha(k){\mathcal{S}}^{[i]}_x(k))\|^2,\nnum
\end{align}
and thus $\|e^{[i]}_x(k) + \alpha(k){\mathcal{S}}^{[i]}_x(k)\|\leq
\alpha(k)\|{\mathcal{S}}^{[i]}_x(k)\|$. With this relation, from
Lemma~\ref{lem4} with $Z = X$ and $d^{[i]}(k) = {\mathcal{S}}^{[i]}_x(k)$, the
following holds for some $\gamma >0$ and $0<\beta<1$:
\begin{align}
  \|x^{[i]}(k) - \hat{x}(k)\|\leq N \gamma
  \beta^{k-1}\sum_{i=0}^N\|x^{[i]}(0)\| +
  2N\gamma\sum_{\tau=0}^{k-1}\beta^{k-\tau}\alpha(\tau)\|{\mathcal{S}}^{[i]}_x(\tau)\|.\label{e39}
\end{align}

Multiplying both sides of~\eqref{e39} by $\alpha(k)M_{\mu}(k)$ and
using~\eqref{e54}, we obtain
\begin{align}
  \alpha(k)M_{\mu}(k)\|x^{[i]}(k) - \hat{x}(k)\|
  &\leq N \gamma \sum_{i=0}^N\|x^{[i]}(0)\|\alpha(k)M_{\mu}(k)\beta^{k-1} + 2N\gamma\alpha(k)M_{\mu}(k)\nnum\\
  &\times\sum_{\tau=0}^{k-1}\beta^{k-\tau}\alpha(\tau)(D_F + {D_{G^+}}
  M_{\mu}(\tau) + D_H M_{\lambda}(\tau)).\nnum
\end{align}

Notice that the above inequalities hold for all $i\in V$. Then by
employing the relation of $ab\leq \frac{1}{2}(a^2+b^2)$ and regrouping
similar terms, we obtain
\begin{align}
  &\alpha(k)M_{\mu}(k)\max_{i\in V}\|x^{[i]}(k) - \hat{x}(k)\|\leq N\gamma\big(\frac{1}{2}\sum_{i=0}^N\|x^{[i]}(0)\|+(D_F+{D_{G^+}}+D_H)\sum_{\tau=0}^{k-1}\beta^{k-\tau}\big)\nnum\\
  &\times\alpha(k)^2M_{\mu}^2(k)+\frac{1}{2}N \gamma
  \sum_{i=0}^N\|x^{[i]}(0)\|\beta^{2(k-1)}
  +N\gamma\sum_{\tau=0}^{k-1}\beta^{k-\tau}\alpha(\tau)^2(D_F+{D_{G^+}}M_{\mu}^2(\tau)+D_HM_{\lambda}^2(\tau)).\nnum
\end{align}

Part (a) gives that $\{\alpha(k)^2M_{\mu}^2(k)\}$ is
summable. Combining this fact with $\sum_{\tau=0}^{k-1}\beta^{k-\tau}
\leq \sum_{k=0}^{+\infty}\beta^k = \frac{1}{1-\beta}$, then we can say
that the first term on the right-hand side in the above estimate is
summable. It is easy to check that the second term is also
summable. It follows from Part (a) that
$\displaystyle{\lim_{k\rightarrow+\infty}\alpha(k)^2(D_F+{D_{G^+}}M_{\mu}^2(k)+D_HM_{\lambda}^2(k))
  = 0}$ and
$\{\alpha(k)^2(D_F+{D_{G^+}}M_{\mu}^2(k)+D_HM_{\lambda}^2(k))\}$ is
summable. Then Lemma 7 in~\cite{AN-AO-PAP:08} with $\gamma_{\ell} =
N\gamma
\alpha(\ell)^2(D_F+{D_{G^+}}M_{\mu}^2(\ell)+D_HM_{\lambda}^2(\ell))$
ensures that the third term is summable. Therefore, the summability of
$\{\alpha(k)M_{\mu}(k)\max_{i\in V}\|x^{[i]}(k) - \hat{x}(k)\|\}$ is
guaranteed. Following the same lines in~\eqref{e57}, one can show the
summability of $\{\alpha(k)M_{\mu}(k)\|v^{[i]}_x(k) -
\hat{x}(k)\|\}$. Following analogous arguments, we have that
$\{\alpha(k)M_{\lambda}(k)\|v^{[i]}_x(k) - \hat{x}(k)\|\}$ and
$\{\alpha(k)\|v^{[i]}_x(k) - \hat{x}(k)\|\}$ are summable.
\end{proof}

\begin{remark} In Lemma~\ref{lem6}, the assumption of all local constraint sets being identical is utilized to find an upper bound of the convergence rate of $\|\hat{x}(k) - v^{[i]}_x(k)\|$ to zero. This property is crucial to establish the summability of expansions pertaining to $\|\hat{x}(k) - v^{[i]}_x(k)\|$ in part (b). \oprocend\label{rem4}
\end{remark}

The following is a basic iteration relation of the DPPDS algorithm.

\begin{lemma}[Basic iteration relation] The following estimates hold
  for any $x\in X$ and
  $(\mu,\lambda)\in{\real}^m_{\geq0}\times{\real}^{\nu}_{\geq0}$:
\begin{align}
  &\sum_{i=1}^N \|e^{[i]}_x(k)+\alpha(k) {\mathcal{S}}_x^{[i]}(k)\|^2 \leq \sum_{i=1}^N \alpha(k)^2\|{\mathcal{S}}_x^{[i]}(k)\|^2\nnum\\
  &-\sum_{i=1}^N2\alpha(k)({\HH}^{[i]}(v^{[i]}_x(k),v_{\mu}^{[i]}(k),v_{\lambda}^{[i]}(k))
  -{\HH}^{[i]}(x,v_{\mu}^{[i]}(k),v_{\lambda}^{[i]}(k)))\nnum\\
  &+\sum_{i=1}^N(\|x^{[i]}(k)-x\|^2 - \|x^{[i]}(k+1)-x\|^2),\label{e27}
\end{align}
and,
\begin{align}
  &0\leq \sum_{i=1}^N(\|\mu^{[i]}(k)-{\mu}\|^2 - \|\mu^{[i]}(k+1)-{\mu}\|^2) + \sum_{i=1}^N(\|\lambda^{[i]}(k)-{\lambda}\|^2 - \|\lambda^{[i]}(k+1)-{\lambda}\|^2)+\nnum\\
  &\sum_{i=1}^N2\alpha(k)({\HH}^{[i]}(v^{[i]}_x(k),v_{\mu}^{[i]}(k),v_{\lambda}^{[i]}(k))
  -{\HH}^{[i]}(v^{[i]}_x(k),{\mu},\lambda))+
  \sum_{i=1}^N\alpha(k)^2(\|[g(v^{[i]}_x(k))]^+\|^2+\|h(v^{[i]}_x(k))\|^2)\label{e28}.
\end{align}\end{lemma}

\begin{proof}
  One can finish the proof by following analogous arguments in
  Lemma~\ref{lem2}.
\end{proof}


\begin{lemma}[Achieving consensus]
  Let us suppose that the non-degeneracy assumption ~\ref{asm2},
  the balanced communication assumption~\ref{asm3} and the periodical strong connectivity
  assumption~\ref{asm1} hold. Consider the sequences of $\{x^{[i]}(k)\}$, $\{\mu^{[i]}(k)\}$, $\{\lambda^{[i]}(k)\}$ and
  $\{y^{[i]}(k)\}$ of the distributed penalty primal-dual subgradient
  algorithm with the step-size sequence $\{\alpha(k)\}$ and the
  associated $\{s(k)\}$ satisfying
  $\displaystyle{\lim_{k\rightarrow+\infty}\alpha(k) = 0}$ and
  $\displaystyle{\lim_{k\rightarrow+\infty}\alpha(k+1)s(k) = 0}$. Then
  there exists $\tilde{x}\in X$ such that
  $\displaystyle{\lim_{k\rightarrow+\infty}\|x^{[i]}(k) - \tilde{x}\| =
    0}$ for all $i\in V$. Furthermore,
  $\displaystyle{\lim_{k\rightarrow+\infty}\|\mu^{[i]}(k) - \mu^{[j]}(k)\| =
    0}$, $\displaystyle{\lim_{k\rightarrow+\infty}\|\lambda^{[i]}(k) -
    \lambda^{[j]}(k)\| = 0}$ and
  $\displaystyle{\lim_{k\rightarrow+\infty}\|y^{[i]}(k) - y^{[j]}(k)\| = 0}$
  for all $i,j\in V$.\label{lem12}
\end{lemma}


\begin{proof}
Similar to~\eqref{e23}, we have
\begin{align}
  \sum_{i=1}^N \|x^{[i]}(k+1)-x\|^2 \leq \sum_{i=1}^N \|x^{[i]}(k) -x\|^2 +
  \sum_{i=1}^N\alpha(k)^2\|{\mathcal{S}}_x^{[i]}(k)\|^2 + \sum_{i=1}^N
  2\alpha(k)\|{\mathcal{S}}_x^{[i]}(k)\|\|v^{[i]}_x(k)-x\|.\nnum
\end{align}
Since
$\displaystyle{\lim_{k\rightarrow+\infty}\alpha(k)\|{\mathcal{S}}_x^{[i]}(k)\|
  = 0}$, the proofs of
$\displaystyle{\lim_{k\rightarrow+\infty}\|x^{[i]}(k) - \tilde{x}\| = 0}$
for all $i\in V$ are analogous to those in Lemma~\ref{lem3}. The
remainder of the proofs can be finished by Proposition~\ref{pro1} with
the properties of
$\displaystyle{\lim_{k\rightarrow+\infty}u^{[i]}_{\mu}(k)} = 0$,
$\displaystyle{\lim_{k\rightarrow+\infty}u^{[i]}_{\lambda}(k)} = 0$ and
$\displaystyle{\lim_{k\rightarrow+\infty}u^{[i]}_y(k) = 0}$ (due to
$\displaystyle{\lim_{k\rightarrow+\infty}x^{[i]}(k) = \tilde{x}}$ and
$f^{[i]}$ is continuous).
\end{proof}


%

We now proceed to show Theorem~\ref{the3} based on five claims.

\noindent\textbf{Proof of Theorem~\ref{the3}:}

\noindent\textbf{Claim 1:} For any $x^*\in X^*$ and
$(\mu^*,\lambda^*)\in D^*_P$, the sequences of
$\{\alpha(k)\big[\sum_{i=1}^N{\HH}^{[i]}(x^*,v_{\mu}^{[i]}(k),v_{\lambda}^{[i]}(k))
-{\HH}(x^*,\hat{\mu}(k),\hat{\lambda}(k))\big]\}$ and
$\{\alpha(k)\big[\sum_{i=1}^N{\HH}^{[i]}(v_x^{[i]}(k),\mu^*,\lambda^*)
-{\HH}(\hat{x}(k),\mu^*,\lambda^*)\big]\}$ are summable.

\begin{proof}
  Observe that
  \begin{align}
    &\|{\HH}^{[i]}(x^*,v_{\mu}^{[i]}(k),v_{\lambda}^{[i]}(k))
    -{\HH}^{[i]}(x^*,\hat{\mu}(k),\hat{\lambda}(k))\|\nnum\\
    &\leq \|v_{\mu}^{[i]}(k)-\hat{\mu}(k)\|\|[g(x^*)]^+\| +
    \|v_{\lambda}^{[i]}(k)-\hat{\lambda}(k)\|\|h(x^*)\|\nnum\\&\leq G^+
    \|v_{\mu}^{[i]}(k)-\hat{\mu}(k)\| + H
    \|v_{\lambda}^{[i]}(k)-\hat{\lambda}(k)\|.\label{e25}
  \end{align}
  By using the summability of $\{\alpha(k)\|\hat{\mu}(k) -
  v^{[i]}_{\mu}(k)\|\}$ and $\{\alpha(k)\|\hat{\lambda}(k) -
  v^{[i]}_{\lambda}(k)\|\}$ in Part (b) of Lemma~\ref{lem6}, we have that
  $\{\alpha(k)\sum_{i=1}^N\|{\HH}^{[i]}(x^*,v_{\mu}^{[i]}(k),v_{\lambda}^{[i]}(k))
  -{\HH}^{[i]}(x^*,\hat{\mu}(k),\hat{\lambda}(k))\|\}$ and thus
  $\{\alpha(k)\big[\sum_{i=1}^N\big({\HH}^{[i]}(x^*,v_{\mu}^{[i]}(k),v_{\lambda}^{[i]}(k))
  -{\HH}^{[i]}(x^*,\hat{\mu}(k),\hat{\lambda}(k))\big)\big]\}$ are
  summable. Similarly, the following estimates hold:
  \begin{align*}
    &\|{\HH}^{[i]}(v_x^{[i]}(k),\mu^*,\lambda^*)
    -{\HH}^{[i]}(\hat{x}(k),\mu^*,\lambda^*)\|\leq \|f^{[i]}(v_x^{[i]}(k))-f^{[i]}(\hat{x}(k))\|\nnum\\
    &+ \|(\mu^*)^T([g(v_x^{[i]}(k))]^+ - [g(\hat{x}(k))]^+)\| + \|(\lambda^*)^T(|h(v^{[i]}_x(k))| - |h(\hat{x}(k))|)\|\nnum\\
    &\leq (D_F + {D_{G^+}}\|\mu^*\| + D_H\|\lambda^*\|)\|v^{[i]}_x(k) -
    \hat{x}(k)\|.
  \end{align*}
  Then the property of $\sum_{k=0}^{+\infty} \alpha(k)\|\hat{x}(k) -
  v^{[i]}_x(k)\| < +\infty$ in Part (b) of Lemma~\ref{lem6} implies the
  summability of the sequence
  $\{\alpha(k)\sum_{i=1}^N\|{\HH}^{[i]}(v_x^{[i]}(k),\mu^*,\lambda^*)
  -{\HH}^{[i]}(\hat{x}(k),\mu^*,\lambda^*)\|\}$ and thus the sequence
  $\{\alpha(k)\sum_{i=1}^N\big({\HH}^{[i]}(v_x^{[i]}(k),\mu^*,\lambda^*)
  -{\HH}^{[i]}(\hat{x}(k),\mu^*,\lambda^*)\big)\}$.
\end{proof}

\noindent\textbf{Claim 2:} Denote the weighted version of the local penalty function $\HH^{[i]}$ over $[0,k-1]$ as
\begin{align*}
  \hat{\HH}^{[i]}(k) := \frac{1}{s(k-1)}
  \sum_{\ell=0}^{k-1}\alpha(\ell){\HH}^{[i]}(v^{[i]}_x(\ell),v_{\mu}^{[i]}(\ell),v_{\lambda}^{[i]}(\ell)).
\end{align*}
The following property holds: $\displaystyle{\lim_{k\rightarrow+\infty}\sum_{i=1}^N
\hat{\HH}^{[i]}(k) = p^*}$.

\begin{proof}
  Summing~\eqref{e27} over $[0,k-1]$ and replacing $x$ by $x^*\in
  X^*$ leads to
  \begin{align}
    &\sum_{\ell=0}^{k-1}\alpha(\ell)\sum_{i=1}^N({\HH}^{[i]}(v^{[i]}_x(\ell),v_{\mu}^{[i]}(\ell),v_{\lambda}^{[i]}(\ell))
    -{\HH}^{[i]}(x^*,v_{\mu}^{[i]}(\ell),v_{\lambda}^{[i]}(\ell)))\nnum\\ &\leq
    \sum_{i=1}^N\|x^{[i]}(0)-{x}^*\|^2 + \sum_{\ell=0}^{k-1}\sum_{i=1}^N
    \alpha(\ell)^2\|{\mathcal{S}}_x^{[i]}(\ell)\|^2.\label{e22}
  \end{align}
  The summability of $\{\alpha(k)^2\|{\mathcal{S}}_x^{[i]}(k)\|^2\}$ in Part (b)
  of Lemma~\ref{lem6} implies that the right-hand side of~\eqref{e22}
  is finite as $k\rightarrow+\infty$, and thus
  \begin{align}
    \limsup_{k\rightarrow\infty}\frac{1}{s(k-1)}\sum_{\ell=0}^{k-1}\alpha(\ell)
    \big[\sum_{i=1}^N\big({\HH}^{[i]}(v^{[i]}_x(\ell),v_{\mu}^{[i]}(\ell),v_{\lambda}^{[i]}(\ell))
    -{\HH}^{[i]}(x^*,v_{\mu}^{[i]}(\ell),v_{\lambda}^{[i]}(\ell))\big)\big] \leq
    0.\label{e32}
  \end{align}

  Pick any $(\mu^*,\lambda^*)\in D^*_P$. It follows from
  Theorem~\ref{the0} that $(x^*,\mu^*,\lambda^*)$ is a saddle point of
  $\HH$ over
  $X\times{\real}^m_{\geq0}\times{\real}^{\nu}_{\geq0}$. Since
  $(\hat{\mu}(k),\hat{\lambda}(k))\in{\real}^m_{\geq0}\times{\real}^{\nu}_{\geq0}$,
  then we have ${\HH}(x^*,\hat{\mu}(k),\hat{\lambda}(k)) \leq
  {\HH}(x^*,\mu^*,\lambda^*) = p^*$. Combining this relation, Claim 1
  and~\eqref{e32} renders that
  \begin{align}
    &\limsup_{k\rightarrow+\infty}\frac{1}{s(k-1)}
    \sum_{\ell=0}^{k-1}\alpha(\ell)\big[\sum_{i=1}^N{\HH}^{[i]}(v^{[i]}_x(\ell),v_{\mu}^{[i]}(\ell),v_{\lambda}^{[i]}(\ell))-p^*\big]\nnum\\
    &\leq \limsup_{k\rightarrow+\infty}\frac{1}{s(k-1)}
    \sum_{\ell=0}^{k-1}\alpha(\ell)\big[\sum_{i=1}^N\big({\HH}^{[i]}(v^{[i]}_x(\ell),v_{\mu}^{[i]}(\ell),v_{\lambda}^{[i]}(\ell))
    -{\HH}^{[i]}(x^*,v_{\mu}^{[i]}(\ell),v_{\lambda}^{[i]}(\ell))\big)\big]\nnum\\
    & + \limsup_{k\rightarrow+\infty}\frac{1}{s(k-1)}\sum_{\ell=0}^{k-1}\alpha(\ell)\big[\sum_{i=1}^N{\HH}^{[i]}(x^*,v_{\mu}^{[i]}(\ell),v_{\lambda}^{[i]}(\ell))
    -{\HH}(x^*,\hat{\mu}(\ell),\hat{\lambda}(\ell))\big]\nnum\\
    &+\limsup_{k\rightarrow+\infty}\frac{1}{s(k-1)}\sum_{\ell=0}^{k-1}
    ({\HH}(x^*,\hat{\mu}(\ell),\hat{\lambda}(\ell))-p^*)\leq
    0,\nnum\end{align}
  and thus $\limsup_{k\rightarrow+\infty}\sum_{i=1}^N \hat{\HH}^{[i]}(k) \leq
    p^*.$

  On the other hand, $\hat{x}(k)\in X$ (due to the fact that $X$ is
  convex) implies that ${\HH}(\hat{x}(k),\mu^*,\lambda^*) \geq
  {\HH}(x^*,\mu^*,\lambda^*) = p^*$. Along  similar lines, by
  using~\eqref{e28} with $\mu = \mu^*$, $\lambda = \lambda^*$, and
  Claim 1, we have the following estimate: $\liminf_{k\rightarrow+\infty}\sum_{i=1}^N \hat{\HH}^{[i]}(k) \geq
    p^*$. Then we have the desired relation.\end{proof}

\noindent\textbf{Claim 3:} Denote by $\pi(k) := \sum_{i=1}^N
  {\HH}^{[i]}(v^{[i]}_x(k),v^{[i]}_{\mu}(k),v^{[i]}_{\lambda}(k)) -
  {\HH}(\hat{x}(k),\hat{\mu}(k),\hat{\lambda}(k))$. And we denote the weighted version of the global penalty function $\HH$ over $[0,k-1]$ as
\begin{align*}
  \hat{\HH}(k) :=
  \frac{1}{s(k-1)}\sum_{\ell=0}^{k-1}\alpha(\ell){\HH}(\hat{x}(\ell),\hat{\mu}(\ell),\hat{\lambda}(\ell)).
\end{align*}
The following property holds: $\displaystyle{\lim_{k\rightarrow+\infty}\hat{\HH}(k) = p^*}$.

\begin{proof}
  Notice that
  \begin{align}
    \pi(k) &= \sum_{i=1}^N (f^{[i]}(v^{[i]}_x(k)) - f^{[i]}(\hat{x}(k)))
    + \sum_{i=1}^N \big(v^{[i]}_{\mu}(k)^T[g(v^{[i]}_x(k))]^+ - v^{[i]}_{\mu}(k)^T[g(\hat{x}(k))]^+\big)\nnum\\
    &+ \sum_{i=1}^N \big(v^{[i]}_{\mu}(k)^T[g(\hat{x}(k))]^+ -
    \hat{\mu}(k)^T[g(\hat{x}(k))]^+\big)
    + \sum_{i=1}^N \big(v^{[i]}_{\lambda}(k)^T |h(v^{[i]}_x(k))| - v^{[i]}_{\lambda}(k)^T|h(\hat{x}(k))|\big)\nnum\\
    &+ \sum_{i=1}^N \big(v^{[i]}_{\lambda}(k)^T |h(\hat{x}(k))| -
    \hat{\lambda}(k)^T|h(\hat{x}(k))|\big).\label{e44}
\end{align}
By using the boundedness of subdifferentials and the primal estimates,
it follows from~\eqref{e44} that
\begin{align}
  &\|\pi(k)\|\leq (D_F + {D_{G^+}} M_{\mu}(k) + D_H M_{\lambda}(k))\times\sum_{i=1}^N \|v^{[i]}_x(k) - \hat{x}(k)\|\nnum\\
  & + G^+ \sum_{i=1}^N \|v^{[i]}_{\mu}(k) - \hat{\mu}(k)\| + H\sum_{i=1}^N \|v^{[i]}_{\lambda}(k) -
  \hat{\lambda}(k)\|.
\end{align}
Then it follows from (b) in Lemma~\ref{lem6} that
$\{\alpha(k)\|\pi(k)\|\}$ is summable. Notice that $\|\hat{\HH}(k) -
\sum_{i=1}^N \hat{\HH}^{[i]}(k)\| \leq
\frac{\sum_{\ell=0}^{k-1}\alpha(\ell)\|\pi(\ell)\|}{s(k-1)}$, and thus
$\displaystyle{\lim_{k\rightarrow+\infty}\|\hat{\HH}(k) - \sum_{i=1}^N
  \hat{\HH}^{[i]}(k)\| = 0}$. The desired result immediately follows from
Claim 2.
\end{proof}

\noindent\textbf{Claim 4:} The limit point $\tilde{x}$ in
Lemma~\ref{lem12} is a primal optimal solution.

\begin{proof}
  Let $\hat{\mu}(k) =
  (\hat{\mu}_1(k),\cdots,\hat{\mu}_m(k))^T\in{\real}_{\geq0}^m$. By
  the balanced communication assumption~\ref{asm3}, we obtain
  \begin{align*}
    \sum_{i=1}^N \mu^{[i]}(k+1) &= \sum_{i=1}^N \sum_{j=1}^N a^i_j(k) \mu^{[j]}(k) + \alpha(k)\sum_{i=1}^N [g(v^{[i]}_x(k))]^+\nnum\\
    &= \sum_{j=1}^N \mu^{[j]}(k) + \alpha(k)\sum_{i=1}^N [g(v^{[i]}_x(k))]^+.
  \end{align*}
  This implies that the sequence $\{\hat{\mu}_{\ell}(k)\}$ is
  non-decreasing in ${\real}_{\geq0}$. Observe that
  $\{\hat{\mu}_{\ell}(k)\}$ is lower bounded by zero. In this way, we
  distinguish the following two cases:

  \underline{Case 1:} The sequence $\{\hat{\mu}_{\ell}(k)\}$ is upper
  bounded. Then $\{\hat{\mu}_{\ell}(k)\}$ is convergent in
  ${\real}_{\geq0}$. Recall that
  $\displaystyle{\lim_{k\rightarrow+\infty}\|\mu^{[i]}(k)-\mu^{[j]}(k)\| = 0}$
  for all $i,j\in V$. This implies that there exists
  $\mu_{\ell}^*\in{\real}_{\geq0}$ such that
  $\displaystyle{\lim_{k\rightarrow+\infty}\|\mu^{[i]}_{\ell}(k)-\mu_{\ell}^*\|
    = 0}$ for all $i\in V$. Observe that $\sum_{i=1}^N \mu^{[i]}(k+1) =
  \sum_{i=1}^N \mu^{[i]}(0) + \sum_{\tau=0}^k
  \alpha(\tau)\sum_{i=1}^N[g(v^{[i]}_x(\tau))]^+$. Thus, we have
  $\sum_{k=0}^{+\infty} \alpha(k)\sum_{i=1}^N[g_{\ell}(v^{[i]}_x(k))]^+ <
  +\infty$, implying that
  $\liminf_{k\rightarrow+\infty}[g_{\ell}(v^{[i]}_x(k))]^+=0$. Since
  $\displaystyle{\lim_{k\rightarrow+\infty}\|x^{[i]}(k)-\tilde{x}\| = 0}$
  for all $i\in V$, then
  $\displaystyle{\lim_{k\rightarrow+\infty}\|v_x^{[i]}(k)-\tilde{x}\| =
    0}$, and thus $[g_{\ell}(\tilde{x})]^+ = 0$.

  \underline{Case 2:} The sequence $\{\hat{\mu}_{\ell}(k)\}$ is not
  upper bounded. Since $\{\hat{\mu}_{\ell}(k)\}$ is non-decreasing,
  then $\hat{\mu}_{\ell}(k)\rightarrow +\infty$. It follows from Claim
  3 and (a) in Lemma~\ref{lem8} that it is impossible that
  ${\HH}(\hat{x}(k),\hat{\mu}(k),\hat{\lambda}(k))\rightarrow+\infty$. Assume
  that $[g_{\ell}(\tilde{x})]^+ > 0$. Then we have
  \begin{align}
    {\HH}(\hat{x}(k),\hat{\mu}(k),\hat{\lambda}(k))
    &= f(\hat{x}(k)) + N \hat{\mu}(k)^T [g(\hat{x}(k))]^+ + N \lambda(k)^T |h(\hat{x}(k))|\nnum\\
    & \geq f(\hat{x}(k)) +
    \hat{\mu}_{\ell}(k)[g_{\ell}(\hat{x}(k))]^+.\label{e58}
  \end{align}

  Taking limits on both sides of~\eqref{e58} and we obtain:
\begin{align*}
  \liminf_{k\rightarrow+\infty}{\HH}(\hat{x}(k),\hat{\mu}(k),\hat{\lambda}(k))
  \geq \limsup_{k\rightarrow+\infty}(f(\hat{x}(k)) +
  \hat{\mu}_{\ell}(k)[g_{\ell}(\hat{x}(k))]^+) = +\infty.
\end{align*}
Then we reach a contradiction, implying that $[g_{\ell}(\tilde{x})]^+
= 0$.

In both cases, we have $[g_{\ell}(\tilde{x})]^+ = 0$ for any $1\leq
\ell\leq m$. By utilizing similar arguments, we can further prove that
$|h(\tilde{x})| = 0$. Since $\tilde{x}\in X$, then $\tilde{x}$ is
feasible and thus $f(\tilde{x})\geq p^*$. On the other hand, since
$\frac{\sum_{\ell=0}^{k-1}\alpha(\ell)
  \hat{x}(\ell)}{\sum_{\ell=0}^{k-1}\alpha(\ell)}$ is a convex
combination of $\hat{x}(0),\cdots,\hat{x}(k-1)$ and
$\displaystyle{\lim_{k\rightarrow+\infty}\hat{x}(k) = \tilde{x}}$,
then Claim 3 and (b) in Lemma~\ref{lem8} implies that
\begin{align*}
  p^* = \lim_{k\rightarrow+\infty}\hat{\HH}(k) =
  \lim_{k\rightarrow+\infty}
  \frac{\sum_{\ell=0}^{k-1}\alpha(\ell){\HH}(\hat{x}(\ell),\hat{\mu}(\ell),\hat{\lambda}(\ell))}{\sum_{\ell=0}^{k-1}\alpha(\ell)}
  \geq
  \lim_{k\rightarrow+\infty}f(\frac{\sum_{\ell=0}^{k-1}\alpha(\ell)
    \hat{x}(\ell)}{\sum_{\ell=0}^{k-1}\alpha(\ell)}) = f(\tilde{x}).
\end{align*}
Hence, we have $f(\tilde{x}) = p^*$ and thus $\tilde{x}\in X^*$.
\end{proof}

\noindent\textbf{Claim 5:} It holds that
$\displaystyle{\lim_{k\rightarrow+\infty}\|y^{[i]}(k) - p^*\| = 0}$.

\begin{proof}
  The proof follows the same lines in Claim 2 of Theorem~\ref{the1}
  and thus omitted here.
\end{proof}

\section{Discussion}\label{sec:discussion}

In this section, we present some possible extensions and interesting
special cases.

\subsection{Discussion on the periodic strong connectivity assumption in Theorem~\ref{the1}}

In the case that $\GG(k)$ is undirected, then the periodic strong connectivity
assumption~\ref{asm1} in Theorem~\ref{the1} can be weakened into:
\begin{assumption} [Eventual strong connectivity]
  The undirected graph $(V, \cup_{k\geq s} E(k))$ is connected for all
  time instant $s\geq0$.\label{asm7}
\end{assumption}

If $\GG(k)$ is undirected, the periodic connectivity
assumption~\ref{asm1} in Theorem~\ref{the1} can also be replaced with the assumption in
Proposition 2 of~\cite{LM:05}; i.e., for any time instant $k\geq0$,
there is an agent connected to all other agents in the undirected
graph $(V, \cup_{k\geq s} E(k))$.

\subsection{A generalized step-size scheme}

The step-size scheme in the DLPDS algorithm can be slightly
generalized the case that the maximum deviation of step-sizes between agents at each time is not large. It is formally stated as follows: $\displaystyle{\lim_{k\rightarrow+\infty}\alpha^{[i]}(k) =
0}$, $\sum_{k=0}^{+\infty}\alpha^{[i]}(k) = +\infty,\quad
\sum_{k=0}^{+\infty}\alpha^{[i]}(k)^2 < +\infty,\quad \min_{i\in
  V}\alpha^{[i]}(k) \geq C_{\alpha}\max_{i\in V}\alpha^{[i]}(k),$ where
$\alpha^{[i]}(k)$ is the step-size of agent~$i$ at time $k$ and
$C_{\alpha}\in (0,1]$.

\subsection{Discussion on the Slater's condition in Theorem~\ref{the3}}

If $g_{\ell}$ ($1\leq \ell \leq m$) is linear, then the Slater's
condition~\ref{asm5} can be weakened to the following: there exists a
relative interior point $\bar{x}$ of $X$ such that $h(\bar{x}) = 0$
and $g(\bar{x}) \leq 0$. For this case, the strong duality and the
non-emptyness of the penalty dual optimal set can be ensured by
replacing Proposition 5.3.5~\cite{DPB:09} with Proposition
5.3.4~\cite{DPB:09} in the proofs of Lemma~\ref{lem1}. In this way,
the convergence results of the DPPDS algorithm still hold for the case
of linear $g_{\ell}$.

\subsection{The special case in the absence of inequality and equality
  constraints}\label{specialcase}
The following special case of problem~\eqref{e2} is studied in~\cite{AN-AO-PAP:08}:
\begin{align}
  & \min_{x\in {\real}^n } \sum_{i=1}^N f^{[i]}(x),\quad {\rm s.t.} \quad
  x\in \cap_{i=1}^N X^{[i]}. \label{e42}
\end{align}
In order to solve problem~\eqref{e42}, we consider the following \emph{Distributed Primal Subgradient} Algorithm which is a special case of the DLPDS algorithm:
\begin{align*}
  x^{[i]}(k+1) = P_{X^{[i]}}[v^{[i]}_x(k) - \alpha(k){\DD}f^{[i]}(v^{[i]}_x(k))].
\end{align*}

\begin{corollary}[Convergence properties of the distributed primal subgradient algorithm] Consider problem~\eqref{e42}, and let the non-degeneracy
  assumption~\ref{asm2}, the balanced communication
  assumption~\ref{asm3} and the periodic strong
  connectivity assumption~\ref{asm1} hold. Consider the sequence $\{x^{[i]}(k)\}$ of
  the distributed primal subgradient algorithm with initial states
  $x^{[i]}(0)\in X^{[i]}$ and the step-sizes satisfying
  $\displaystyle{\lim_{k\rightarrow+\infty}\alpha(k)= 0}$,
  $\displaystyle{\sum_{k=0}^{+\infty}\alpha(k) = +\infty}$, and
  $\displaystyle{\sum_{k=0}^{+\infty}\alpha(k)^2 < +\infty}$. Then
  there exists an optimal solution $x^*$ such that $\displaystyle{\lim_{k\rightarrow+\infty}\|x^{[i]}(k)-x^*\| =
    0}$ all $i\in V$.
\label{cor1}
\end{corollary}
\begin{proof}
  The result is an immediate consequence of Theorem~\ref{the1} with
  $g(x) \equiv0$.
\end{proof}

\section{Numerical examples}\label{sec:example}

In this section, we illustrate the performance of the DLPDS and DPPDS algorithms via two numerical examples.

\subsection{A numerical example of NUM for the DLPDS algorithm}

In order to study the performance of the DLPDS algorithm, we here consider a numerical example of network utility maximization, e.g., in~\cite{KPK-AM-DT:98}. Consider five agents and one link where each agent sends data through the link at a rate of $z_i$, and the link capacity is $5$. The global decision vector $x:=[z_1 \cdots z_5]^T$ is the resource allocation vector. Each agent $i$ is associated a concave utility function $f^{[i]}(z_i) := \sqrt{z_i}$, representing the utility agent $i$ obtains through sending data at a rate of $z_i$. Agents aim to maximize the aggregate sum of local utilities and this problem can be formulated as follows:
\begin{align}\min_{x\in\real^5}\sum_{i\in V}-\sqrt{z_i}
\quad{\rm s.t.}\quad z_1 + z_2 + z_3 + z_4 + z_5 \leq 5,\quad x \in \cap_{i\in V}X^{[i]},\label{e33}
\end{align} where local constraint sets $X^{[i]}$ are given by:
\begin{align*} &X^{[1]} := [0.5,\;\; 5.5]\times [0.5,\;\; 5.5]\times [0.5,\;\; 5.5]\times [0.5,\;\; 5.5]\times [0.5,\;\; 5.5],\\
&X^{[2]} := [0.55,\;\; 5.25]\times [0.55,\;\; 5.25]\times [0.55,\;\; 5.25]\times [0.55,\;\; 5.25]\times [0.55,\;\; 5.25],\\
&X^{[3]} := [0.5,\;\; 6]\times [0.5,\;\; 6]\times [0.5,\;\; 6]\times [0.5,\;\; 6]\times [0.5,\;\; 6],\\
&X^{[4]} := [0.5,\;\; 5]\times [0.5,\;\; 5]\times [0.5,\;\; 5]\times [0.5,\;\; 5]\times [0.5,\;\; 5],\\
&X^{[5]} := [0.525,\;\; 5.75]\times [0.525,\;\; 5.75]\times [0.525,\;\; 5.75]\times [0.525,\;\; 5.75]\times [0.525,\;\; 5.75].\\
\end{align*} We use the DLPDS algorithm to solve problem~\eqref{e33} by choosing step-size $\alpha(k) = \frac{1}{k+1}$.
Figures~\ref{fig1} to~\ref{fig5} shows the simulation results of the DLPDS algorithm in comparison with the centralized subgradient algorithm. It demonstrates that all the agents takes $10^4$ iterates to agree upon the optimal solution $[1\;\;1\;\;1\;\;1\;\;1]^T$. Furthermore, it can be observed that the optimal solution can be found by the centralized subgradient algorithm with the same step-size after 200 iterates which is much less than that of the DLPDS algorithm.

\subsection{A numerical example for the DPPDS algorithm}

Consider a network with five agents and their objective functions are defined as \begin{align*}&f^{[1]}(x) := \frac{1}{5}\big((a-5)^2+(b-2.5)^2+(c-5)^2+(d+2.5)^2+(e+5)^2\big),\\
&f^{[2]}(x) := \frac{1}{5}\big((a-2.5)^2+(b-5)^2+(c+2.5)^2+(d+5)^2+(e-5)^2\big),\\
&f^{[3]}(x) := \frac{1}{5}\big((a-5)^2+(b+2.5)^2+(c+5)^2+(d-5)^2+(e-2.5)^2\big),\\
&f^{[4]}(x) := \frac{1}{5}\big((a+2.5)^2+(b+5)^2+(c-5)^2+(d-2.5)^2+(e-5)^2\big),\\
&f^{[5]}(x) := \frac{1}{5}\big((a+5)^2+(b-5)^2+(c-2.5)^2+(d-5)^2+(e+2.5)^2\big),\end{align*} where the global decision vector $x := [a\;\; b\;\; c\;\; d\;\; e]^T\in\real^5$. The global equality constraint function is given by $h(x) := a + b + c + d + e - 5$, and the global constraint set is as follows:
$X := [-5\;\; 5]\times [-5\;\; 5]\times [-5\;\; 5]\times [-5\;\; 5]\times [-5\;\; 5]$. Consider the optimization problem as follows: \begin{align*}\min_{x\in\real^5}\sum_{i\in V}f^{[i]}(x),\quad {\rm s.t.}\quad h(x) = 0,\quad x\in X.\end{align*}
We employ the DPPDS algorithm to solve the above optimization problem with
the step-size $\alpha(k) = \frac{1}{k+1}$. Its simulation results are included in Figures~\ref{fig6} to~\ref{fig10} in comparison with the performance of the centralized subgradient algorithm. Observe that all the agents asymptotically achieve the optimal solution $[1\;\;1\;\;1\;\;1\;\;1]^T$. Like the DLPDS algorithm, convergence rate of the DPPDS algorithm is slower than the centralized algorithm.

\section{Conclusion}\label{sec:conclusion}

We have studied a multi-agent optimization problem where the agents
aim to minimize a sum of local objective functions
subject to a global inequality constraint, a global equality
constraint and a global constraint set defined as the intersection of
local constraint sets. We have considered two cases: the first one
in the absence of the equality constraint and the second one with
identical local constraint sets. To address these cases, we have
introduced two distributed subgradient algorithms which are based on
Lagrangian and penalty primal-dual methods, respectively. These two
algorithms were shown to asymptotically converge to primal solutions
and optimal values. Two numerical examples were presented to demonstrate the performance our algorithms. Our future work includes explicit characterization of convergence rates of the algorithms in this paper.

\section{Appendix}\label{sec:appendix}

\subsection{Dynamic average consensus algorithms}
The following is the vector version of the first-order dynamic average
consensus algorithm proposed in~\cite{MZ-SM:08a} with $x^{[i]}(k), \xi^{[i]}(k)\in{\real}^n$: \begin{align}
  x^{[i]}(k+1) = \sum_{j=1}^N a^i_j(k)x^{[j]}(k) + \xi^{[i]}(k).\label{e7}
\end{align}

\begin{proposition}
  Denote $\Delta \xi_{\ell}(k) :=
\max_{i\in V}\xi^{[i]}_{\ell}(k) - \min_{i\in V}\xi^{[i]}_{\ell}(k)$ for
$1\leq \ell \leq n$. Let the non-degeneracy assumption~\ref{asm2}, the balanced communication
  assumption~\ref{asm3} and the periodic strong connectivity assumption~\ref{asm1} hold. Assume that
  $\displaystyle{\lim_{k\rightarrow+\infty}\Delta \xi_{\ell}(k) = 0}$
  for all $1\leq \ell\leq n$ and all $k\geq0$. Then
  $\displaystyle{\lim_{k\rightarrow+\infty}\|x^{[i]}(k) - x^{[j]}(k)\| = 0}$
  for all $i, j\in V$.\label{pro1}
\end{proposition}

\subsection{A property of projection operators}

The proof of the following lemma can be found in~\cite{DPB:09},~\cite{DPB-AN-AO:03a}
and~\cite{AN-AO-PAP:08}.

\begin{lemma}
  Let $Z$ be a non-empty, closed and convex set in ${\real}^n$. For
  any $z\in{\real}^n$, the following holds for any $y\in Z$: $\|
  P_Z[z] - y \|^2 \leq \| z - y \|^2 - \| P_Z[z] - z
  \|^2$. \label{lem5}
\end{lemma}

\subsection{Some properties of the distributed projected subgradient
  algorithm in~\cite{AN-AO-PAP:08}}

Consider the following distributed projected subgradient algorithm
proposed in~\cite{AN-AO-PAP:08}: $x^{[i]}(k+1) = P_{Z}[v_x^{[i]}(k) - \alpha(k) d^{[i]}(k)]$.
Denote by $e^{[i]}(k) := P_{Z}[v^{[i]}_x(k) - \alpha(k) d^{[i]}(k)] -
v^{[i]}_x(k)$. The following is a slight modification of Lemma 8 and its
proof in~\cite{AN-AO-PAP:08}.

\begin{lemma}
  Let the non-degeneracy assumption~\ref{asm2}, the balanced communication
  assumption~\ref{asm3} and the periodic strong connectivity assumption~\ref{asm1} hold. Suppose $Z\in{\real}^n$ is a closed and convex set. Then there exist $\gamma>0$ and $\beta\in(0,1)$ such that
  \begin{align}
    \|x^{[i]}(k) - \hat{x}(k)\|\leq
    N\gamma\sum_{\tau=0}^{k-1}\beta^{k-\tau}\{\alpha(\tau)\|d^{[i]}(\tau)\| + \|e^{[i]}(\tau) +
    \alpha(\tau)d^{[i]}(\tau)\|\} + N \gamma \beta^{k-1}\sum_{i=0}^N\|x^{[i]}(0)\|.\nnum
  \end{align}
  Suppose $\{d^{[i]}(k)\}$ is uniformly bounded for each $i\in V$, and
  $\sum_{k=0}^{+\infty}\alpha(k)^2 < +\infty$, then we have
  ${\sum_{k=0}^{+\infty}\alpha(k)\max_{i\in V}\|x^{[i]}(k) - \hat{x}(k)\|
    < +\infty}$.\label{lem4}
\end{lemma}


\begin{thebibliography}{10}

\bibitem{KJA-LH-HU:58}
K.J. Arrow, L.~Hurwicz, and H.~Uzawa.
\newblock {\em Studies in linear and nonlinear programming}.
\newblock Stanford University Press, 1958.

\bibitem{DPB-JNT:97}
D.~P. Bertsekas and J.~N. Tsitsiklis.
\newblock {\em Parallel and Distributed Computation: Numerical Methods}.
\newblock Athena Scientific, 1997.

\bibitem{DPB:09}
D.P. Bertsekas.
\newblock {\em Convex optimization theory}.
\newblock Anthena Scietific, 2009.

\bibitem{DPB-AN-AO:03a}
D.P. Bertsekas, A.~Nedic, and A.~Ozdaglar.
\newblock {\em Convex analysis and optimization}.
\newblock Anthena Scietific, 2003.

\bibitem{VDB-JMH-AO-JNT:05}
V.~D. Blondel, J.~M. Hendrickx, A.~Olshevsky, and J.~N. Tsitsiklis.
\newblock Convergence in multiagent coordination, consensus, and flocking.
\newblock In {\em {IEEE} Conf. on Decision and Control and European Control
  Conference}, pages 2996--3000, Seville, Spain, December 2005.

\bibitem{SB-AG-BP-DS:06}
S.~Boyd, A.~Ghosh, B.~Prabhakar, and D.~Shah.
\newblock Randomized gossip algorithms.
\newblock {\em IEEE Transactions on Information Theory}, 52(6):2508--2530,
  2006.

\bibitem{Cortes:06}
J.~Cort\'{e}s.
\newblock Analysis and design of distributed algorithms for $\chi$-consensus.
\newblock In {\em {IEEE} Conf. on Decision and Control}, pages 3363--3368, San
  Diego, USA, December 2006.

\bibitem{MCD-AJ:06}
M.~C. DeGennaro and A.~Jadbabaie.
\newblock Decentralized control of connectivity for multi-agent systems.
\newblock In {\em {IEEE} Conf. on Decision and Control}, pages 3947--3952, San
  Diego, USA, Dec 2006.

\bibitem{JD-JS:07}
J.~Derenick and J.~Spletzer.
\newblock Convex optimization strategies for coordinating large-scale robot
  formations.
\newblock {\em IEEE Transactions on Robotics}, 23(6):1252--1259, 2007.

\bibitem{JD-JS-MAH:09}
J.~Derenick, J.~Spletzer, and M.~Ani Hsieh.
\newblock An optimal approach to collaborative target tracking with performance
  guarantees.
\newblock {\em Journal of Intelligent and Robotic Systems}, 56(1-2):47--67,
  2009.

\bibitem{JAF-RMM:04}
J.~A. Fax and R.~M. Murray.
\newblock Information flow and cooperative control of vehicle formations.
\newblock {\em IEEE Transactions on Automatic Control}, 49(9):1465--1476, 2004.

\bibitem{AJ-JL-ASM:02}
A.~Jadbabaie, J.~Lin, and A.~S. Morse.
\newblock Coordination of groups of mobile autonomous agents using nearest
  neighbor rules.
\newblock {\em IEEE Transactions on Automatic Control}, 48(6):988--1001, 2003.

\bibitem{BJ-TK-MJ-KHJ:08}
B.~Johansson, T.~Keviczky, M.~Johansson, and K.~H. Johansson.
\newblock Subgradient methods and consensus algorithms for solving convex
  optimization problems.
\newblock In {\em {IEEE} Conf. on Decision and Control}, pages 4185--4190,
  Cancun, Mexico, December 2008.

\bibitem{AK-TB-RS:07}
A.~Kashyap, T.~Ba{\c s}ar, and R.~Srikant.
\newblock Quantized consensus.
\newblock {\em Automatica}, 43(7):1192--1203, 2007.

\bibitem{KPK-AM-DT:98}
F.~P. Kelly, A.~Maulloo, and D.~Tan.
\newblock Rate control in communication networks: {S}hadow prices, proportional
  fairness and stability.
\newblock {\em Journal of the Operational Research Society}, 49(3):237--252,
  1998.

\bibitem{PM-ME:08}
P.~Martin and M.~Egerstedt.
\newblock Optimization of multi-agent motion programs with applications to
  robotic marionettes.
\newblock In {\em Hybrid Systems: Computation and Control}, April 2009.

\bibitem{MM-DS-JP-SHL-RMM:07a}
M.~Mehyar, D.~Spanos, J.~Pongsajapan, S.~H. Low, and R.~M. Murray.
\newblock Asynchronous distributed averaging on communication networks.
\newblock {\em IEEE/ACM Transactions on Networking}, 15(3):512--520, 2007.

\bibitem{LM:05}
L.~Moreau.
\newblock Stability of multiagent systems with time-dependent communication
  links.
\newblock {\em IEEE Transactions on Automatic Control}, 50(2):169--182, 2005.

\bibitem{AIM-SIR:05}
A.~I. Mourikis and S.~I. Roumeliotis.
\newblock Optimal sensing strategies for mobile robot formations:
  Resource-constrained localization.
\newblock In {\em Proceedings of Robotics: Science and Systems}, pages
  281--288, Cambridge, USA, 2005.

\bibitem{AN-AO:08b}
A.~Nedic and A.~Ozdaglar.
\newblock Approximate primal solutions and rate analysis for dual subgradient
  methods.
\newblock {\em SIAM Journal on Optimization}, 19(4):1757--1780, 2009.

\bibitem{AN-AO:09}
A.~Nedic and A.~Ozdaglar.
\newblock Distributed subgradient methods for multi-agent optimization.
\newblock {\em IEEE Transactions on Automatic Control}, 54(1):48--61, 2009.

\bibitem{AN-AO:08a}
A.~Nedic and A.~Ozdaglar.
\newblock Subgradient methods for saddle-point problems.
\newblock {\em Journal of Optimization Theory and Applications},
  142(1):205--228, 2009.

\bibitem{AN-AO-PAP:08}
A.~Nedic, A.~Ozdaglar, and P.A. Parrilo.
\newblock Constrained consensus and optimization in multi-agent networks.
\newblock {\em IEEE Transactions on Automatic Control}, 55(4):922--938, 2010.

\bibitem{RDN:03}
R.~D. Nowak.
\newblock Distributed {EM} algorithms for density estimation and clustering in
  sensor networks.
\newblock {\em IEEE Transactions on Signal Processing}, 51:2245--2253, 2003.

\bibitem{ROS-RMM:03c}
R.~Olfati-Saber and R.~M. Murray.
\newblock Consensus problems in networks of agents with switching topology and
  time-delays.
\newblock {\em IEEE Transactions on Automatic Control}, 49(9):1520--1533, 2004.

\bibitem{AO-SS-LN:03}
A.~Oliveira, S.~Soares, and L.~Nepomuceno.
\newblock Optimal active power dispatch combining network flow and interior
  point approaches.
\newblock {\em IEEE Transactions on Power Systems}, 18(4):1235--1240, 2003.

\bibitem{AO-JNT:07}
A.~Olshevsky and J.~N. Tsitsiklis.
\newblock Convergence speed in distributed consensus and averaging.
\newblock {\em SIAM Journal on Control and Optimization}, 48(1):33--55, 2009.

\bibitem{MGR-RDN:04}
M.~G. Rabbat and R.~D. Nowak.
\newblock Decentralized source localization and tracking.
\newblock In {\em IEEE Int. Conf. on Acoustics, Speech and Signal Processing},
  pages 921--924, May 2004.

\bibitem{AR:08}
A.~Rantzer.
\newblock Using game theory for distributed control engineering.
\newblock In {\em Games 2008, 3rd World Congress of the Game Theory Society},
  2008.

\bibitem{SSR-AN-VVV:08b}
S.~{Sundhar~Ram}, A.~Nedic, and V.~V. Veeravalli.
\newblock Distributed and recursive parameter estimation in parametrized linear
  state-space models.
\newblock {\em IEEE Transactions on Automatic Control}, 55(2):488--492, 2010.

\bibitem{ATS-AJ:08}
A.~Tahbaz-Salehi and A.~Jadbabaie.
\newblock Consensus over random networks.
\newblock {\em IEEE Transactions on Automatic Control}, 53(3):791--795, 2008.

\bibitem{JNT:84}
J.~N. Tsitsiklis.
\newblock {\em Problems in Decentralized Decision Making and Computation}.
\newblock PhD thesis, Massachusetts Institute of Technology, November 1984.
\newblock Available at
  \texttt{http://web.mit.edu/jnt/www/Papers/PhD-84-jnt.pdf}.

\bibitem{HW-HS-JK-PY:98}
H.~Wei, H.~Sasaki, J.~Kubokawa, and R.~Yokoyama.
\newblock An interior point nonlinear programming for optimal power flow
  problems with a novel data structure.
\newblock {\em IEEE Transactions on Power Systems}, 13:870--877, 1998.

\bibitem{LX-SB:04}
L.~Xiao and S.~Boyd.
\newblock Fast linear iterations for distributed averaging.
\newblock {\em Systems \& Control Letters}, 53:65--78, 2004.

\bibitem{MZ-SM:08a}
M.~Zhu and S.~Mart{\'\i}nez.
\newblock Discrete-time dynamic average consensus.
\newblock {\em Automatica}, 46(2):322--329, 2010.

\end{thebibliography}

\begin{figure}[ht]
  \centerline{\epsfxsize=3.5in \epsffile{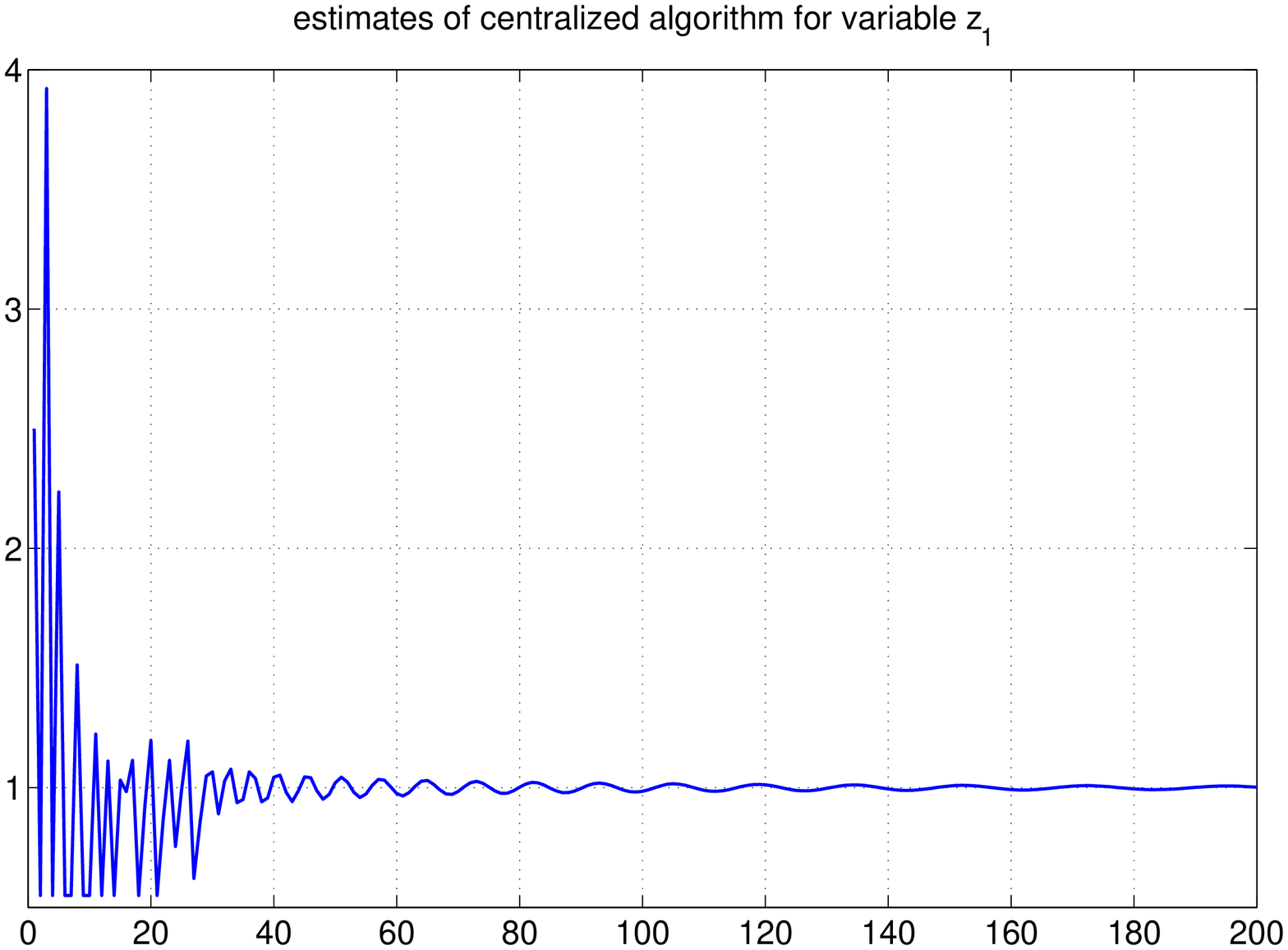} \epsfxsize=3.5in \epsffile{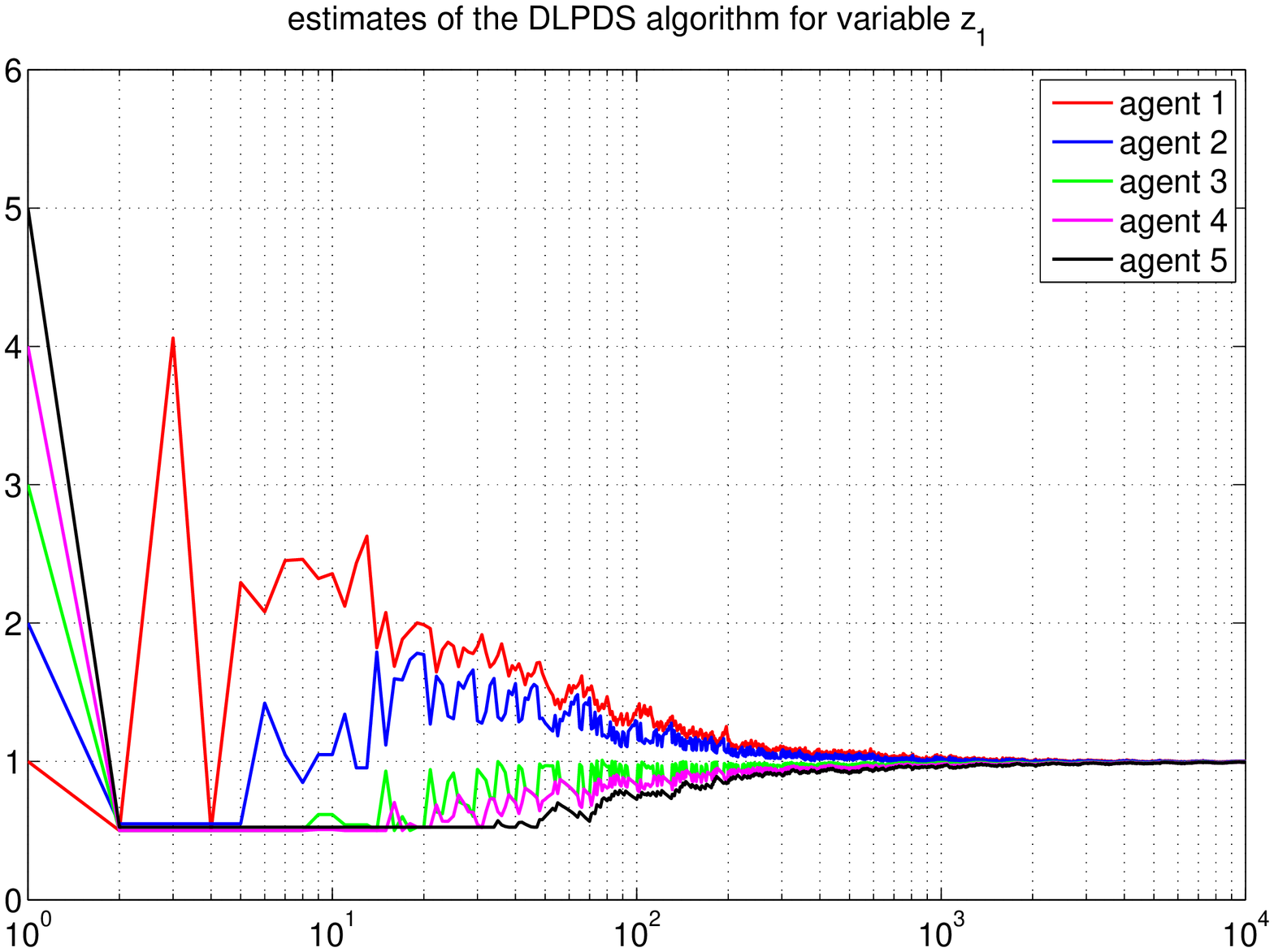}}
  \caption{Estimates of variable $z_1$ of centralized algorithm and the DLPDS algorithm}\label{fig1}
\end{figure}

\begin{figure}[ht]
  \centerline{\epsfxsize=3.5in \epsffile{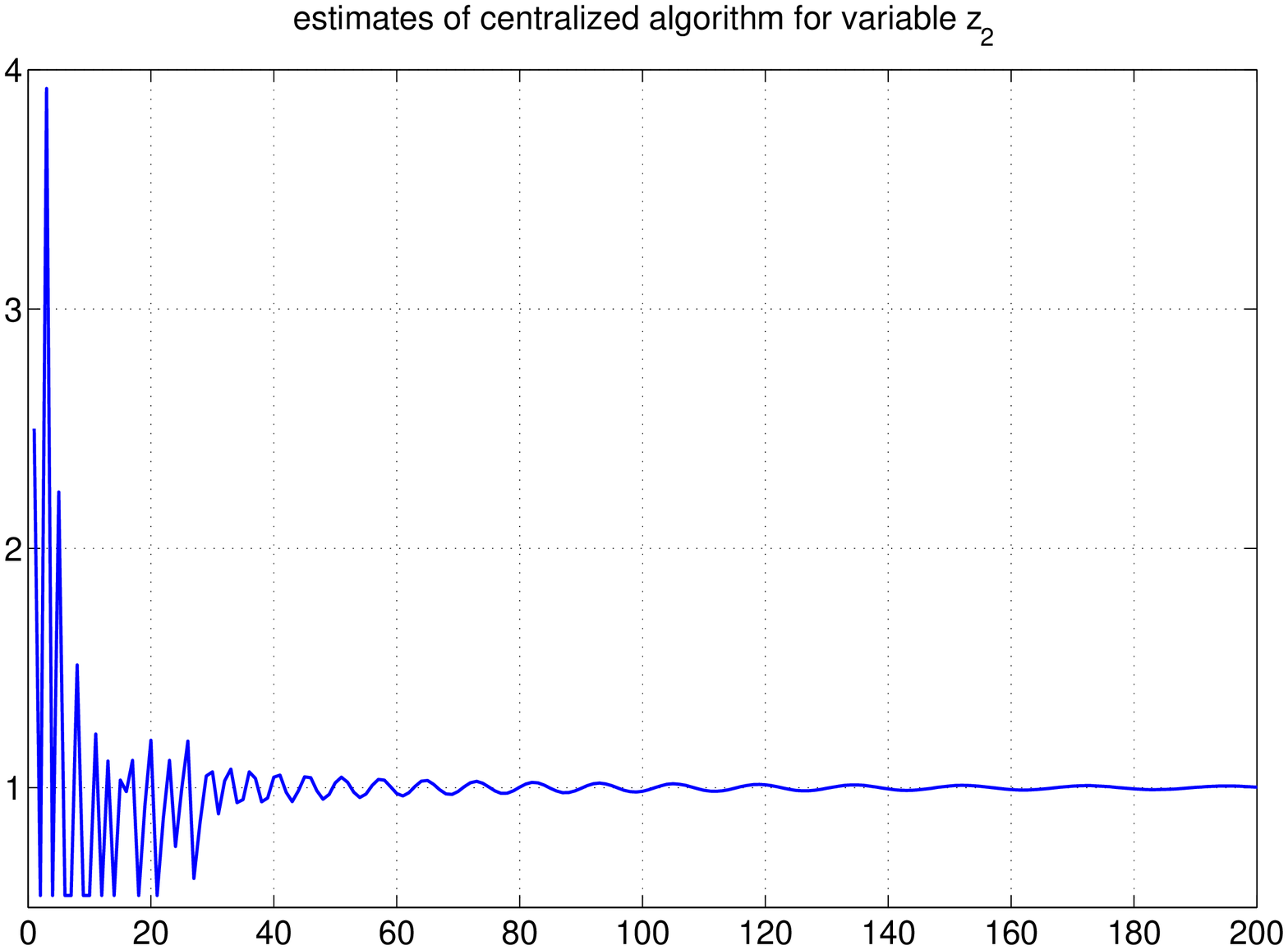} \epsfxsize=3.5in \epsffile{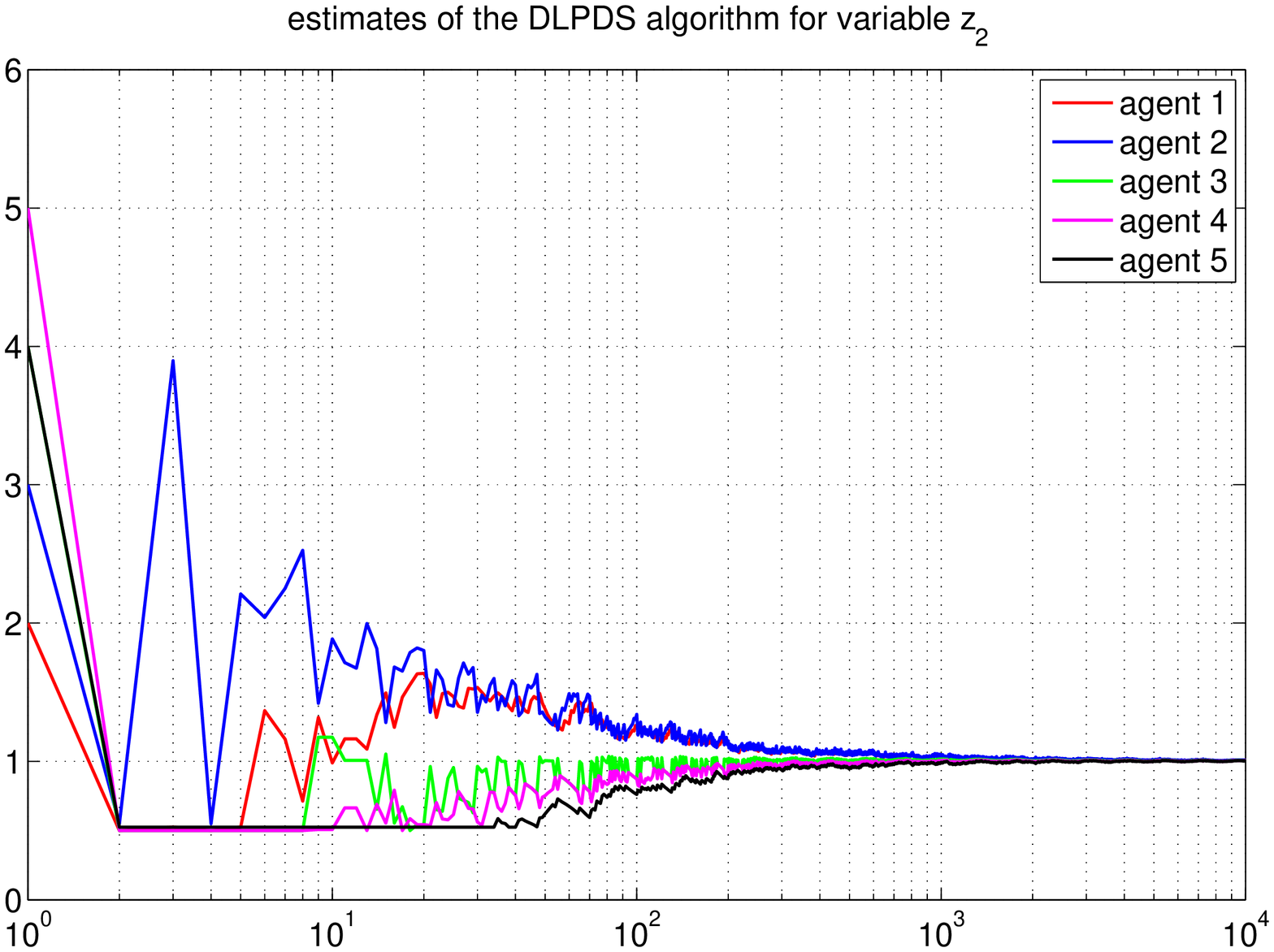}}
  \caption{Estimates of variable $z_2$ of centralized algorithm and the DLPDS algorithm}\label{fig2}
\end{figure}

\begin{figure}[ht]
  \centerline{\epsfxsize=3.5in \epsffile{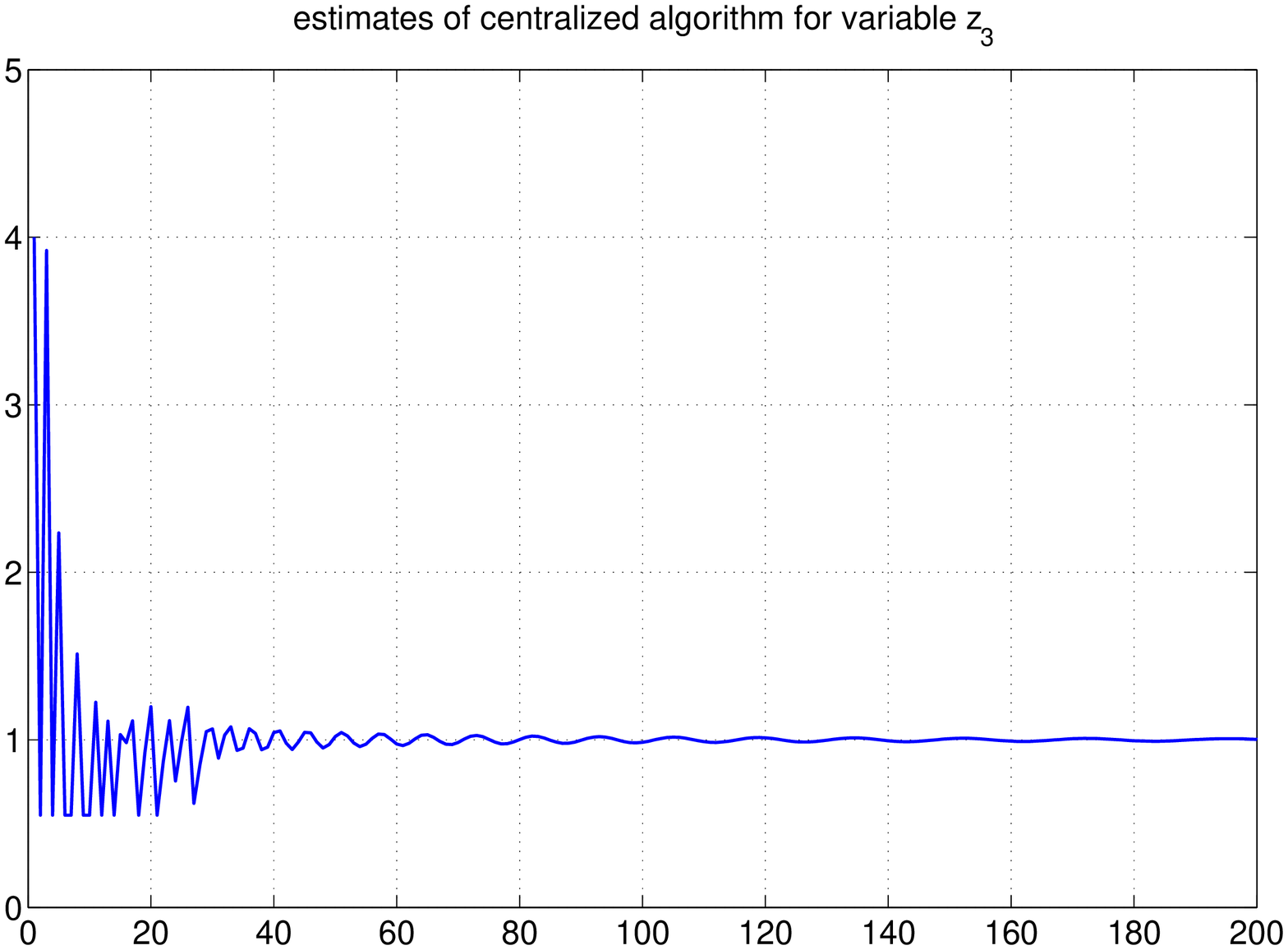} \epsfxsize=3.5in \epsffile{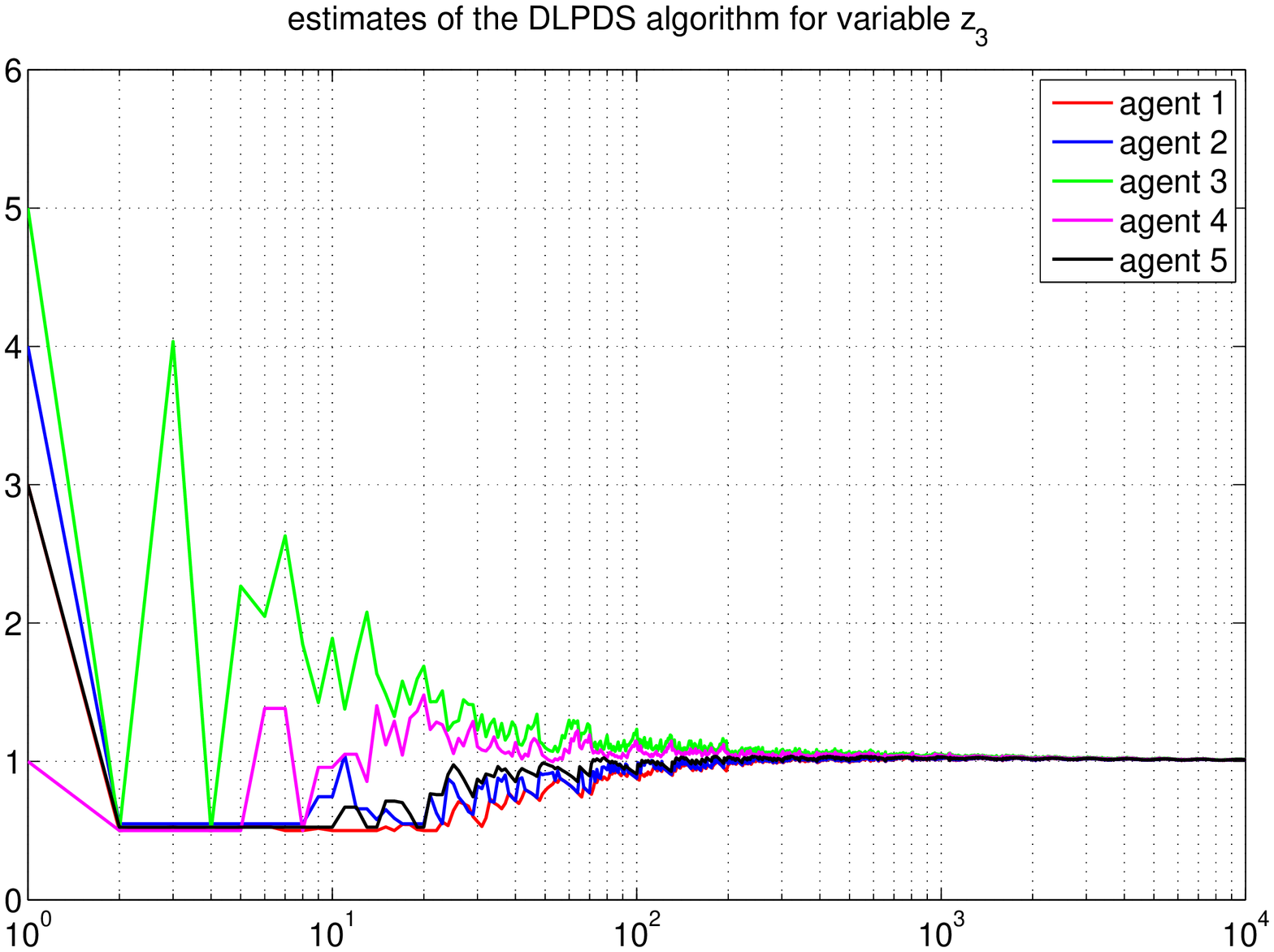}}
  \caption{Estimates of variable $z_3$ of centralized algorithm and the DLPDS algorithm}\label{fig3}
\end{figure}

\begin{figure}[ht]
  \centerline{\epsfxsize=3.5in \epsffile{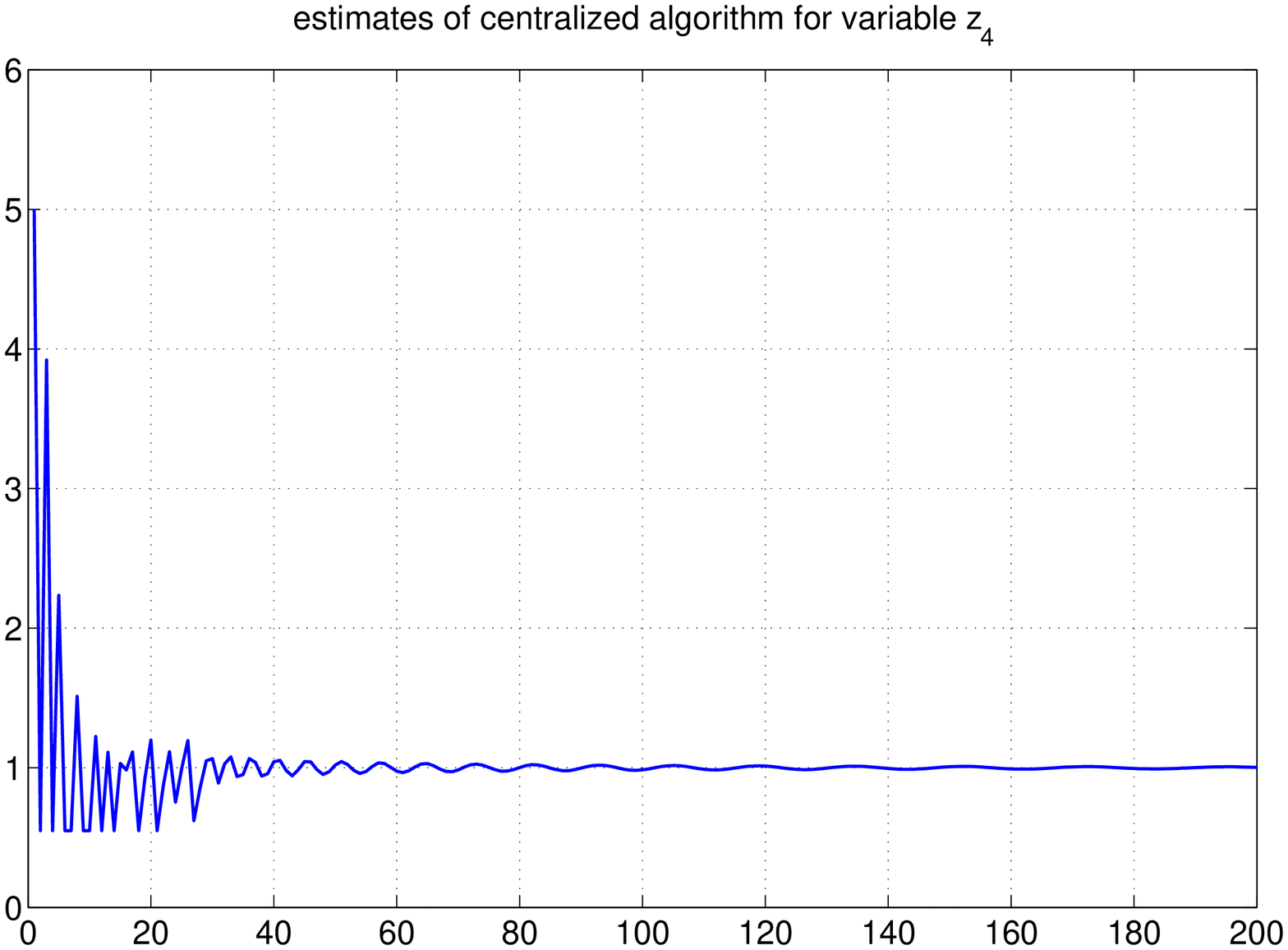} \epsfxsize=3.5in \epsffile{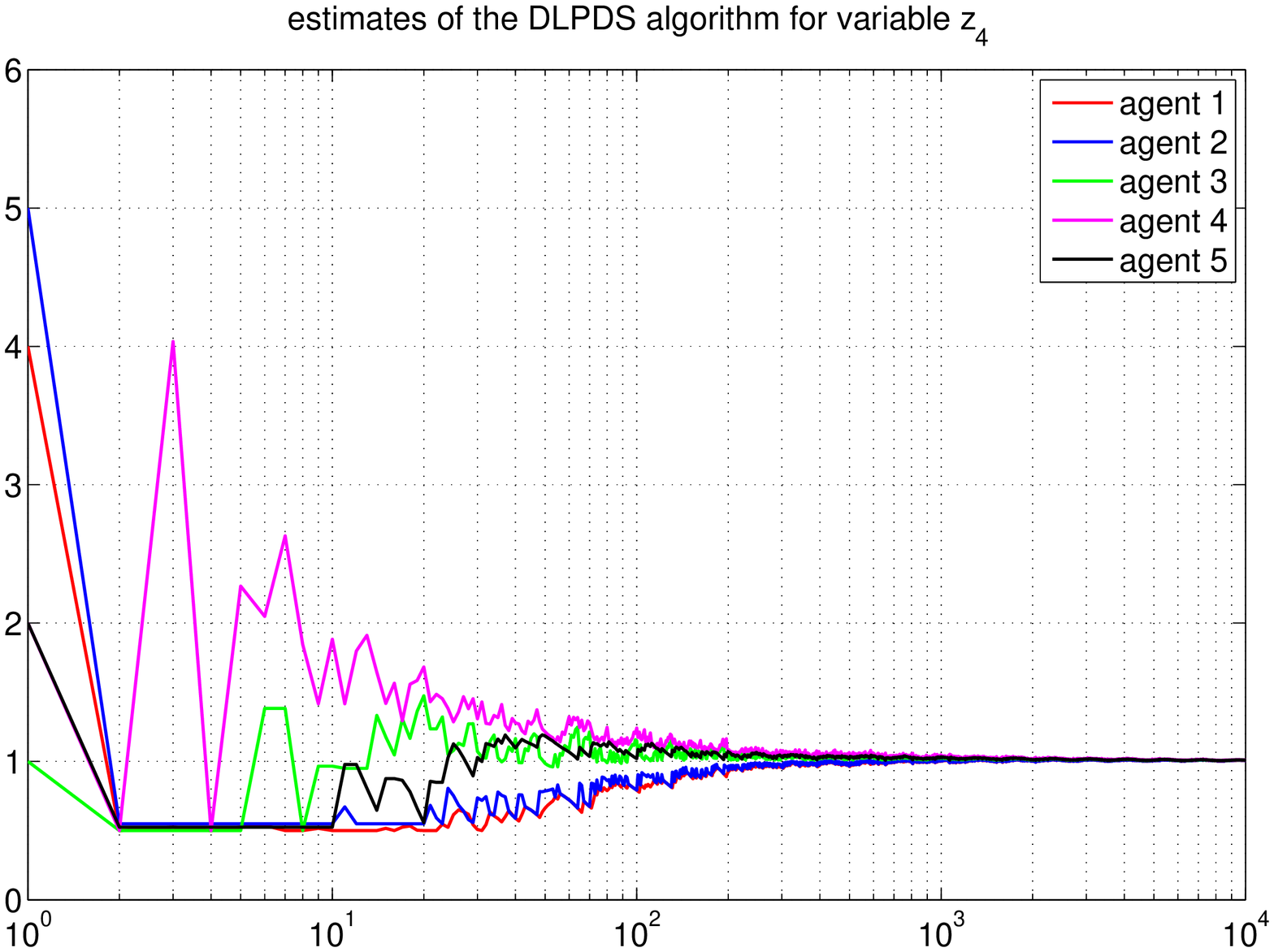}}
  \caption{Estimates of variable $z_4$ of centralized algorithm and the DLPDS algorithm}\label{fig4}
\end{figure}

\begin{figure}[ht]
  \centerline{\epsfxsize=3.5in \epsffile{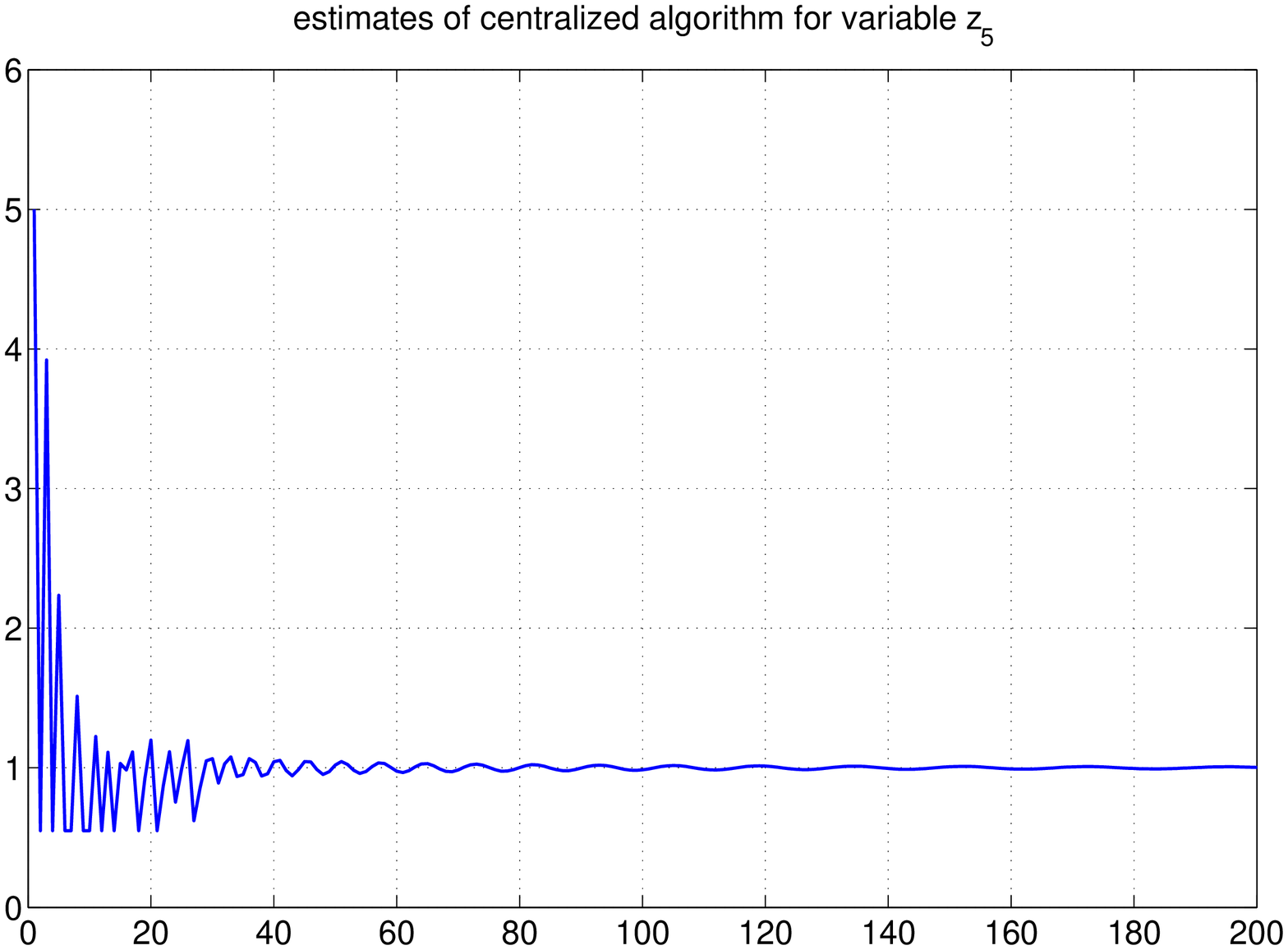} \epsfxsize=3.5in \epsffile{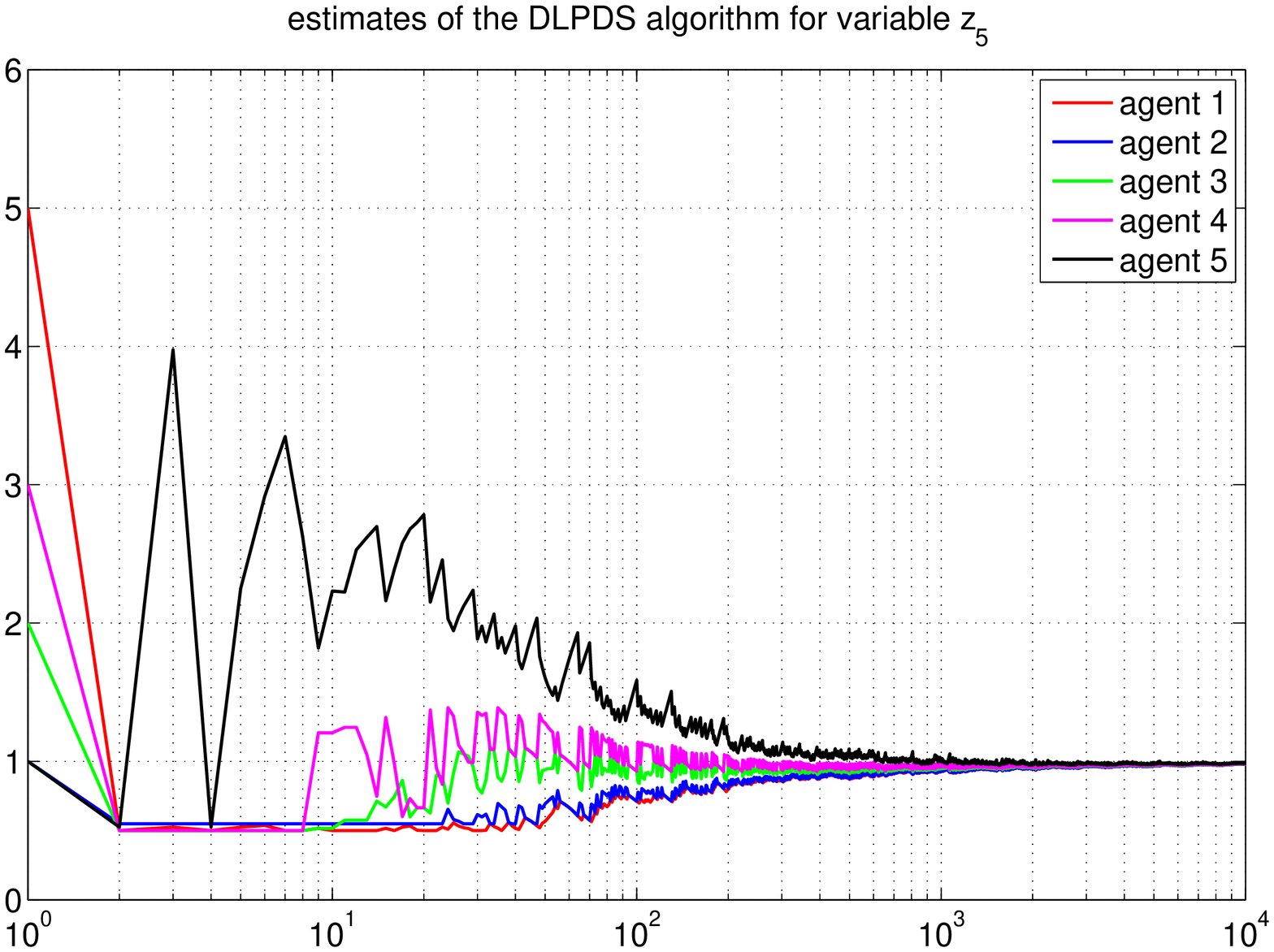}}
  \caption{Estimates of variable $z_5$ of centralized algorithm and the DLPDS algorithm}\label{fig5}
\end{figure}

\begin{figure}[ht]
  \centerline{\epsfxsize=6in \epsffile{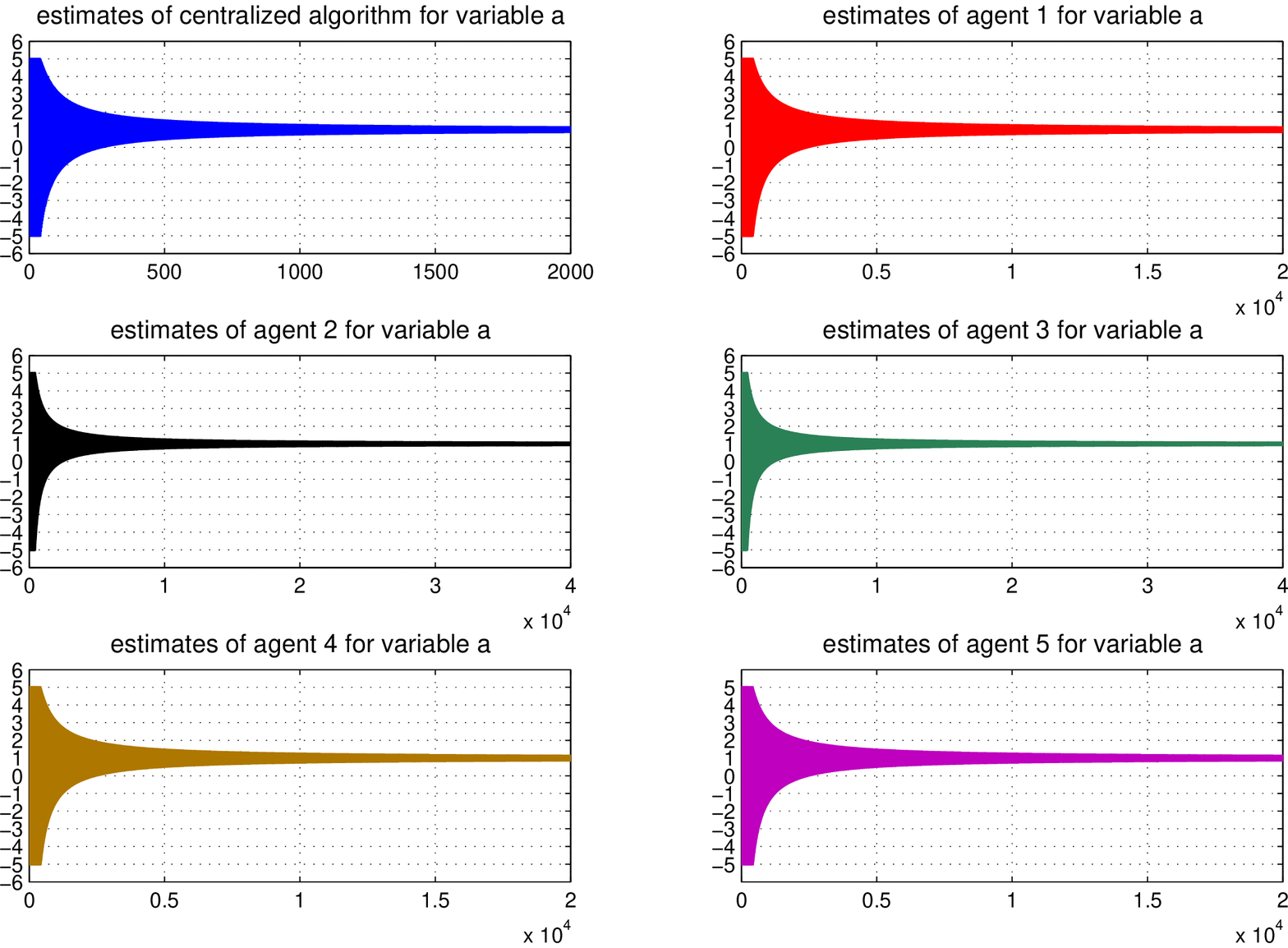}}
  \caption{Estimates of variable a in the DPPDS algorithm}\label{fig6}
  \centerline{\epsfxsize=6in \epsffile{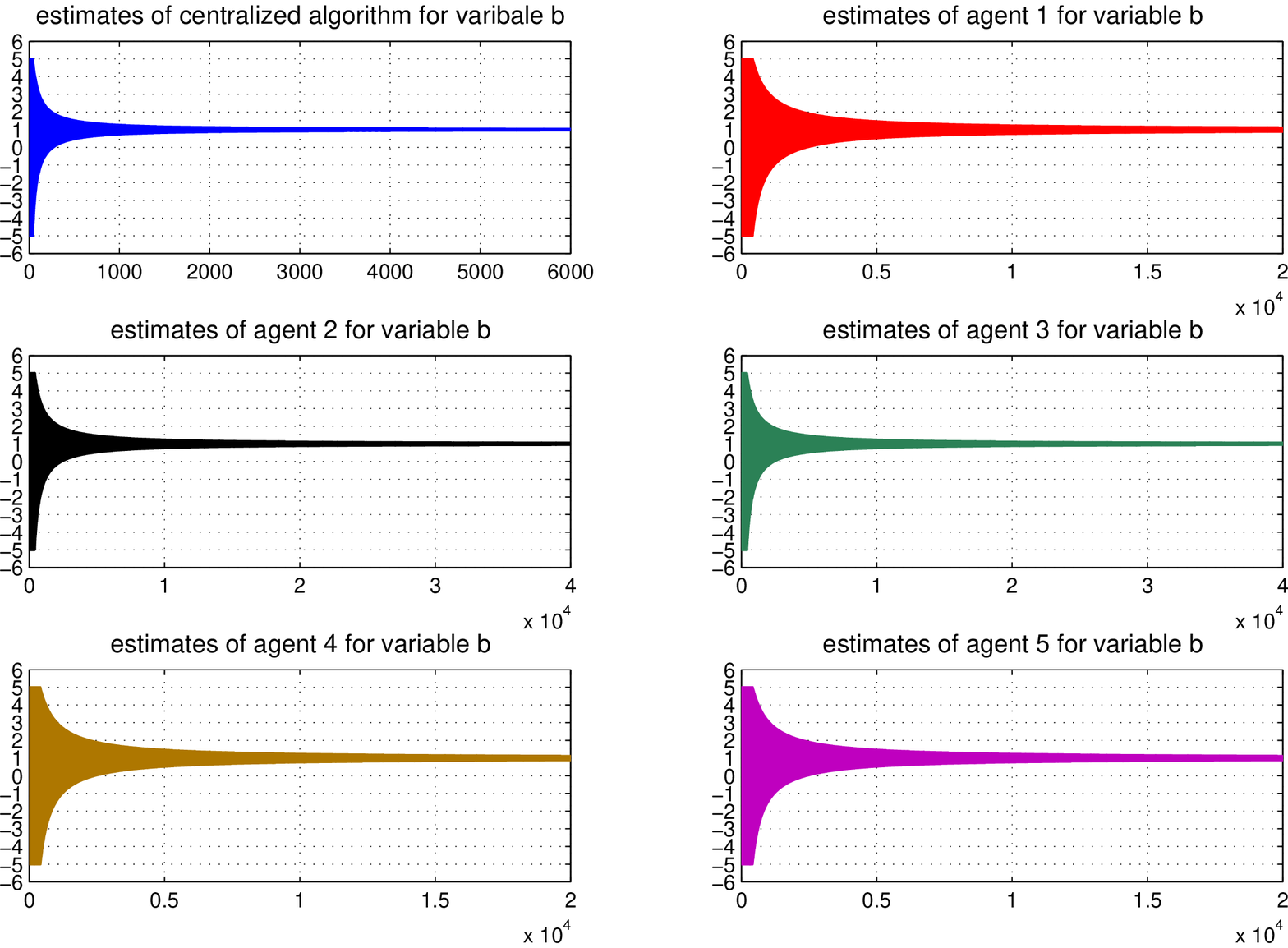}}
  \caption{Estimates of variable b in the DPPDS algorithm}\label{fig7}
\end{figure}

\begin{figure}[ht]
  \centerline{\epsfxsize=6in \epsffile{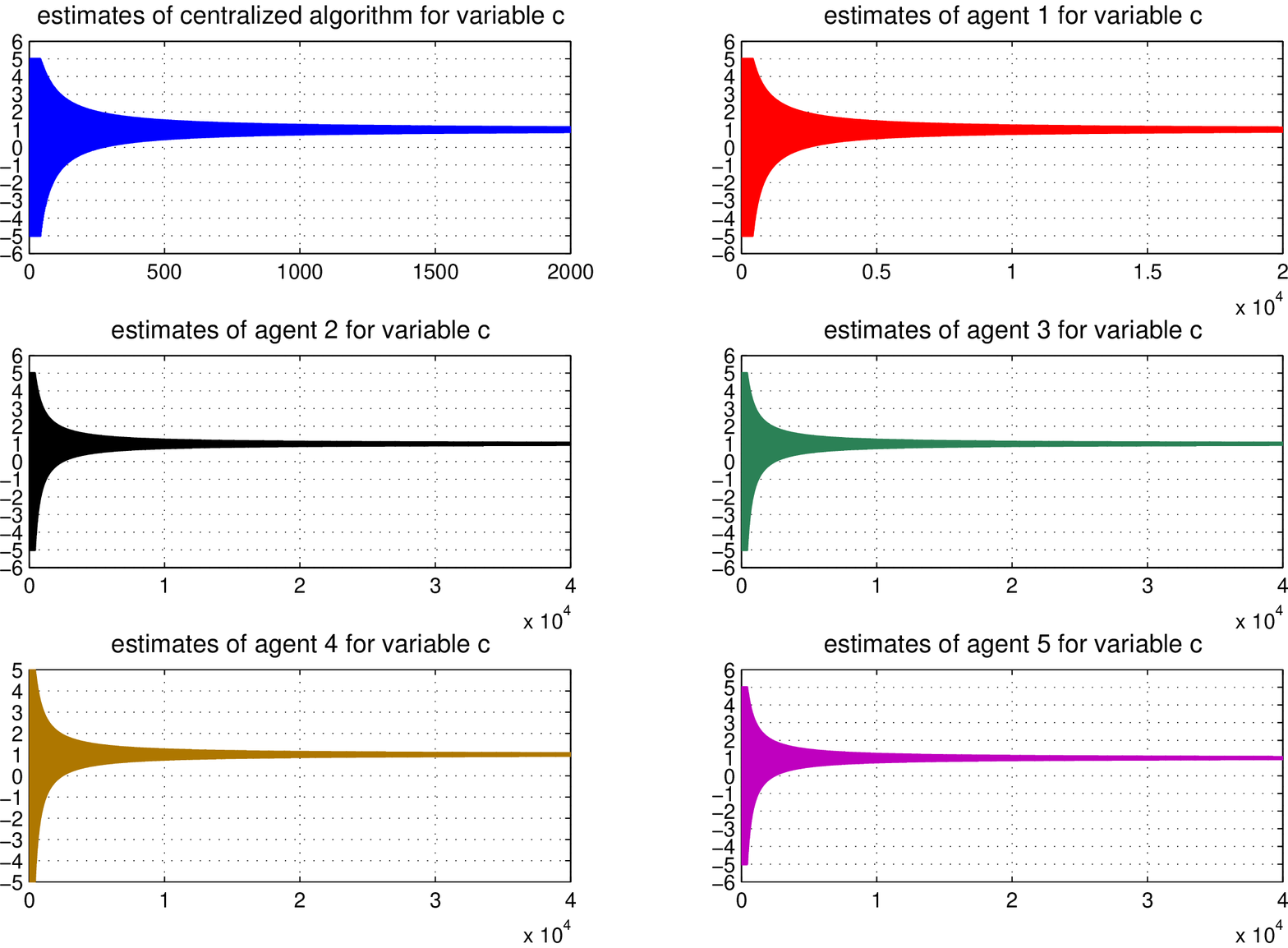}}
  \caption{Estimates of variable c in the DPPDS algorithm}\label{fig8}
  \centerline{\epsfxsize=6in \epsffile{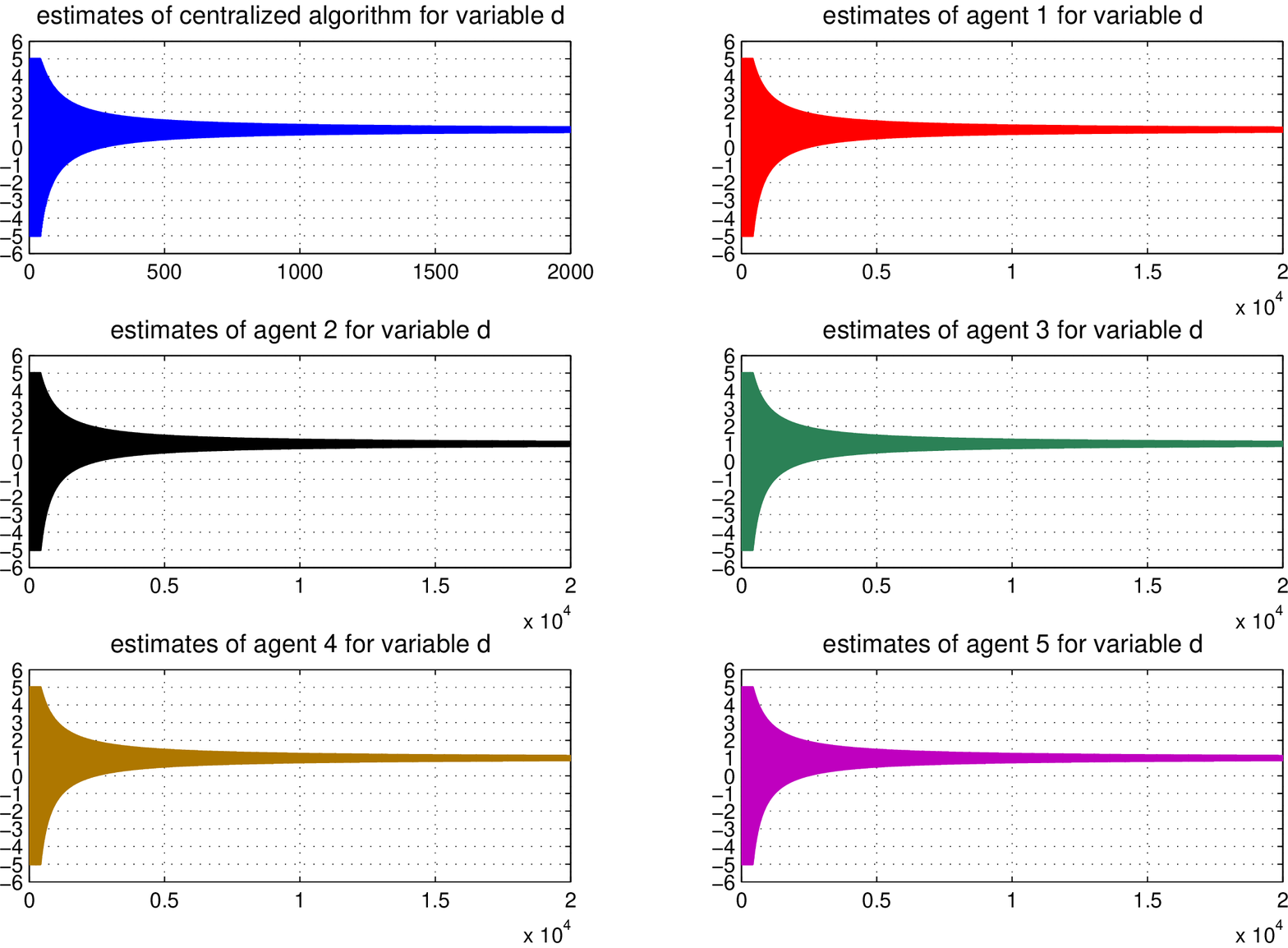}}
  \caption{Estimates of variable d in the DPPDS algorithm}\label{fig9}
\end{figure}

\begin{figure}[ht]
  \centerline{\epsfxsize=6in \epsffile{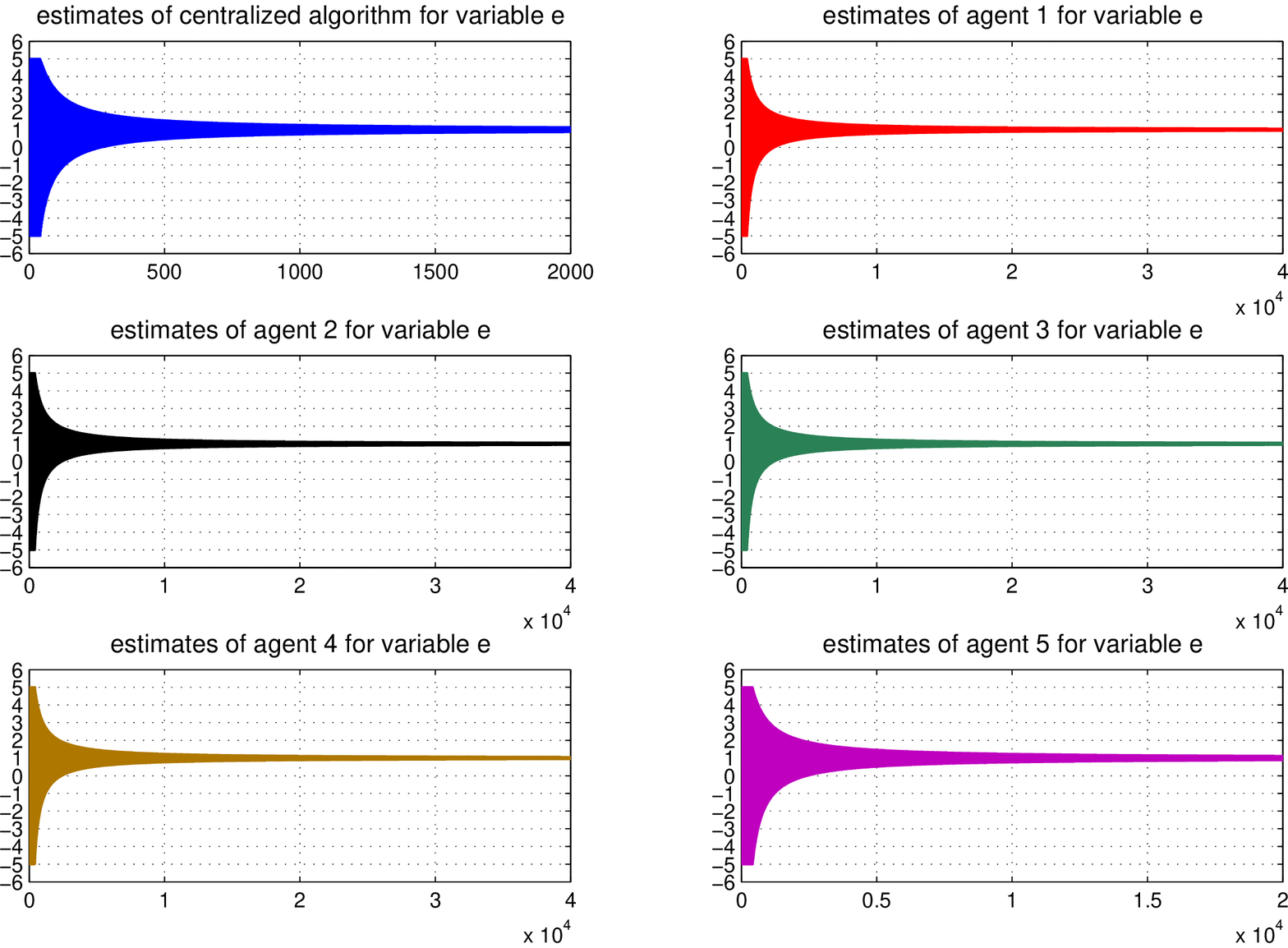}}
  \caption{Estimates of variable e in the DPPDS algorithm}\label{fig10}
\end{figure}

\end{document}